\documentclass[a4paper,11pt]{amsart}

\usepackage[a4paper,left=30mm,right=30mm,top=35mm,bottom=35mm,marginpar=25mm]{geometry} 

\usepackage[utf8]{inputenc}

\usepackage{epsfig}
\usepackage{color}

\usepackage{mathrsfs}
\usepackage{amsmath}    
\usepackage{amsfonts,amssymb}
\usepackage{amsthm}
\usepackage{graphicx}
\usepackage{esint}
\usepackage{stackrel}
\usepackage{enumerate}
\numberwithin{equation}{section}
\newtheorem{theorem}{Theorem}[section]
\newtheorem{lemma}[theorem]{Lemma}
\newtheorem{proposition}[theorem]{Proposition}
\newtheorem{corollary}[theorem]{Corollary}
\theoremstyle{definition}
\newtheorem{definition}[theorem]{Definition}
\newtheorem{remark}[theorem]{Remark}

\newtheorem{notation}[theorem]{Notation}
\newtheorem{Convention}[theorem]{Convention}

\newcommand{\Hf}{\mathcal{H}}
\newcommand{\e}{\mathbf{e}}

\newcommand{\p}{\mathbf{p}}

\newcommand{\C}{\mathbf{C}}
\newcommand{\D}{\mathbf{D}}
\newcommand{\R}{\mathbb{R}}

\newcommand{\A}{\mathcal{A}}
\newcommand{\G}{\mathcal{G}}
\newcommand{\J}{\mathcal{J}}
\newcommand{\F}{\mathcal{F}}
\newcommand{\N}{\mathbb{N}}
\newcommand{\eps}{\varepsilon}

\newcommand{\1}{\mathbf{1}}

\AtBeginDocument{
  \let\div\relax
 \DeclareMathOperator{\div}{div}
\DeclareMathOperator{\Div}{div}

\DeclareMathOperator{\Id}{Id}

}

\title[Minimality of the ball for a model of charged liquid droplets]{Minimality of the ball for a model of charged liquid droplets}

\author[E. Mukoseeva]{Ekaterina Mukoseeva}
\address{E.M: Department of Mathematics and Statistics, P.O. Box 68 (Gustaf H\"allstr\"omin katu 2), FI-00014 University of Helsinki, Finland}
\email{ekaterina.mukoseeva@helsinki.fi}

\author[G. Vescovo]{Giulia Vescovo}
\address{G.V.: SISSA, Via Bonomea 265, 34136 Trieste, Italy}
\email{giulia.vescovo88@gmail.com}

\subjclass[2010]{}

\begin{document}

\begin{abstract} We deduce that charged liquid droplets minimizing \emph{Debye-H\"uckel-type free energy} are spherical in the small charge regime.
The variational model was proposed by Muratov and Novaga in 2016 to avoid the ill-posedness of the classical one.
By combining a recent (partial) regularity result with Selection Principle of Cicalese and Leonardi, we prove that the ball is the unique minimizer in the small charge regime.
%
\end{abstract}

\maketitle

\section{Introduction}
\subsection{Background and description of the model.}
In this paper we deal with a variational model describing the shape of charged liquid droplets.
We investigate the droplets minimizing a suitable free energy composed by an attractive term, coming from surface tension forces, and a repulsive one, due to the electric forces generated by the interaction between charged particles. 
Thanks to the particular structure of the energy, one may expect that for small values of the total charge the attractive part is predominant, forcing in this way the spherical shape.

The experiments agree with this guess - one observes the following phenomenon:
the shape of the liquid droplet is spherical in a small charge regime. Then, as soon as the value of the total charge increases, the droplet gradually deforms into an ellipsoid, it develops conical singularities, the so-called Taylor cones, \cite{Taylor64}, and finally, the liquid starts emitting a thin jet (\cite{DoyleMoffettVonnegut64},\cite{DuftAchtzehnMullerHuberLeisner03},\cite{RichardsonPiggHightower89}, \cite{WilsonTaylor25}).
The first experiments were conducted by Zeleny in 1914, \cite{Zeleny17}, but in a slightly different context.

Several mathematical models of charged liquid droplets have been studied over the years. 
A difficulty is that contrary to the numerical and experimental observations these models
are in general mathematically ill-posed, see \cite{GoldmanNovagaRuffini15}. For a more exhaustive discussion we refer the reader to \cite{MuratovNovaga16}. 

The main issue with the variational model studied in \cite{GoldmanNovagaRuffini15} comes from the tendency of charges to concentrate at the interface of the liquid. To restore the well-posedness one should consider a physical regularizing mechanism in the functional. With this purpose in mind, Muratov and Novaga in \cite{MuratovNovaga16} integrate the entropic effects associated with the presence of free ions in the liquid. The advantage of this model is that the charges are now distributed inside of the droplet.
More precisely, they suggest considering the following \emph{Debye-H\"uckel-type free energy} (in every dimension):
\begin{equation}\label{e:DH}
\mathcal{F}(E,u,\rho):=P(E)+Q^{2}\Bigg\{\int_{\mathbb{R}^{n}}a_{E}|\nabla{u}|^{2}\,dx+K\int_{E}\rho^{2}\,dx\Bigg\}.
\end{equation}
Here $E\subset\R^{n}$ represents the droplet, $P(E)$ is the De Giorgi perimeter, \cite[Chapter 12]{Maggi12}, the constant $Q>0$ is the total charge enclosed in $E$ and
\[
a_{E}(x):=\1_{E^{c}}+\beta\1_{E},
\]
where \(\1_{F}\) is the characteristic function of a set  \(F\) and $\beta>1$
\footnote{ Mathematically,
considering $\beta\leq 1$ amounts to considering the complement of the set $E$ in place of $E$.
Our proof would work without change also for the case $\beta\leq 1$. However, some changes
would be needed in \cite{DPHV}, so we wouldn't be able to use their regularity results directly.}
is the permittivity of the liquid.

The normalized density of charge $\rho\in L^{2}(\mathbb{R}^{n};\mathbb{R}^{n})$ satisfies 
\begin{equation}\label{e:rho}
\rho \1_{E^c}=0 \qquad\text{and}\qquad \int \rho=1,
\end{equation}
and the  electrostatic potential  \(u\) is such that  $\nabla u\in L^{2}(\mathbb{R}^{n})$ and 
\begin{equation} \label{vincolo1}
-\div \big(a_{E}\,\nabla{u}\big)=\rho\qquad\text{in}\;\mathcal{D}'(\mathbb{R}^{n}).
\end{equation}
For a fixed set $E$ we define the set of admissible pairs of functions $u$ and $\rho$: 
\begin{equation}\label{e:admissible}
\mathcal{A}(E):=\big\{\text{\((u,\rho) \in D^{1}(\R^{n})\times  L^2(\R^n)\): \(u\) and \(\rho\) satisfy~\eqref{vincolo1} and~\eqref{e:rho}}\big\},
\end{equation}
where
\[
D^{1}(\R^{n})=\overline{ C_{c}^{\infty}(\R^{n})}^{\mathring{W}^{1,2}(\R^{n})},\qquad  \|\varphi\|_{\mathring{W}^{1,2}(\R^{n})}=\|\nabla \varphi\|_{L^{2}(\R^{n})}.
\]
Note that the class of admissible couples \(\mathcal A(E)\) is non-empty only if \(n\ge 3\) (see \cite[Remark 2.2]{DPHV}). For this reason the assumption \(n\ge 3\) will be in force throughout the whole paper.
The variational problem proposed in \cite{MuratovNovaga16} is the following:
\begin{equation}\label{e:variational}
\min \big\{\F(E,u,\rho):|E|=V, E\subset B_{R}, (u,\rho)\in \A(E)\big\}.
\end{equation}
The a-priori boundedness assumption $E\subseteq B_R$ ensures the existence of a minimizer in the class of sets of finite perimeter with a prescribed volume, \cite[Theorem 3]{MuratovNovaga16}.

For convenience we introduce the following notation:
\begin{equation}\label{e:G}
\G_{\beta,K}(E):=\inf_{(u,\rho)\in\mathcal A(E)}\left\{\int_{\mathbb{R}^{n}}a_{E}|\nabla{u}|^{2}+K\int_{E}\rho^{2}\right\}.
\end{equation}
For $E\subset\R^{n}$ we set
\begin{equation*}
\mathcal F_{\beta,K,Q}(E):= P(E)+Q^{2}\G_{\beta,K}(E).
\end{equation*}
By scaling (see the introduction of \cite{DPHV}), we can reduce the problem to the case $|E|=|B_{1}|$ and so in the rest of the paper we will work with the following problem:
\begin{equation} \label{e:problem} \tag{$\mathcal{P}_{\beta,K,Q,R}$}
 \min \big\{\F_{\beta, K,Q}(E):|E|=|B_{1}|,\  E\subset B_{R}\big\}.
\end{equation}
We will often omit the subscripts $\beta$ and $K$ as those are fixed physical parameters.

We note that the model we investigate can be seen as "interpolation" between Gamow model and the free interface problems
arising in optimal design (see, for example, \cite{FJ}). For the former, it has been recently shown (\cite{KM},\cite{Jul}) that for small enough charges
the unique minimizers are balls. However, in Gamow model the non-local term is Lipschitz with respect to 
symmetric difference between sets, implying that on small scales  the perimeter dominates the non-local part of the energy.
This is not the case for the energy $\F$ defined in \eqref{e:DH} and our analysis is thus more complicated.   
\subsection{Main results.}
As we mentioned above, one can expect that the shape of the droplet in a small charge regime is spherical. We confirm this intuition by proving that the ball is the unique minimizer of the functional $\F$ for small values of the total charge $Q$. Precisely, we obtain the following result.

\begin{theorem}\label{thm:main} 
	Fix $K>0$, $\beta>1$. Then there exists $Q_0>0$ such that
	for all $Q<Q_0$ and any $R\geq 1$ the only minimizers of (\ref{e:problem})
	are the balls of radius $1$.
\end{theorem}

The condition $E\subset B_R$ in the minimizing problem (\ref{e:problem})
is required to have existence of minimizers. However, thanks to Theorem \ref{thm:main} it can be dropped for small enough charges.

\begin{corollary}\label{cor:existence in R^n}
	Fix $K>0$, $\beta>1$. Then there exists $Q_0>0$ such that
	for all $Q<Q_0$ the infimum in the problem
	\begin{equation} \label{e:problem in R^n} \tag{$\mathcal{P}_{\beta,K,Q}$}
		 \inf \big\{\F_{\beta, K,Q}(E):|E|=|B_{1}|\big\}
	\end{equation}
	is attained. Moreover, the only minimizers 
	are the balls of radius $1$.
\end{corollary}

\begin{remark} In the proof we provide the constant $Q_0$ is the same as in Theorem \ref{thm:main}. However,
almost the same proof would give existence of minimizers (but not the fact that they are balls) for $Q<Q_c$
where $Q_c$ is such that the minimizers for \ref{e:problem} are close enough to the ball in $L^1$ norm.
A priori $Q_c$ might be bigger than $Q_0$ and indeed we expect that existence fails later than minimizers
cease to be spherical.
\end{remark}

For the proof of Theorem \ref{thm:main} we combine an improved version of (partial) regularity results for the minimizers of \cite[Theorem 1.2]{DPHV} with second variation techniques. 
The first step is to obtain the partial $C^{2,\vartheta}$-regularity of minimizers. 
In fact, we are able to prove the following partial $C^\infty$-regularity of minimizers, a result that is interesting in itself.

\begin{theorem} [$C^{\infty}$-regularity] \label{smooth reg apx}
Given \(n\ge 3\), \(A>0\) and \(\vartheta\in (0,1/2)\), there exists \(\varepsilon_{\textnormal{reg}}=\varepsilon_{\textnormal{reg}}(n, A, \vartheta)>0\)  such that if \(E\) is a minimizer of\(~\eqref{e:problem}\) with \(Q+\beta+K+\frac{1}{K}\le A\),
\[
x_0\in\partial E \quad \text{and} \quad r+\e_E(x_0, r)+Q^2\,D_E(x_0, r)\le \varepsilon_{\textnormal{reg}},
\] 
then \(E\cap \C(x_0,r/2)\)  coincides with  the  epi-graph of a \(C^{\infty}\)-function $f$. In particular, we have that \(\partial E\cap \C(x_0,r/2)\) is a \(C^{\infty}\) \((n-1)\)-dimensional manifold. Moreover \footnote{Let $\Omega\subset\R^m$ be an open and bounded set, $f\in C(\overline{\Omega})$. Then 
\[ [f]_{ C^{0,\vartheta}(\overline{\Omega}) }:=\sup_{x\neq y, x,y\in \overline{\Omega} } \frac{|f(x)-f(y)|}{|x-y|^\vartheta}.\] 
Moreover, if $f\in C^k(\overline{\Omega})$ then \[ [f]_{C^{k,\vartheta}(\overline{\Omega})}:=\sum_{|\alpha|\leq k} \|D^{\alpha} f\|_{C(\overline{\Omega})}+\sum_{|\alpha|=k} [D^{\alpha} f]_{C^{0,\vartheta}(\overline{\Omega})}.\]},
\begin{equation}
[f]_{C^{k,\vartheta}(\D(x_0',r/2))}\leq C(n,A,k,\vartheta)
\end{equation}
for every $k\in\mathbb{N}$ with $k\geq 2$.
\end{theorem}
We refer the reader to Notation \ref{notations1} for the definition of \(\e_E(x_0, r)\), \(D_E(x_0, r)\) and \(\C(x_0,r/2)\).
\subsection{Strategy of the proof and structure of the paper.}

We use Selection Principle, the technique introduced by Cicalese and Leonardi
in \cite{CL} for the proof of quantitative isoperimetric inequality
(see also \cite{AFM}, where the authors use a similar approach
to investigate a nonlocal isoperimetric problem).

To prove Theorem \ref{thm:main} we first reduce our problem to the so-called \emph{nearly-spherical sets}. Those are the sets which can be described as subgraphs of smooth functions defined over the boundary of the unitary ball. 
The advantage is that for this particular class of sets we are able to deduce a Taylor expansion for the energy near the ball $B_{1}$.

In the first part of the paper (from Section \ref{closed to ball} to Section \ref{nearly spherical}) we show that a minimizer is 
nearly-spherical whenever the total charge is small enough. We argue by contradiction and get a sequence
of minimizers with corresponding total charge going to zero.
In Section \ref{closed to ball} we prove the $L^{1}$-convergence of the minimizers to the unitary ball and the convergence of the perimeters as the charge goes to zero.
Thanks to \emph{uniform} density estimates for the volume and the perimeter of a minimizer we obtain the Kuratowski convergence of sets as well as their boundaries.

Now we need to improve the convergence deduced in Section \ref{closed to ball}. For this purpose it is crucial to enhance the regularity result obtained in \cite{DPHV}. Hence, Section \ref{h reg} is dedicated to the higher regularity of minimizers. By exploiting the Euler-Lagrange equation and the $C^{1,\eta}$-regularity of $u$ up to the boundary $\partial E$, we deduce the partial $C^{2,\vartheta}$-regularity of minimizers. 

In Section \ref{smooth}, by a bootstrap argument, we obtain the partial smooth regularity of minimizers. 

Since for each \(Q\) small enough the corresponding minimal set $E_{Q}$ has $C^{2,\vartheta}$-regular boundary (with uniform bounds), by Ascoli-Arzel\`a, up to extracting a subsequence, we get that $E_{Q}$ converges to $B_{1}$ in a stronger $C^{2,\vartheta'}$-sense
for every $\vartheta'<\vartheta$. This is the content of Section \ref{nearly spherical}.

In Sections \ref{s:thm for nearly spherical} and \ref{fuglede} we prove Theorem \ref{thm:main} for nearly spherical sets.
To this end, we write Taylor expansion of the energy $\G$ using shape derivatives and providing a bound for the "Hessian".
A direct computation provides a similar bound for the perimeter and this allows us to conclude.

In Appendix \ref{sec: second der on the ball} we provide a sharp bound for the second variation of the energy $\G$ at the ball.
We don't need it for the main results but we think it might be of some interest.

 \subsection*{Acknowledgements}

We warmly thank our advisor Guido De Philippis for introducing
us to the problem and for many fruitful discussions.

The work of the authors  is supported by the INDAM-grant ``Geometric Variational Problems".

\section{Notation and preliminary results} \label{notation and preliminaries}
In this section we fix the notation and collect some results obtained in \cite{DPHV} which will be useful in the proof of regularity.
\begin{notation} \label{notations1}
Let $E\subset\R^n$ be a set of finite perimeter, $x\in\R^{n}$, $\nu\in{}\mathbb{S}^{n-1}$ and $r>0$.
\begin{enumerate}
\item[-]  We call $\mathbf{p}^{\nu}(x):=x-(x\cdot{}\nu)\,\nu$ and $\mathbf{q}^{\nu}(x):=(x\cdot{}\nu)\,\nu$, respectively, the \emph{orthogonal projection} onto the plane $\nu^\perp$ and the \emph{projection} on $\nu$. For simplicity we write $\mathbf{p}(x):=\mathbf{p}^{e_{n}}(x)$ and $\mathbf{q}(x):=\mathbf{q}^{e_{n}}(x)=x_{n}$.
\item[-] We define the \emph{cylinder} with center at $x_0\in{}\R^{n}$ and radius $r>0$ with respect to the direction $\nu\in{}\mathbb{S}^{n-1}$ as 
\[
\mathbf{C}(x_0,r,\nu):=\bigl\{x\in\R^{n}\,:\,|\mathbf{p}^{\nu}(x-x_{0})|<r\,\mbox{,}\,|\mathbf{q}^{\nu}(x-x_{0})|<r\bigr\},
\]
and write $\mathbf{C}_{r}:=\mathbf{C}(0,r,e_{n})$, $\mathbf{C}:=\mathbf{C}_{1}$.
\item[-] We denote the (\(n-1\))-dimensional \emph{disk} centered at \(y_0\in\R^{n-1}\) and of radius \(r\) by 
 \[
 \mathbf{D}(y_0,r):=\bigl\{y\in{}\R^{n-1}: |y-y_{0}|<r\bigr\}.
 \]
We let $\mathbf{D}_{r}:=\mathbf{D}(0,r)$ and $\mathbf{D}:=\mathbf{D}(0,1)$.
\item[-] We define
$$\textbf{e}_E(x,r):=\inf_{\nu\in\mathbb{S}^{n-1}} \frac{1}{r^{n-1}}\,\int_{\partial^* E\cap B_r(x)} \frac{|\nu_E(y)-\nu|^2}{2}\,d\mathcal{H}^{n-1}(y).$$
We call $\textbf{e}_E(x,r)$ the \emph{spherical excess}.
Note that from from the definition it follows that
$$\textbf{e}_E(x,\lambda r)\leq\frac{1}{\lambda^{n-1}}\textbf{e}_E(x,r)$$
for any $\lambda\in (0,1)$.
\item[-] Let $(u,\rho)\in\mathcal{A}(E)$ be the minimizer of $\G_{\beta,K}(E)$. We define the \emph{normalized Dirichlet energy} at $x$ as
$$D_E(x,r):=\frac{1}{r^{n-1}}\,\int_{B_r(x)} |\nabla u|^2\,dx.$$
\end{enumerate}
\end{notation}
\begin{Convention}[Universal constants]
Let \(A>0\) be a positive constant. We say that
\itemize{}
\item the parameters \(\beta, K, Q\) with \(\beta\ge 1\) are \emph{controlled by} \(A\) if 
\[
\beta+K+\frac{1}{K}+Q\le A;
\]
\item a constant is \emph{universal} if  it depends only on the dimension \(n\) and on \(A\). 
%
%
\end{Convention}
Note that in particular universal constants \emph{do not depend} on the size of the container where the minimization problem is set.
\\
\indent In the following theorem we collect some properties of minimizers. For the proofs we refer the reader to \cite{DPHV}.
\begin{theorem} \label{properties}
Let $E\subset\mathbb{R}^{n}$ be a set of finite measure. Then 
\begin{enumerate}
\item[{\rm (i)}] there exists a unique pair \((u_E,\rho_E)\in \A(E)\) minimizing \(\G_{\beta,K}(E)\). Moreover,
\[
u_E+K\rho_E=\G_{\beta,K}(E)\qquad \textrm{in \(E\)}, 
\]
and 
\[
0\le u_E\le \G_{\beta,K}(E),\qquad  0\le K\rho_E \le \G_{\beta,K}(E)\1_E.
\]
In particular, \(\rho_E\in L^p\) for all \(p\in [1,\infty]\) with 
\[
\|\rho_E\|_{p}\le C(n,\beta, K, 1/|E|).
\]
\item[{\rm (ii)}] {\rm (Euler-Lagrange equation)} If \(E\) is a minimizer of~\eqref{e:problem}, then 
\[
\begin{split}
\int_{\partial^* E} {\rm div}_E \eta \,d\Hf^{n-1}-Q^2\,\int_{\R^n}a_E \Big(|\nabla u_E|^2{\rm div} \eta &-2\nabla u_E\cdot (\nabla \eta\, \nabla u_E))\,dx\\
&-Q^2\,K\int_{\R^n} \rho_E^2 {\rm div} \eta\,dx=0
\end{split}
\]
for all \(\eta\in C^{1}_c(B_R;\R^n)\) with \(\int_E {\rm div} \eta \,dx=0\). 
\item[{\rm (iii)}] {\rm (Compactness)}
Let  \(K_h, Q_h\in \R\), \(\beta_h\ge 1\) and \(R_h\ge 1\) be such that 
\[
K_h\to K>0\,,\quad \beta_h\to \beta\ge 1\,,\quad R_h\to R\ge 1\,,\quad Q_h\to Q\ge 0,
\]
when $h\to\infty$. For every  \(h\in \mathbb N\) let \(E_h\) be  a minimizer of $\left(\mathcal{P}_{\beta_h,K_h,Q_h,R_h}\right)$. 
\\
Then, up to a non relabelled subsequence,  there exists a set of finite perimeter \(E\) such that
\[
\lim_{h\to\infty}|E\Delta E_h|= 0.
\]
Moreover, \(E\) is a minimizer of~\eqref{e:problem} and 
\[
 \F_{\beta,K, Q}(E)=\lim_{h\to \infty} \F_{\beta_h,K_h, Q_h}(E_h), \qquad \lim_{h\to\infty}P(E_h)=P(E).
\]
\end{enumerate}
Let \(A>0\). For the following properties we require that  \(\beta, K\) and \(Q\) are controlled by \(A\).
\begin{enumerate}
\item[{\rm (iv)}] {\rm (Boundedness of the normalized Dirichlet energy)} There exists a universal  constant \(C_{\mathrm{e}}>0\) such that,  if  \(E\) is a minimizer of  \eqref{e:problem}, then for all \(x\in \overline{B_R}\),
\[
Q^2D_{E}(x,r)=\frac{Q^2}{r^{n-1}} \int_{B_r(x)}|\nabla u|^2\,dx \le C_{\mathrm{e}}.
\]
\item[{\rm (v)}] {\rm (Density estimates)} There exist universal  constants \(C_{\mathrm{o}}\), \(C_{{\rm i}}>0\) 
and \(\bar{r}>0\) such that, if  \(E\) is a minimizer of  \eqref{e:problem}, then\footnote{Here and in the sequel we will always work with the representative of \(E\) such that
 \[
 \partial E=\Biggl\{x:   \frac{|B_r(x)\setminus E|}{|B_r(x)|} \cdot \frac{|B_r(x)\cap E|}{|B_r(x)|}>0\quad \text{for all \(r>0\)}\Biggr\},
 \] 
 see \cite[Proposition 12.19]{Maggi12}.}
\[
\frac{1}{C_{{\rm i}}}r^{n-1}\le P(E,B_{r}(x))\le C_{\mathrm{o}}r^{n-1} \qquad \text{for all   \(x \in \partial E\) and \(r\in (0,\bar{r})\),}
\]
and
\[
 \frac{1}{C_{{\rm i}}} \le \frac{|B_r(x)\cap E|}{|B_r(x)|}\le C_{\mathrm{o}}   \qquad \text{for all  \(x \in E\) and \(r\in (0,\bar{r})\).}
\]
\item[{\rm (vi)}] {\rm (Excess improvement)} There exists a universal constant \(C_{\mathrm {dec}}>0$ such that for all   \(\lambda\in (0,1/4)\) there exists $\varepsilon_{\mathrm{dec}}=\varepsilon_{\mathrm{dec}}(n,A,\lambda)>0$ satisfying the following:   if \(E\) is a minimizer of \eqref{e:problem} and 
\begin{equation*}
x\in\partial E,\quad  r+Q^2D_{E}(x,r)+\e_E(x,r)\leq{}\varepsilon_{\mathrm{dec}},
\end{equation*}
then
\[
Q^2D_{E}(x,\lambda r)+\e_E(x,\lambda r)\leq{} C_{\mathrm{dec}}\lambda \Bigl(\e_E(x,r) +Q^2D_{E}(x,r)+r\Bigr).
\]
\item[{\rm (vii)}] {\rm (Decay of the Dirichlet energy)}\label{p:dir}
There exists a universal constant \(C_{\mathrm{dir}}>0\) such that for all \(\lambda \in (0,1/2)\)  there exists  \(\eps_{\mathrm{dir}}=\eps_{\mathrm{dir}}(n,A, \lambda)\) satisfying the following:  if \(E\) is a minimizer of \eqref{e:problem}, \(x\in \partial E\) and 
\[
r+\e_E(x,r,e_n)\le \eps_{\mathrm{dir}},
\]
then 
\[
D_E(x,\lambda r)\le C_{\mathrm{dir}}\lambda \Bigl(D_E(x, r)+r\Bigr).
\]
\end{enumerate}
\end{theorem}
\begin{proof} The proofs of (i), (iii), (iv), (v), (vi) and (vii) can be found respectively in \cite[Proposition 2.1, Proposition 5.1, Lemma 6.5, Proposition 6.4, Proposition 6.6, Theorem 7.1, Proposition 7.6]{DPHV}.
There is no detailed proof of (ii) in \cite{DPHV}. Moreover, the formula given in \cite[Corollary 3.3]{DPHV} has a sign mistake and thus we give a proof of (ii) here.

We start by showing the following identity for any $\rho\in L^2(\mathbb{R}^n)$:
\begin{equation}\label{e:duality dirichlet energy}
\begin{split}
	&\inf\left\{\int_{\mathbb{R}^n}{a_E\vert\nabla u\vert^2}: u\in D^1(\mathbb{R}^n), -\Div(a_E\nabla u)=\rho\right\}\\
	&\qquad\qquad=\inf\left\{\int_{\mathbb{R}^n}{\frac{\vert V\vert^2}{a_E}}: V\in L^2(\mathbb{R}^n;\mathbb{R}^n), -\Div(V)=\rho\right\}.
\end{split}	
\end{equation}
Right-hand side is trivially not larger than the left-hand side, as we can take $V=a_E\nabla u$ as a competitor. So we only need to show that
\begin{equation*}
\begin{split}
	&\inf\left\{\int_{\mathbb{R}^n}{a_E\vert\nabla u\vert^2}: u\in D^1(\mathbb{R}^n), -\Div(a_E\nabla u)=\rho\right\}\\
	&\qquad\qquad\leq\inf\left\{\int_{\mathbb{R}^n}{\frac{\vert V\vert^2}{a_E}}: V\in L^2(\mathbb{R}^n;\mathbb{R}^n), -\Div(V)=\rho\right\}.
\end{split}	
\end{equation*}
We use that the infimum is achieved in both cases by convexity. Hence, the right-hand side has a minimizer $V_0$ and it satisfies the corresponding Euler-Langrange equation, that is 
\begin{equation*}
\frac{V_0}{a_E}\cdot X=0\qquad\text{ for any }X\in C^\infty_c(\mathbb{R}^n;\mathbb{R}^n)\text{ such that }\Div X=0.
\end{equation*} 
But that
gives us that $\frac{V_0}{a_E}=\nabla u_0$ for some $u_0\in D^1(\mathbb{R}^n)$. Since $\Div(V_0)=\rho$, we get
$-\Div(a_E\nabla u_0)=\rho$ and thus
\begin{equation*}
\begin{split}
	&\inf\left\{\int_{\mathbb{R}^n}{a_E\vert\nabla u\vert^2}: u\in D^1(\mathbb{R}^n), -\Div(a_E\nabla u)=\rho\right\}\leq \int_{\mathbb{R}^n}{a_E\vert\nabla u_0\vert^2}\\
	&\qquad\leq \int_{\mathbb{R}^n}{\frac{\vert V_0\vert^2}{a_E}}=\inf\left\{\int_{\mathbb{R}^n}{\frac{\vert V\vert^2}{a_E}}: V\in L^2(\mathbb{R}^n;\mathbb{R}^n), -\Div(V)=\rho\right\},
\end{split}	
\end{equation*}
which finishes the proof of \eqref{e:duality dirichlet energy}. 

Suppose now that $E$ is a minimizer of~\eqref{e:problem} and $(u,\rho)\in\mathcal{A}$ is the pair minimizing $\G(E)$. We fix a vector field $\eta\in C^\infty_c(B_R;\mathbb{R}^n)$ with \(\int_E {\rm div} \eta \,dx=0\) and, following \cite[Lemma 3.1]{DPHV}, we define 
\begin{equation*}
\varphi_t(x):=x+t\eta,
\qquad u_t:=u\circ\varphi_t^{-1},
\qquad \tilde{\rho}_t:=\mathrm{det}(\nabla\varphi_t^{-1})\rho\circ\varphi_t^{-1}.
\end{equation*}
By \cite[Lemma 3.1]{DPHV} we have 
\begin{equation}\label{e:new electrostatic potential}
-\Div(a_{E_t}A_t\nabla u_t)=\rho_t,
\end{equation}
where $E_t=\varphi_t(E)$, $A_t=\mathrm{det}(\nabla\varphi_t^{-1})(\nabla\varphi_t^{-1})^{-t}(\nabla\varphi_t^{-1})^{-1}$.
Note that $\vert E_t\vert=\vert E\vert+o(t)=\vert B_1\vert+o(t)$ since  \(\int_E {\rm div} \eta \,dx=0\).

Now we recall that $E$ is a minimizer and we use \eqref{e:duality dirichlet energy} to get that
\begin{equation*}
\begin{split}
&P(E)+Q^{2}\Bigg\{\int_{\mathbb{R}^{n}}a_{E}|\nabla{u}|^{2}\,dx+K\int_{E}\rho^{2}\,dx\Bigg\}\\
&=\min\left\{P(E)+Q^2\inf_{(u,\rho)\in\mathcal A(E)}\left\{\int_{\mathbb{R}^{n}}a_{E}|\nabla{u}|^{2}+K\int_{E}\rho^{2}\right\}:\vert E\vert=\vert B_1\vert, E\subset B_R\right\}\\
&=\min\left\{P(E)+Q^2\inf_{\rho\in L^2(\mathbb{R}^n)}\left\{\inf_{-\Div(a_E\nabla u)=\rho}\int_{\mathbb{R}^{n}}a_{E}|\nabla{u}|^{2}+K\int_{E}\rho^{2}\right\}:\vert E\vert=\vert B_1\vert, E\subset B_R\right\}\\
&=\min\left\{P(E)+Q^2\inf_{\rho\in L^2(\mathbb{R}^n)}\left\{\inf_{-\Div(V)=\rho}\int_{\mathbb{R}^{n}}\frac{|V|^{2}}{a_E}+K\int_{E}\rho^{2}\right\}:\vert E\vert=\vert B_1\vert, E\subset B_R\right\}\\
&\leq P(E_t)+Q^2\left\{\int_{\mathbb{R}^{n}}a_{E_t}|A_t\nabla u_t|^{2}+K\int_{E_t}\rho_t^{2}\right\},
\end{split}
\end{equation*}
where for the last inequality we used \eqref{e:new electrostatic potential} and the fact that $\vert E_t\vert=\vert B_1\vert+o(t)$ by the choice of $\eta$. To get Euler-Lagrange equation it remains to expand the last quantity in terms of $t$.
It is well known (see, for example, \cite[Theorem 17.8]{Maggi12}) that
\begin{equation}\label{e:expansion of the perimeter}
	P(E_t)=P(E)+t\int_{\partial^* E}{\Div_E\eta}+o(t),
\end{equation}
where $\Div_E\eta=\Div\eta-\nu_E\cdot (\nabla\eta\;\nu_E)$.
As for the other part of the energy, using change of variables, we have
\begin{equation*}
\begin{split}
&\int_{\mathbb{R}^{n}}a_{E_t}|A_t\nabla u_t|^{2}+K\int_{E_t}\rho_t^{2}
=\int_{\mathbb{R}^{n}}a_{E}|\mathrm{det}(\nabla\varphi_t^{-1})(\nabla\varphi_t)^{t}(\nabla\varphi_t)(\nabla\varphi_t)^{-t}\nabla u|^{2}\mathrm{det}(\nabla\varphi_t)\\
&\qquad\qquad+K\int_{E}\rho^{2}\mathrm{det}^2(\nabla\varphi_t^{-1})\mathrm{det}(\nabla\varphi_t)\\
&=\int_{\mathbb{R}^{n}}a_{E}|(1-t\Div\eta)(\mathrm{Id}+t(\nabla\eta)^t)(\mathrm{Id}+t\nabla\eta)(\mathrm{Id}-t(\nabla\eta)^t)\nabla u|^{2}(1+t\Div\eta)\\
&\qquad\qquad+K\int_{E}\rho^{2}(1-t\Div\eta)+o(t)\\
&=\int_{\mathbb{R}^{n}}a_{E}|(1-t\Div\eta)(\mathrm{Id}+t\nabla\eta)\nabla u|^{2}(1+t\Div\eta)
+K\int_{E}\rho^{2}(1-t\Div\eta)+o(t)\\
&=\int_{\mathbb{R}^{n}}a_{E}|(1+t\left(-\Div\eta\;\mathrm{Id}+\nabla\eta\right))\nabla u|^{2}(1+t\Div\eta)
+K\int_{E}\rho^{2}(1-t\Div\eta)+o(t)\\
&=\int_{\mathbb{R}^{n}}a_{E}\left(|\nabla u|^{2}+t\left(-2\Div\eta\vert\nabla u\vert^2+2\nabla u\cdot(\nabla\eta\nabla u)\right)\right)(1+t\Div\eta)\\
&\qquad\qquad+K\int_{E}\rho^{2}(1-t\Div\eta)+o(t)\\
&=\int_{\mathbb{R}^{n}}a_{E}|\nabla u|^{2}+K\int_{E}\rho^{2}+t\left(\int_{\mathbb{R}^{n}}a_{E}\left(-\Div\eta\vert\nabla u\vert^2+2\nabla u\cdot(\nabla\eta\nabla u)\right)-K\int_{E}\rho^{2}\Div\eta\right)\\
&\qquad\qquad\qquad+o(t),
\end{split}
\end{equation*}
where for the first equality we used that $\nabla\varphi_t=\left(\nabla\varphi_t^{-1}\right)^{-1}\circ\varphi_t$.
Bringing it all together we get that for any $\eta\in C^\infty_c(B_R;\mathbb{R}^n)$ with \(\int_E {\rm div} \eta \,dx=0\)
for any $t$ we have
\begin{equation*}
\begin{split}
&P(E)+Q^{2}\Bigg\{\int_{\mathbb{R}^{n}}a_{E}|\nabla{u}|^{2}\,dx+K\int_{E}\rho^{2}\,dx\Bigg\}
\leq P(E)+Q^2\left\{\int_{\mathbb{R}^{n}}a_{E}|\nabla u|^{2}+K\int_{E}\rho^{2}\right\}\\
&+t\left(\int_{\partial^* E}{\Div_E\eta}
+Q^2 \left(\int_{\mathbb{R}^{n}}a_{E}\left(-\Div\eta\vert\nabla u\vert^2+2\nabla u\cdot(\nabla\eta\nabla u)\right)-K\int_{E}\rho^{2}\Div\eta\right)\right)+o(t),
\end{split}
\end{equation*}
which gives us (ii).
\end{proof}

We state now the $\eps$-regularity theorem.
\begin{theorem}[{\cite[Theorem 1.2]{DPHV}}] \label{eps reg}
Given \(n\ge 3\), \(A>0\) and \(\vartheta\in (0,1/2)\), there exists \(\varepsilon_{\textnormal{reg}}=\varepsilon_{\textnormal{reg}}(n, A, \vartheta)>0\)  such that if \(E\) is minimizer of\(~\eqref{e:problem}\) with \(Q+\beta+K+\frac{1}{K}\le A\), \(x\in \partial E\)  and 
\[
r+\e_E(x, r)+Q^2\,D_E(x, r)\le \varepsilon_{\textnormal{reg}},
\] 
then \(E\cap \C(x,r/2)\)  coincides with  the  epi-graph of a \(C^{1,\vartheta}\) function. In particular,  \(\partial E~\cap~\C(x,r/2)\) is a \(C^{1,\vartheta}\) \((n-1)\)-dimensional manifold.
\end{theorem}
\section{Closeness to the ball} \label{closed to ball}
In this section we deduce the $L^{\infty}$-closeness of minimizers to the unitary ball in the small charge regime. Let us start with the following proposition.
\begin{proposition} [$L^{1}$-closeness to the ball] \label{L1 closed}
Let $\{Q_h\}_{h\in\mathbb{N}}$ be a sequence in $\mathbb{R}$ such that $Q_h>0$ and $Q_h\to 0$ when $h\to\infty$. Let $\{E_h\}_{h\in\mathbb{N}}$ be a sequence of minimizers of $\left(\mathcal{P}_{\beta,K,Q_h,R}\right)$. Then, up to translations, $E_h\to B_{1}$ in $L^{1}$  and $P(E_{h})\to P(B_{1})$ when $h\to\infty$.
\end{proposition}
\begin{proof} By the quantitative isoperimetric inequality, \cite[Theorem 1.1]{FuscoMaggiPratelli08}, for every $h~\in~\mathbb{N}$ there exists a point $x_h\in\mathbb{R}^n$ such that
\[
|E_{h}\Delta B_{1}(x_h)|^{2}\leq C\left(P(E_{h})-P(B_{1})\right)
\]
for some constant $C=C(n)>0$ which depends only on $n$. By translating each set $E_h$ we can assume without loss of generality that the following inequality holds:
\begin{equation} \label{qii}
|E_{h}\Delta B_{1}|^{2}\leq C\left(P(E_{h})-P(B_{1})\right).
\end{equation}
By the minimality of $E_{h}$ we have
\[ 
\begin{split}
\F_{\beta,K,Q_{h},R}\left(E_{h}\right) &=P(E_{h})+Q^{2}_{h}\,\G_{\beta,K}(E_{h})\\
&\leq P(B_{1})+Q^{2}_{h}\,\G_{\beta,K}(B_{1})=\F_{\beta,K,Q_{h},R}\left(B_{1}\right),\quad \forall\,h\in\mathbb{N}.
\end{split}
\]
Hence, \eqref{qii} yields
\[
|E_{h}\Delta B_{1}|^{2}\leq C\left(P(E_{h})-P(B_{1})\right)\leq C\,Q^{2}_{h}\,\G_{\beta,K}(B_{1})\quad \forall\,h\in\mathbb{N},
\]
for some constant $C=C(n)>0$ which depends only on the dimension $n$. 

Then $Q_h\to 0$ implies $E_h\to B_{1}$ in $L^{1}$ and $P(E_{h})\to P(B_{1})$ when $h\to\infty$.
\end{proof}
Thanks to the density estimates (see Theorem \ref{properties} (v)), we can improve the convergence of Proposition \ref{L1 closed}.
\begin{proposition} [$L^{\infty}$-closeness to the ball] \label{Linfinity closed}
Let $\{Q_h\}_{h\in\mathbb{N}}$ be a sequence such that $Q_h>0$ and $Q_h\to 0$ when $h\to\infty$. Let $\{E_h\}_{h\in\mathbb{N}}$ be a sequence of minimizers of $\left(\mathcal{P}_{\beta,K,Q_h,R}\right)$. Then, up to translations, $E_h\to \overline{B}_{1}$ and $\partial E_h\to \partial B_{1}$ in the Kuratowski sense. 
\end{proposition}
\begin{proof} 
By Proposition \ref{L1 closed} we know that up to translations 
$E_h\rightarrow B_1$ in $L^1$.
First, we prove the Kuratowski convergence of $E_h$ to the ball $\overline{B}_{1}$, i.e.
\begin{enumerate}
\item[\rm{(i)}] $x_{h}\to x, \, x_{h}\in E_h \, \Rightarrow \, x\in \overline{B}_{1}$,
\item[\rm{(ii)}] $x\in \overline{B}_{1} \Rightarrow \exists x_h \in E_h \text{ such that }x_h \to x $.
\end{enumerate}
In order to prove (i) let $x_{h}\to x$ and $x_{h}\in E_{h}$. Assume by contradiction that $x\notin \overline{B}_{1}$. Then there exits $B_{s}(x)\subset \R^{n}$ such that $B_{s}(x)\cap B_{1}=\emptyset$.
By Theorem \ref{properties} (v), 
for $Q_{h}$ small enough there exist a radius $\bar{r}>0$ and a constant $C>0$, both independent of $Q_{h}$, such that
\begin{equation} \label{density volume}
\left|B_{r}(x_h)\cap E_h\right|\geq C\,r^{n} \quad \forall r\leq\bar{r}.
\end{equation}
%
Since $x_h\rightarrow x$, for any $r>0$ we can define $h(r)\in\mathbb{N}$ such that $B_{\frac{r}{2}}(x_{h})\subset B_{r}(x)$ for every $h\geq h(r)$. Then, for any $r\leq\overline{r}$, $h\geq h(r)$,

\begin{equation} \label{density volume ball}
\left|B_{r}(x)\cap E_h\right|\geq |B_{\frac{r}{2}}(x_h)\cap E_h|\geq C\,r^{n}.
\end{equation}
By the $L^{1}$-convergence of $E_h$ to $B_{1}$ and \eqref{density volume ball} we deduce $|B_r(x)\cap B_1|>0$ for any $r\leq\overline{r}$, a contradiction with $B_s(x)\cap B_1=\emptyset$.

The proof of (ii) follows by arguing similarly as above, exploiting the $L^{1}$-convergence.
Analogously, by using density estimates for the perimeter of $E_{h}$ and the convergence of perimeters $P(E_{h})\to P(B_{1})$, one can prove that $\partial E_h\to \partial B_{1}$ in the Kuratowski sense. 
\end{proof}
\section{Higher regularity} \label{h reg}
In this section we improve Theorem \ref{eps reg}. To be more precise, we deduce the partial $C^{2,\vartheta}$ regularity of minimizers. 
The first step is to obtain better regularity for a couple $(u,\rho)\in\A(E)$, where $E\subset\R^{n}$ is a minimizer of the problem \eqref{e:problem}: we prove that $u$ is $C^{1,\eta}$-regular up to the boundary of $E$.
We start with some preliminary results. 
\begin{notation} \label{TE}  Let $E\subset\R^n$ be such that $\partial E\cap \C(x_0,r)$ is described by the graph of a regular function $f$.
\itemize
\item If $x\in\R^n$, we write $x=(x',x_n)$, where $x'\in\R^{n-1}$ and $x_n\in\R$.
\item We denote by $\nu_E$ the outer-unit normal to $\partial E$. Moreover, we extend $\nu_E$ at every point in the following way
\[
\nu_E(x',x_n)=\nu_{E}(x',f(x')) \quad \forall x=(x',x_n)\in\C(x_0,r).
\]
\item Let $u$ be a solution of 
$$-{\rm div }(a_{E} \nabla u)=\rho_{E} \quad \text{in } \mathcal{D}'\left(B_{r}(x_{0})\right),$$ 
where 
\[
\rho_E\in L^{\infty}\left(B_r(x_0)\right) \quad \text{and} \quad a_{E}=\beta \textbf{1}_{E}+\textbf{1}_{E^{c}}.
\]
We denote by 
\[
T_E u \index{\(T_E u\)}:=\partial_{\nu_E^{\perp}}u+ (1+(\beta-1) \textbf{1}_{E}) \partial_{\nu_E}u,
\]
where
\[
\partial_{\nu_E^{\perp}} u:= \nabla u-(\nabla u \cdot \nu_E)\,\nu_E \quad \text{and } \quad \partial_{\nu_E} u:= (\nabla u \cdot \nu_E)\,\nu_E.
\]
%
\item We denote by 
$$[g]_{x,r}:=\frac{1}{\vert B_r\vert}\int_{B_{r}(x)}g\,dx$$
the \emph{mean value} of $g\in L^{1}(B_r(x))$. We simply write $[g]_{r}:=[g]_{0,r}$.
\item We denote the restrictions of a function $v$ to $E$ and $E^c$
by $v^+$ and $v^-$ respectively:
$$v^+:=v\,\1_E,\quad v^-:=v\,\1_{E^c}.$$

\end{notation}
Let us recall the following integral characterization of H\"older continuous functions.
\begin{lemma} [Campanato's lemma, see {\cite[Theorem 7.51]{AFP}} ] \label{campanato}
Let $p\geq1$ and $g~\in~L^{p}(B_{2R}(x_0))$. Assume that there exist $\sigma\in(0,1)$ and $A>0$ such that for every $x\in B_R(x_0)$
\begin{equation} \label{campanatoeq}
\frac{1}{\vert B_r\vert}\int_{B_r(x)}|g(y)-[g]_{x,r}|^p\,dy\leq A^p\,\left(\frac{r}{R}\right)^{p\sigma}, \quad \forall B_r(x)\subset B_R(x_0).
\end{equation}
Then there exists a constant $C=C(n,p,\sigma)$ such that $g$ is $\sigma$-H\"older continuous in $B_R(x_0)$ with a constant $C\frac{A}{R^\sigma}$ and
\begin{equation*}
\max_{x\in B_R(x_0)}{\left\vert g(x)\right\vert}\leq CA+\left\vert [g]_{x_0,R}\right\vert.
\end{equation*}
\end{lemma}

We also recall a simple iteration lemma.
\begin{lemma} [{\cite[Lemma 7.54]{AFP}}] \label{iterating lemma}
Let $0<q<p$, $s>0$. Suppose that $h:(0,a)~\to~[0,+\infty)$ is an increasing function such that
\begin{equation*}
h(r)\leq c_1 \,\left(\left(\frac{r}{R}\right)^p+R^s\right) h(R)+c_2\,R^q \quad \text{for every }\, 0<r<R\leq a,
\end{equation*}
where $c_1$ and $c_2$ are positive constants.
Then there exists  $R_0=R_0(p,q,s,c_1)>0$, $c~=~c(p,q,s)>0$ such that
\begin{equation*}
h(r)\leq c \,\left\{\left(\frac{r}{R}\right)^q\,h(R)+c_2\,r^q\right\} \quad \text{for every }\, 0<r<R\leq\min(R_0,a).
\end{equation*}
\end{lemma}

We are going to use the following lemma.
\begin{lemma}[{\cite[Theorem 7.53]{AFP}}] \label{piano}
Let $v$ be a solution of 
$$-{\rm div }(a_{H} \nabla v)=\rho_{H} \quad \text{in } \mathcal{D}'\left(B_{1}(x_{0})\right),$$ 
where $\rho_H\in L^{\infty}\left(B_1(x_0)\right)$ and
\[
H:=\{y\in\R^{n}\,:\,(y-x_{0})\cdot e_{n}\leq 0\}, \quad a_{H}=\beta \1_{H}+\1_{H^{c}}.
\]
Then there exist $\gamma\in(0,1)$ and a constant $C_{0}=C_{0}(n,\beta,\|\rho_{H}\|_{\infty})>0$ such that
\begin{equation*}
\int_{B_{\lambda r}(x_{0})}|T_H  v-[T_H v]_{x_{0},\lambda r}|^{2}\,dx\leq C_{0} \lambda^{n+2\gamma} \int_{B_{r}(x_{0})}|T_H v-[T_H v]_{x_{0},r}|^{2}\,dx+C_{0}\, r^{n+1},
\end{equation*}
for all $\lambda \in (0,1)$ small enough. Note that
$T_H v:=\big(\partial_{1} v,\dots,\partial_{n-1}v, (1+(\beta-1) \1_{H}) \partial_{n} v\big)$.
\end{lemma}

We argue similarly to the proof of Theorem 7.53 in \cite{AFP} to show the following lemma.
\begin{lemma} \label{tiposchauder}
Let $H\subset \R^n$ be the half space. Let $v\in W^{1,2}(B_1)$ be a solution of
\begin{equation}\label{e:eq for v}
-\Div(A \nabla v)=\Div G \quad \text{in } \mathcal{D}'\left(B_{1}\right),
\end{equation}
where 
\[
G^{+}:=G\,\1_{H}\in C^{0,\alpha}(H), \quad G^{-}:=G\,\1_{H^c}\in C^{0,\alpha}(H^c),
\]
$A$ is an elliptic matrix and $A^{+}=A\,\1_{H}$, $A^{-}=A\,\1_{H^c}$ have coefficients respectively in $C^{0,\alpha}(B_r \cap \overline{H})$ and $C^{0,\alpha}(B_1 \cap \overline{H^c})$. 
Then 
\[
v^+:=v\,\1_{H}\in  C^{1,\alpha}(B_{1/2} \cap \overline{H}), \quad v^-:=v\,\1_{H^c} \in C^{1,\alpha}(B_{1/2} \cap \overline{H^c}). 
\]
Moreover, there exists a constant \(C=C\left(\Vert G^{+}\Vert_{C^{0,\alpha}},\Vert G^-\Vert_{C^{0,\alpha}},\Vert A^+\Vert_{C^{0,\alpha}},\Vert A^-\Vert_{C^{0,\alpha}}  \right)>0\) such that
\begin{equation} \label{darichiamareancora}
[\nabla v^{+}]_{C^{0,\alpha}(\overline{H}\cap B_{1/2})} \leq C \quad \text{ and } \quad  [\nabla v^{-}]_{C^{0,\alpha}(\overline{H^c}\cap B_{1/2})} \leq C.
\end{equation}
\end{lemma}

\begin{proof} 

Fix $x_0\in B_{1/2}$, and let $r$ be such that $B_r(x_0)\subset B_1$.
We denote by $a^+$ and $a^-$ the averages of $A$ in $B_r(x_0)\cap H$
and $B_r(x_0)\cap H^c$ respectively. In an analogous way we define
$g^+$ and $g^-$ as the averages of $G$ in $B_r(x_0)\cap H$
and $B_r(x_0)\cap H^c$. For $x\in B_r(x_0)$ we set
\begin{equation*}
	\overline{A}:=
	\begin{cases}
		a^+ \text{ if }x_n>0\\
		a^- \text{ if }x_n<0
	\end{cases}\quad
	\text{ and }\quad	
	\overline{G}:=
	\begin{cases}
		g^+ \text{ if }x_n>0\\
		g^- \text{ if }x_n<0
	\end{cases}.
\end{equation*} 
By the assumptions of the lemma,
\begin{equation}\label{e:ineq on A bar and G bar}
	\vert A(x)-\overline{A}(x)\vert\leq cr^\alpha\quad\text{ and }\quad
	\vert G(x)-\overline{G}(x)\vert\leq cr^\alpha.
\end{equation}
Let $w$ be the solution of 
\begin{equation*}
	\begin{cases}
		-\div(\overline{A}\nabla w)=\div\overline{G} \text{ in }B_r,\\
		w=v \text{ on }\partial B_r(x_0).
	\end{cases}
\end{equation*}
Note that the last equation can be rewritten as
\begin{equation}\label{e:eq for w}
	\begin{cases}
		-\div(a^+\nabla w^+)=0 \text{ in }H\cap B_r(x_0),\\
		-\div(a^-\nabla w^-)=0 \text{ in }H^c\cap B_r(x_0),\\
		w^+=w^-\text{ on }\partial H\cap B_r(x_0),\\
		a^+\nabla w^+\cdot e_n-a^-\nabla w^-\cdot e_n=g^+\cdot e_n-g^-\cdot e_n\text{ on }\partial H\cap B_r(x_0),\\
		w=v \text{ on }\partial B_r(x_0),
	\end{cases}
\end{equation}
where $w^+:=w\,\1_{H\cap B_r(x_0)}$, $w^-:=w\,\1_{H^c\cap B_r(x_0)}$.
For a function $u$ set
\begin{equation}\label{def:conormal}
	\overline{D}_c{u}(x)=\sum_{i=1}^n{\overline{A}_{i,n}\nabla_i u(x)}
	+\overline{G}\cdot e_n;
\end{equation}
\begin{equation}\label{def:proper conormal}
	D_c{u}(x)=\sum_{i=1}^n{A_{i,n}\nabla_i u(x)}
	+G\cdot e_n.
\end{equation}
The reason for such a definition is that $D_c v$ and $\overline{D}_c w$ have no jumps on the boundary thanks to the transmission condition
in \eqref{e:eq for w}.
We are going to estimate the decay of $D_\tau w$ and $\overline{D}_c w$,
which will lead to H\"older continuity of $D_\tau v$ and $D_c v$,
yielding the desired estimate on $\nabla v$.

{\bf Step 1:} tangential derivatives of $w$. Since both $\overline{A}$
and $\overline{G}$ are constant along the tangential directions,
the classical difference quotient method (see, for example, \cite[Section 4.3]{GM}) gives us that $D_\tau w\in W^{1,2}_{loc}(B_r(x_0))$ and $\div(\overline{A}\nabla(D_\tau w))=0$ in $B_r(x_0)$. Hence, Caccioppoli's inequality holds:
\begin{equation}\label{e:Caccioppoli for D_tau}
\int_{B_\rho(x)}{\left\vert\nabla(D_\tau w)\right\vert^2}dy
\leq C\rho^{-2}\int_{B_{2\rho}(x)}{\left\vert D_\tau w-(D_\tau w)_{x,2\rho}\right\vert^2}dy
\end{equation}
for all balls $B_{2\rho}(x)\subset B_r(x_0)$ and by De Giorgi's regularity theorem (see, for example, \cite[Theorem 7.50]{AFP}),
$D_\tau w$ is H\"older-continuous and, thus, if $B_{\rho'}(x)\subset B_r(x_0)$,
\begin{equation}\label{e:decay of D_tau}
\int_{B_\rho(x)}{\left\vert D_\tau w-(D_\tau w)_{x,\rho}\right\vert^2}dy\leq c\left(\frac{\rho}{\rho'}\right)^{n+2\gamma}\int_{B_{\rho'}(x)}{\left\vert D_\tau w-(D_\tau w)_{x,\rho'}\right\vert^2}dy
\end{equation}
for any $\rho\in(0,\rho'/2)$ and
\begin{equation}\label{e:max of D_tau}
\max_{B_{\rho'/2}(x)} \left\vert D_\tau w\right\vert^2\leq \frac{C}{(\rho')^{n}}\int_{B_{\rho'}(x)}{\left\vert D_\tau w\right\vert^2}dy.
\end{equation}

{\bf Step 2:} regularity of $\overline{D}_c w$. 
First let us show that the distributional gradient of $\overline{D}_c w$ 
is given by the gradient of $\overline{D}_c w$ on the upper half ball plus the one on the lower,
i.e. that there is no contribution on the hyperplane.
For that, we need to check that
$$-\int_{B_r(x_0)}\overline{D}_c w\div\varphi\,dx= 
\int_{B_r(x_0)^+}\nabla \overline{D}_c w\cdot\varphi\,dx+\int_{B_r(x_0)^-}\nabla \overline{D}_c w\cdot\varphi\,dx$$
for any $\varphi\in C^\infty_c(B_r(x_0);\mathbb{R}^n)$.
Indeed, if we perform integration by parts on the left hand side, we get
\begin{equation*}
\begin{split}
&-\int_{B_r(x_0)}\overline{D}_c w\div\varphi\,dx= 
\int_{B_r(x_0)^+}\nabla \overline{D}_c w\cdot\varphi\,dx+\int_{B_r(x_0)^-}\nabla \overline{D}_c w\cdot\varphi\,dx\\
&+
\int_{\partial H\cap B_r(x_0)}{\left(\sum_{i=1}^n{a^+_{i,n}\nabla_i w(x)}
	+g^+\cdot e_n
	-\sum_{i=1}^n{a^-_{i,n}\nabla_i w(x)}
	-g^-\cdot e_n\right)(\varphi\cdot e_n)}\,d\mathcal{H}^{n-1}
\end{split}
\end{equation*}
for any $\varphi\in C^\infty_c(B_r(x_0);\mathbb{R}^n)$
and the last term vanishes thanks to the transmission condition
in \eqref{e:eq for w}. Thus, the distributional gradient of $\overline{D}_c w$ coincides with the point-wise one.

Since $D_\tau(\overline{D}_c w)=\overline{D}_c(D_\tau w)-\overline{G}\cdot e_n$, the tangential derivatives of $\overline{D}_c w$ are in $L^2_{loc}$.
As for the normal derivative, by the definition \eqref{def:conormal}
$$\left\vert\frac{\partial \overline{D}_c{w}}{\partial\nu}(x)\right\vert\leq C\vert \nabla D_\tau w\vert+2\Vert\overline{G}\Vert_{L^\infty}.
$$
It implies
$$\left\vert\nabla \overline{D}_c{w}(x)\right\vert\leq C\left(\vert \nabla D_\tau w\vert+\Vert\overline{G}\Vert_{L^\infty}\right).
$$
and thus $\overline{D}_c w$ is in $W^{1,2}_{loc}$.
Now, using Poincar\'e's inequality and \eqref{e:Caccioppoli for D_tau}, we have
\begin{equation*}
\begin{split}
&\int_{B_\rho(x)}{\left\vert \overline{D}_c w-(\overline{D}_c w)_{x,\rho}\right\vert^2}dy
\leq C\rho^2\int_{B_\rho(x)}{\left\vert\nabla(\overline{D}_c w)\right\vert^2}dy\\
&\leq C\rho^2\int_{B_\rho(x)}{\left\vert\nabla(D_\tau w)\right\vert^2}dy+C\rho^{n+2}
\leq C\int_{B_{2\rho}(x)}{\left\vert D_\tau w-(D_\tau w)_{x,2\rho}\right\vert^2}dy+C\rho^{n+2}
\end{split}
\end{equation*}
for any $B_{2\rho}(x)\subset B_r(x_0)$.
Remembering \eqref{e:decay of D_tau}, we obtain
\begin{equation}\label{e:decay of D_c w}
\begin{split}
\int_{B_\rho(x)}{\left\vert \overline{D}_c w-(\overline{D}_c w)_{x,\rho}\right\vert^2}dy
&\leq C\left(\frac{\rho}{r}\right)^{n+2\gamma}\int_{B_{r/2}(x)}{\left\vert D_\tau w-(D_\tau w)_{x,r/2}\right\vert^2}dy+C\rho^{n+2}\\
&\leq C\left(\frac{\rho}{r}\right)^{n+2\gamma}\int_{B_r(x_0)}{\left\vert D_\tau w\right\vert^2}dy+C\rho^{n+2}
\end{split}
\end{equation}
for any $x\in B_{r/4}(x_0)$, $\rho\leq r/4$. Hence, by Lemma \ref{campanato}, $\overline{D}_c w$ is H\"older-continuous
and
\begin{equation}\label{e:max of D_c}
\max_{B_{r/4}(x_0)} \left\vert \overline{D}_c w\right\vert^2\leq \frac{C}{r^{n}}\int_{B_r(x_0)}{\left\vert \overline{D}_c w\right\vert^2}dy+C.
\end{equation}

{\bf Step 3:} compairing $v$ and $w$.
Subtracting the equation for $w$ from the equation for $v$ we get
\begin{equation}\label{e:equation for v-w}
\begin{split}
\int_{B_r(x_0)}{\overline{A}_{i,j}(y)\left(\frac{\partial v}{\partial y_i}-\frac{\partial w}{\partial y_i}\right)\frac{\partial \varphi}{\partial y_j}}dy=\int_{B_r(x_0)}{\left(\overline{A}_{i,j}(y)-A_{i,j}(y)\right)\frac{\partial v}{\partial y_i}\frac{\partial \varphi}{\partial y_j}}dy\\
+\int_{B_r(x_0)}{\left(\overline{G}_{i}-G_i\right)\frac{\partial \varphi}{\partial y_i}}dy
\end{split}
\end{equation}
for any $\varphi\in W^{1,2}_0(B_r(x_0))$.
We test \eqref{e:equation for v-w} with $\varphi=v-w$ to get
\begin{equation}\label{e:bound on the norm of v-w}
\int_{B_r(x_0)}{\left\vert\nabla v-\nabla w\right\vert^2}dy\leq Cr^{2\alpha}\int_{B_r(x_0)}{\left\vert\nabla v\right\vert^2}dy+Cr^{n+2\alpha},
\end{equation}
which in turn gives us
\begin{equation*}
\begin{split}
\int_{B_\rho(x_0)}{\left\vert\nabla v\right\vert^2}dy
\leq 2\int_{B_\rho(x_0)}{\left\vert\nabla w\right\vert^2}dy
+2\int_{B_\rho(x_0)}{\left\vert\nabla v-\nabla w\right\vert^2}dy\\
\leq 2\omega_n\rho^n\sup_{B_{r/4}(x_0)}{\left\vert\nabla w\right\vert^2}+Cr^{2\alpha}\int_{B_r(x_0)}{\left\vert\nabla v\right\vert^2}dy+Cr^{n+2\alpha}
\end{split}
\end{equation*}
for $\rho\leq r/4$. Recalling \eqref{e:max of D_tau} and \eqref{e:max of D_c}, we obtain
\begin{equation*}
\begin{split}
\int_{B_\rho(x_0)}{\left\vert\nabla v\right\vert^2}dy
\leq C\left(\frac{\rho}{r}\right)^n\int_{B_r(x_0)}{\left\vert\nabla w\right\vert^2}dy+C\rho^n+Cr^{2\alpha}\int_{B_r(x_0)}{\left\vert\nabla v\right\vert^2}dy+Cr^{n+2\alpha}\\
\leq C\left(\frac{\rho}{r}\right)^n\int_{B_r(x_0)}{\left\vert\nabla v\right\vert^2}dy+Cr^{2\alpha}\int_{B_r(x_0)}{\left\vert\nabla v\right\vert^2}dy+Cr^{n}.
\end{split}
\end{equation*}
Now we can apply Lemma \ref{iterating lemma} and get that there exists $r_0>0$ such that for $\rho<r/4<r_0$
\begin{equation*}
\int_{B_\rho(x_0)}{\left\vert\nabla v\right\vert^2}dy
\leq C\left(\frac{\rho}{r}\right)^{n-\alpha}\int_{B_r(x_0)}{\left\vert\nabla v\right\vert^2}dy+C\rho^{n-\alpha}.
\end{equation*}
In particular, for $\rho<r_0$ we have
\begin{equation}\label{e:decay of means of nabla v}
\int_{B_\rho(x_0)}{\left\vert\nabla v\right\vert^2}dy
\leq C\rho^{n-\alpha},
\end{equation}
where $C=C\left(\Vert G^{+}\Vert_{C^{0,\alpha}},\Vert G^-\Vert_{C^{0,\alpha}},\Vert A^+\Vert_{C^{0,\alpha}},\Vert A^-\Vert_{C^{0,\alpha}}\right)$.
Note that the $L^2$ norm of $\nabla v$ in $B_1$ is bounded by some constant depending
only on $L^\infty$ norms of $A$ and $G$, as can be seen by testing the equation
\eqref{e:eq for v} with $v$.

{\bf Step 4:} H\"older-continuity of $\nabla v$. We show local H\"older continuity of $D_c v$ and $D_\tau v$, H\"older-continuity of $\nabla v$ in $B_{1/2}\cap \overline{H}$ and in $B_{1/2}\cap \overline{H^c}$ follows immediately.

Take $\rho<r_0$, where $r_0$ is from the previous step. Let $d$ be any real number. Using the definitions \eqref{def:conormal} and \eqref{def:proper conormal}, we get
\begin{equation}\label{e:ineq for the norm of D_c v-d}
\begin{split}
	&\int_{B_\rho(x_0)}{\vert D_c v-d\vert^2}dy
	=\int_{B_\rho(x_0)}{\left\vert \overline{D}_c v-d+\sum_{i=1}^n(A_{i,n}-\overline{A}_{i,n})\nabla_i v+(G-\overline{G})e_n\right\vert^2}dy\\
	&\leq 2\int_{B_\rho(x_0)}{\vert \overline{D}_c v-d \vert^2}dy+4\int_{B_\rho(x_0)}{\left\vert \sum_{i=1}^n(A_{i,n}-
\overline{A}_{i,n})\nabla_i v\right\vert^2}dy+4\int_{B_\rho(x_0)}{\left\vert(G-\overline{G})e_n\right\vert^2}dy\\		
	&\leq 2\int_{B_\rho(x_0)}{\vert \overline{D}_c w-d+\sum_{i=1}^n\overline{A}_{i,n}(\nabla_i v-\nabla_i w) \vert^2}dy+4\int_{B_\rho(x_0)}{Cr^{2\alpha}\vert \nabla v\vert^2}dy+4\int_{B_\rho(x_0)}{Cr^{2\alpha}}dy\\		
	&\leq 4\int_{B_\rho(x_0)}{\vert \overline{D}_c w-d\vert^2}dy
	+C r^{n+\alpha},
\end{split}	
\end{equation}
where we used inequalities \eqref{e:ineq on A bar and G bar} for the second to last inequality, and inequalities \eqref{e:decay of means of nabla v} and \eqref{e:bound on the norm of v-w} for the last inequality.
Thus, using \eqref{e:decay of D_c w} we have for $\rho<r/4$, $r<r_0$
\begin{equation}\label{e:decay of D_c v}
\begin{split}
	&\int_{B_\rho(x_0)}{\vert D_c v-(D_c v)_{x_0,\rho}\vert^2}dy\leq
	\int_{B_\rho(x_0)}{\vert D_c v-(\overline{D}_c w)_{x_0,\rho}\vert^2}dy\\ 
	&\leq 4\int_{B_\rho(x_0)}{\vert \overline{D}_c w-(\overline{D}_c w)_{x_0,\rho}\vert^2}dy
	+C r^{n+\alpha}
	\leq C\left(\frac{\rho}{r}\right)^{n+2\gamma}\int_{B_r(x_0)}{\left\vert D_\tau w\right\vert^2}dy+C r^{n+\alpha},
\end{split}	
\end{equation}
where we used the fact that $\int_{\Omega}{(f(x)-t)}\,dx$ is minimized by $t^*=\fint_\Omega f$ for the first inequality and the inequality \eqref{e:ineq for the norm of D_c v-d} with $d=(\overline{D}_c w)_{x_0,\rho}$ for the second inequality.
Similarly, using \eqref{e:decay of D_tau} instead of \eqref{e:decay of D_c w},
we get 
\begin{equation}\label{e:decay of D_tau v}
\begin{split}
	&\int_{B_\rho(x_0)}{\vert D_\tau v-(D_\tau v)_{x_0,\rho}\vert^2}dy
	\leq C\left(\frac{\rho}{r}\right)^{n+2\gamma}\int_{B_r(x_0)}{\left\vert D_\tau w\right\vert^2}dy+C r^{n+\alpha}.
\end{split}	
\end{equation}
Applying Lemma \ref{iterating lemma} to \eqref{e:decay of D_c v} and \eqref{e:decay of D_tau v}, we deduce that $D_c v$ and $D_\tau v$ are H\"older by Lemma \ref{campanato}.

\end{proof}
\begin{lemma} \label{Dirichlet bounded}
Given a minimizer $E$ of \eqref{e:problem}, let $(u,\rho)\in\mathcal{A}(E)$ be the minimizing pair of $\G_{\beta,K}(E)$. Assume that $\partial E \cap \C(x_0,r)$ is a $C^{1,\vartheta}$-manifold.
Then for every $\gamma\in(0,1)$ there exist $0<\bar{r}\leq r$ and $C>0$ such that the following inequality holds true
\begin{equation*}
Q^2\,\int_{B_{\tilde{r}}(x_0)} |\nabla u|^{2}\,dx\leq C\,\tilde{r}^{n-\gamma} 
\end{equation*}
for every $\tilde{r}\leq \bar{r}$.
\end{lemma}
\begin{proof} Fix $\gamma\in(0,1)$.
Choose $\lambda\in(0,1/4)$ such that  
\[
\left(1+C_{\text{dec}}\right)\lambda\leq\lambda^{1-\gamma},
\]
where $C_{\text{dec}}$ is as in Theorem \ref{properties} (vi). Let $s=s(\lambda)<\frac{1}{2}$ be such that 
\begin{equation} \label{s(l)}
C_{\text{dir}}(C_{\text{e}}+1)\, s(\lambda) \leq \frac{\varepsilon_{\text{dec}}(\lambda)}{2},
\end{equation}
where $\varepsilon_{\text{dec}}$ , $C_{\text{dir}}$  and  $C_{\text{e}}$ are as in Theorem \ref{eps reg} and Theorem \ref{properties} (vii), (iv). 
Define
\[
\varepsilon(\lambda):=\min \left\{ s^{n-1} \frac{\varepsilon_{\text{dec}}(\lambda)}{2}, \varepsilon_{\text{dir}}(\lambda) \right\}.
\]
Since $\partial E \cap \C(x_0,r)$ is regular, we can take a radius $0<\bar{r}<\mathrm{min}\left(r,1,\frac{1}{Q^2}\right)$ such that 
\[
\bar{r}+\e_E(x_0,\bar{r}) \leq \varepsilon(\lambda).
\]
Then, thanks to the definition of $\varepsilon(\lambda)$, Theorem \ref{properties} (vii), (iv), and \eqref{s(l)} we have
\begin{equation} \label{condr1}
Q^2\,D_E(x_0,s\bar{r}) \leq C_{\text{dir}}s\, (Q^2\,D_E(x_0,\bar{r})+Q^2\bar{r}) \leq C_{\text{dir}} (C_{\text{e}}+1) s \leq \frac{\varepsilon_{\text{dec}}(\lambda)}{2}.
\end{equation}
Furthermore, notice that
\begin{equation} \label{condr2}
s\bar{r}+\e_E(x_0,s\bar{r}) \leq \bar{r}+ \frac{1}{s^{n-1}} \e_E(x_0,\bar{r})\leq \frac{\varepsilon_{\text{dec}}(\lambda)}{2}.\\
\end{equation}
Combining \eqref{condr1} and \eqref{condr2}, we have
\begin{equation*} 
s\bar{r}+Q^2\,D_{E}(x_0,s\bar{r})+\textbf{e}_E(x_0,s\bar{r})\leq\varepsilon_{\text{dec}}(\lambda).
\end{equation*}
The hypothesis of Theorem \ref{properties} (vi) is satisfied, hence (recall that $\lambda s\bar{r}\leq \varepsilon_{\text{dec}}(\lambda)$)
\begin{equation*}
\begin{split}
&Q^2\,D_{E}(x_0,\lambda s\bar{r})+\textbf{e}_E(x_0,\lambda s\bar{r})+\lambda s\bar{r} \leq \lambda^{1-\gamma}\,\left(\textbf{e}_E(x_0,s\bar{r})+Q^2\,D_{E}(x_0,s\bar{r})+s\bar{r}\right)\\
&\leq \lambda^{1-\gamma}\varepsilon_{\text{dec}}(\lambda)\leq\varepsilon_{\text{dec}}(\lambda).
\end{split}
\end{equation*}
Exploiting again Theorem \ref{p:dir}, we obtain
\begin{equation*}
\begin{split}
Q^2\,D_{E}(x_0,\lambda^2 s\bar{r})+\textbf{e}_E(x_0,\lambda^2  s\bar{r})+\lambda^2  s\bar{r}&\leq \lambda^{(1-\gamma)}\,\left(\textbf{e}_E(x_0,\lambda  s\bar{r})+Q^2\,D_{E}(x_0,\lambda  s\bar{r})+\lambda s\bar{r}\right)\\
&\leq \lambda^{2(1-\gamma)}\,\left(\textbf{e}_E(x_0, s\bar{r})+Q^2\,D_{E}(x_0, s\bar{r})+ s\bar{r}\right)\\
&\leq \lambda^{2(1-\gamma)}\varepsilon_{\text{dec}}(\lambda)\leq\varepsilon_{\text{dec}}(\lambda).
\end{split}
\end{equation*}
Iterating this argument $k$ times, we conclude that
\begin{equation*}
\begin{split}
Q^2\,D_{E}(x_0,\lambda^k s\bar{r})+\textbf{e}_E(x_0,\lambda^k  s\bar{r})+\lambda^k  s\bar{r} \leq \lambda^{k(1-\gamma)}\varepsilon_{\text{dec}}(\lambda), \quad \forall k\in\mathbb{N}.
\end{split}
\end{equation*}
In particular, the inequality above yields
\begin{equation*}
\begin{split}
Q^2\,D_{E}(x_0,\lambda^k  s\bar{r}) \leq \lambda^{k(1-\gamma)}\varepsilon_{\text{dec}}(\lambda), \quad \forall k\in\mathbb{N}.
\end{split}
\end{equation*}
Therefore,
\begin{equation*}
\begin{split}
Q^2\,\int_{B_{\lambda^k s\bar{r}}(x_0)}|\nabla u|^2\,dx\leq C\,(\lambda^k s\bar{r})^{(n-\gamma)}, \quad  \forall k\in\mathbb{N}
\end{split}
\end{equation*}
for some constant $C>0$.
Now if we take any $\tilde{r}\leq\lambda s\bar{r}$, there exists an integer
$k>0$ such that $\lambda^{k+1}s\bar{r}<\tilde{r}\leq\lambda^k s\bar{r}$, hence
$$Q^2\,\int_{B_{\tilde{r}}(x_0)}|\nabla u|^2\,dx\leq Q^2\,\int_{B_{\lambda^k s\bar{r}}(x_0)}|\nabla u|^2\,dx\leq C\,(\lambda^k s\bar{r})^{(n-\gamma)}\leq \frac{C}{\lambda^{n-\gamma}}\,\tilde{r}^{(n-\gamma)}.
$$

\end{proof}
\begin{proposition} \label{mdirichletp}
Let $E$ be a minimizer of \eqref{e:problem}, let $(u,\rho)\in\mathcal{A}(E)$ be the minimizing pair of $\G_{\beta,K}(E)$, $x_0\in\partial E$, and $f\in C^{1,\vartheta}(\D(x_{0}',r))$. Suppose that $Q\leq 1$ and
$$E\cap \C(x_0,r)=\left\{ x=(x',x_n)\in \D(x_{0}',r)\times\R\,:\,x_{n}<f(x') \right\} \cap \C(x_0,r),$$ 
for some $0<r\leq \min\{\bar{r},1\}$, where $\bar{r}$ is as in Lemma \ref{Dirichlet bounded}. Then there exist 
$\alpha~=~\alpha(\vartheta)~\in~(0,1)$ and a constant $C=C(n,\beta,\vartheta,\|\rho\|_{\infty})>0$ such that 
\begin{equation}\label{mdirichlet}
                         Q^2\,\int_{B_{\lambda r}(x_{0})}|T_E u-[T_E u]_{x_{0},\lambda r}|^{2}\,dx\leq C\,Q^2 \lambda^{n+2\alpha} \int_{B_{r}(x_{0})}|T_E u-[T_E u]_{x_{0},r}|^{2}\,dx+C \,r^{n+\alpha}.
                          \end{equation}
\end{proposition}
\begin{proof} Without loss of generality assume $0\in\partial E$, $x_{0}=0$. 
Let $\lambda \in (0,1/2)$ be given and let $v$ be the solution of
\[
\begin{split}
\left\{ \begin{array}{cc}
                        -{\rm div}(a_{H}\nabla v)=\rho &\quad \text{ in }B_{r/2},\\ 
                           v=u  &\quad \text{ on }\partial B_{r/2},
           \end{array}
\right.
\end{split}
 \]
 where $H$ is the half-space $\{x=(x',x_n): x_n<0\}$.
In particular, $w=v-u\in W^{1,2}_{0}(B_{r/2})$ and
\begin{equation} \label{eqw}
-{\rm div}(a_{H}\nabla w)=-{\rm div}\left((a_{E}-a_{H})\nabla u\right).
\end{equation}
Since $\left[T_E g\right]_{s}$ minimizes the functional $m\mapsto \int_{B_{s}}|T_E g-m|^{2}\,dx $, we have
\begin{equation} \label{TEesti}
\begin{split}
\int_{B_{\lambda r}} |T_E u-[T_E u]_{\lambda r}|^2\,dx 
&\leq \int_{B_{\lambda r}} |T_E u-[T_H u]_{\lambda r}|^2\,dx\\
&\leq 2 \left( \int_{B_{\lambda r}} |T_H u-[T_H u]_{\lambda r}|^2\,dx+\int_{B_{\lambda r}} |T_E u-T_H u|^2\,dx  \right).
\end{split}
\end{equation}
We want now to estimate the first term in the right hand side of \eqref{TEesti}. 
Notice that, since $u=v-w$, by linearity of $T_H$ we have
\[
|T_H u-[T_H u]_{\lambda r}|^2 \leq 2 \left( |T_H v-[T_H v]_{\lambda r}|^2+|T_H w-[T_H w]_{\lambda r}|^2  \right).
\]
Hence, integrating the above inequality on $B_{\lambda r}$ we obtain
\begin{equation} \label{TEesti1}
\begin{split}
\int_{B_{\lambda r}}|T_H u-[T_H u]_{\lambda r}|^2\,dx &\leq 2 \left( \int_{B_{\lambda r}}|T_H v-[T_H v]_{\lambda r}|^2\,dx+\int_{B_{\lambda r}}|T_H w-[T_H w]_{\lambda r}|^2\,dx  \right)\\
&\leq  2 \left( \int_{B_{\lambda r}}|T_H v-[T_H v]_{\lambda r}|^2\,dx+\int_{B_{\lambda r}}|T_H w|^2\,dx  \right)\\
&\leq C\,\left( \int_{B_{\lambda r}} |\nabla w|^2\,dx+ \int_{B_{\lambda r}} |T_H v-[T_H v]_{\lambda r}|^2\,dx\right).
\end{split}
\end{equation}
To estimate the second term in the right hand side of \eqref{TEesti}, recall the Notation \ref{TE}
\[
\partial_{\nu_E^{\perp}} u= \nabla u-(\nabla u \cdot \nu_E)\,\nu_E \quad \text{and } \quad \partial_{e_{n}^{\perp}} u= \nabla u-(\nabla u \cdot e_n)e_n.
\]
Hence,
\begin{equation*}
\begin{split}
 &\vert T_E u-T_H u\vert=\vert\nabla u-(\nabla u \cdot \nu_E)\nu_E+(1+(\beta-1)\1_E)(\nabla u \cdot \nu_E)\,\nu_E\\
 &\qquad\qquad-\left(\nabla u-(\nabla u \cdot e_n)e_n+(1+(\beta-1)\1_H)(\nabla u \cdot e_n)\,e_n\right)\vert\\
 &\,=\vert(\nabla u \cdot e_n)e_n-(\nabla u \cdot \nu_E)\nu_E+(1+(\beta-1)\1_E)(\nabla u \cdot \nu_E)\,\nu_E
 -\left((1+(\beta-1)\1_H)(\nabla u \cdot e_n)\,e_n\right)\vert\\
 &\,\leq(1+\beta)\vert(\nabla u \cdot e_n)e_n-(\nabla u \cdot \nu_E)\nu_E\vert+\vert\big((1+(\beta-1)\1_E) -(1+(\beta-1)\1_H)\big)(\nabla u \cdot e_n)\,e_n\vert\\
 &\,=(1+\beta)\vert\big((\nabla u \cdot e_n)-(\nabla u \cdot \nu_E)\big)e_n+(\nabla u \cdot \nu_E)(e_n-\nu_E)\vert+(\beta-1)\1_{E\Delta H}\vert\nabla u \cdot e_n\vert\\
 &\leq\big(2(1+\beta)\vert\nu_E-e_n\vert+(\beta-1)\1_{E\Delta H}\big)\vert\nabla u\vert.  
\end{split} 
\end{equation*}
Therefore,
\begin{equation} \label{TEesti2}
\int_{B_{\lambda r}} |T_E u-T_H u|^2\,dx \leq C \left(\int_{B_{\lambda r}} |\nabla u|^2\, |\nu_E-e_n|^2\,dx+\int_{B_{\lambda r}} |\nabla u|^2\, \1_{E\Delta H}\,dx\right).
\end{equation}
Combining \eqref{TEesti}, \eqref{TEesti1} and \eqref{TEesti2} we obtain
\begin{equation*}
\begin{split}
 &\int_{B_{\lambda r}}|T_E u-[T_E u]_{\lambda r}|^{2}\,dx  \leq C\, \int_{B_{\lambda r}}|T_H v-[T_H v]_{\lambda r}|^{2}\,dx\\
&\qquad+ C\,\int_{B_{r/2}}|\nabla w|^{2}\,dx+C \left(\int_{B_{\lambda r}} |\nabla u|^2\, |\nu_E-e_n|^2\,dx+\int_{B_{\lambda r}} |\nabla u|^2\, \1_{E\Delta H}\,dx\right).
\end{split}
\end{equation*}
By Lemma \ref{piano} we have 
\begin{equation} \label{conti0}
\begin{split}
 \int_{B_{\lambda r}}|T_H v-[T_H v]_{\lambda r}|^{2}\,dx &\leq C\,\lambda^{n+2\gamma} \int_{B_{r/2}}|T_H v-[T_H v]_{r/2}|^{2}\,dx+C\,r^{n+1}.
\end{split}
\end{equation}
By arguing as above one can easily see that
\begin{equation*}
\begin{split}
 &\int_{B_{r/2}}|T_H v-[T_H v]_{r/2}|^{2}\,dx  \leq C\,\int_{B_{r/2}}|T_E u-[T_E u]_{r/2}|^{2}\,dx\\
&\qquad +C\,\int_{B_{r/2}}|\nabla w|^{2}\,dx+ C \left(\int_{B_{r/2}} |\nabla v|^2\, |\nu_E-e_n|^2\,dx+\int_{B_{r/2}} |\nabla v|^2\, \1_{E\Delta H}\,dx\right).
\end{split}
\end{equation*}
We note that
\begin{equation*}
\begin{split}
 &\int_{B_{r/2}} |\nabla v|^2\, |\nu_E-e_n|^2\,dx+\int_{B_{r/2}} |\nabla v|^2\, \1_{E\Delta H}\,dx\\
 &\qquad\leq  2\int_{B_{r/2}} |\nabla w|^2\, |\nu_E-e_n|^2\,dx+\int_{B_{r/2}} |\nabla u|^2\, |\nu_E-e_n|^2\,dx\\
 &\qquad\qquad+2\int_{B_{r/2}} |\nabla w|^2\, \1_{E\Delta H}\,dx+2\int_{B_{r/2}} |\nabla u|^2\, \1_{E\Delta H}\,dx\\
 &\qquad\leq  8\int_{B_{r/2}} |\nabla w|^2\,dx+\int_{B_{r/2}} |\nabla u|^2\, |\nu_E-e_n|^2\,dx\\
 &\qquad\qquad+2\int_{B_{r/2}} |\nabla w|^2\,dx+2\int_{B_{r/2}} |\nabla u|^2\, \1_{E\Delta H}\,dx 
.
\end{split}
\end{equation*}
Bringing it all together, we get
\begin{equation} \label{impes}
\begin{split}
 &\int_{B_{\lambda r}}|T_E u-[T_E u]_{\lambda r}|^{2}\,dx \leq C\,\lambda^{n+2\gamma} \int_{B_{r/2}}|T_E u-[T_E u]_{r/2}|^{2}\,dx\\
&\qquad+C\, \int_{B_{r/2}} |\nabla u|^2\,|\nu_E-e_n|^2  \,dx+C\int_{B_{r/2}} |\nabla u|^2\, \1_{E\Delta H}\,dx+C\,\int_{B_{r/2}}|\nabla w|^{2}\,dx .
\end{split}
\end{equation}
We need to estimate the last three terms in the right hand side of the above inequality.
Since $E$ is parametrised by $f\in C^{1,\vartheta}(\D_r)$ in the cylinder $\C(x_0,r)$, 
there exists a constant $C>0$ such that
\begin{equation} \label{condsigma}
\frac{\left|(E\Delta H)\cap B_{r}\right|}{\left| B_{r}\right|}\leq C\, r^{\vartheta}.
\end{equation}
We will estimate the last two terms together. By testing \eqref{eqw} with $w$ we deduce
\begin{equation} \label{tantestime}
\begin{split}
\int_{B_{r/2}}\left|\nabla w\right|^{2}\,dx\leq\int_{B_{r/2}}a_H\,\left|\nabla w\right|^{2}\,dx=\int_{B_{r/2}} (a_E-a_H)\,\nabla u\cdot \nabla w\, dx.
\end{split}
\end{equation}
By applying H\"older inequality in \eqref{tantestime} we obtain
\begin{equation} \label{conti1}
\begin{split}
&C\int_{B_{r/2}} |\nabla u|^2\, \1_{E\Delta H}\,dx+C\,\int_{B_{r/2}}|\nabla w|^{2}\\
&\qquad\leq C\int_{B_{r/2}} |\nabla u|^2\, \1_{E\Delta H}\,dx+C\int_{B_{r/2}}(a_E-a_H)^2\,\left|\nabla u\right|^{2}\,dx\\
&\qquad\qquad\leq C\, \int_{(E\Delta H)\cap B_{r/2}}\left|\nabla u\right|^2\,dx.
\end{split}
\end{equation}
By the higher integrability \cite[Lemma 6.1]{DPHV}, there exists $p>1$ such that 
\begin{equation} \label{conti2}
\left(\frac{1}{\vert B_{r/2}\vert}\int_{B_{r/2}}\left|\nabla u \right|^{2p}\,dx \right)^{\frac{1}{p}} \leq C\,\frac{1}{\vert B_r\vert}\int_{B_r}|\nabla u|^2\,dx+C\,r^{n+2}\,\| \rho\|^2_\infty.
\end{equation}
Hence by exploiting H\"older inequality, \eqref{condsigma}, and \eqref{conti2} we have
\begin{equation} \label{conti3}
\begin{aligned}
 \int_{(E\Delta H)\cap B_{r/2}}\left|\nabla u\right|^2\,dx &\leq \left|(E\Delta H)\cap B_{r/2}\right|^{1-\frac{1}{p}}\, \left(\int_{B_{r/2}}\left|\nabla u\right|^{2p}\,dx\right)^{\frac{1}{p}}\\
&\leq C\,|B_r| \left(\frac{ \left|(E\Delta H)\cap B_{r}\right|}{|B_r|}\right)^{1-\frac{1}{p}}\,\left(\frac{1}{\vert B_{r/2}\vert}\int_{B_{r/2}}\left|\nabla u\right|^{2p}\,dx\right)^{\frac{1}{p}}\\
&\leq C\,r^{\vartheta\,\left(1-\frac{1}{p}\right)}\,\left\{\int_{B_r} \left|\nabla u\right|^2\,dx+r^{n+2}\,\|\rho\|^2_\infty\right\}.
\end{aligned}
\end{equation}
Therefore, \eqref{conti1} together with \eqref{conti3} (recall $r<1$) yield
\begin{equation} \label{Dw}
C\int_{B_{r/2}} |\nabla u|^2\, \1_{E\Delta H}\,dx+C\,\int_{B_{r/2}}|\nabla w|^{2} \leq C\,\left\{r^{\vartheta\left(1-\frac{1}{p}\right)}\int_{B_{r}}|\nabla u|^{2}+r^{n+2}\| \rho\|^{2}_{\infty}\right\}.
\end{equation}
On the other hand, by Lemma \ref{Dirichlet bounded} we have
\begin{equation}\label{Du}
Q^2\,\int_{B_{s}} |\nabla u|^{2}\,dx\leq C\, s^{n-\gamma} \quad \forall\, s<\bar{r}.
\end{equation}
Hence, combining \eqref{Dw} and \eqref{Du}, we obtain
\begin{equation*} 
Q^2\left(C\int_{B_{r/2}} |\nabla u|^2\, \1_{E\Delta H}\,dx+C\,\int_{B_{r/2}}|\nabla w|^{2}\right) \leq C\,\left\{r^{\vartheta\left(1-\frac{1}{p}\right)+n-\gamma}+r^{n+2}\| \rho\|^{2}_{\infty}\right\}.
\end{equation*}
Finally, we estimate the second term in \eqref{impes}. Notice that
\[
\begin{split}
\int_{B_{r/2}} |\nabla u|^2\,|\nu_E-e_n|^2  \,dx &=\int_{B_{r/2}} |\nabla u(x',x_n)|^2\,|\nu_E(x',x_n)-e_n|^2  \,dx\\
&=\int_{B_{r/2}} |\nabla u|^2\,|\nu_E(x',f(x'))-e_n|^2  \,dx.
\end{split}
\]
Since $\sqrt{1+t} \leq 1+\frac{t}{2}$ for every $t>0$,
\begin{equation} \label{nonso22}
\left| \nu_E(x',f(x'))-e_n \right|^2=2-\frac{2}{\sqrt{1+|\nabla f(x')|^2}} \leq  2\, \left( \frac{\sqrt{1+|\nabla f(x')|^2}-1}{\sqrt{1+|\nabla f(x')|^2}} \right) \leq |\nabla f(x')|^2.
\end{equation}
Thanks to \eqref{Du} and \eqref{nonso22}, and using that $\nabla f$ is $\vartheta$-H\"older, we deduce
\begin{equation} \label{equazionefinale}
Q^2\,\int_{B_{r/2}} |\nabla u|^2\,|\nu_E-e_n|^2  \,dx \leq C\, r^{n+2\vartheta-\gamma}.
\end{equation}
Let 
\[
\alpha:=\min \left\{ \gamma, \vartheta \left(1-1/p\right)-\gamma, 2\vartheta-\gamma\right\}.
\]
Therefore, by multiplying \eqref{impes} and \eqref{Dw} with $Q^2$ and by recalling that $Q<1$ we have that \eqref{equazionefinale} implies \eqref{mdirichlet}.
\end{proof}
We are now ready to prove that $u$ is regular up to the boundary.
Recall that $u^+=u\, \1_E$ and $u^-=u\,\1_{E^c}$.
\begin{theorem} \label{ureg}
Let $E$ be a minimizer of \eqref{e:problem}, let $(u,\rho)\in\mathcal{A}(E)$ be the minimizing pair of $\G_{\beta,K}(E)$ and $f\in C^{1,\vartheta}(\D(x_{0}',r))$. Suppose $Q\leq 1$ and
$$E\cap \C(x_0,r)=\left\{ x=(x',x_n)\in \D(x_{0}',r)\times\R\,:\,x_{n}<f(x') \right\} \cap \C(x_0,r)$$ 
for some $0<r\leq \min\{\bar{r},1\}$, where $\bar{r}$ is as in Lemma \ref{Dirichlet bounded}.
Then there exists $\eta~=~\eta(\vartheta)~\in~(0,1)$ such that $u^+\in C^{1,\eta}(\overline{E}\cap \C_{r/2}(x_0))$ and $u^- \in C^{1,\eta}(\overline{E}^{c}\cap \C_{r/2}(x_0))$.
Furthermore, let $A>0$ and let $\beta,K,Q$ be controlled by $A$ and $R\geq 1$. Then there exists a universal constant $C=C(n,A)> 0$ such that 
\begin{equation} \label{darichiamare}
\Vert Q\,u^+ \Vert_{C^{1,\eta}(\overline{E}\cap \C_{r/2}(x_0))}\leq C \;\text{ and }\; \Vert Q\,u^- \Vert_{C^{1,\eta}(\overline{E}^{c}\cap \C_{r/2}(x_0))}\leq C.
\end{equation}
\end{theorem}
\begin{proof} Let $u_Q:=Q\,u$.
By Proposition \ref{mdirichletp} there exists $C=C(n,\beta,\gamma,\|\rho\|_{\infty})>0$ such that
\begin{equation}
                         \int_{B_{\lambda r}(x_{0})}|T_E u_Q-[T_E u_Q]_{x_{0},\lambda r}|^{2}\,dx\leq C \lambda^{n+2\alpha} \int_{B_{r}(x_{0})}|T_E u_Q-[T_E u_Q]_{x_{0},r}|^{2}\,dx+C \,r^{n+\alpha},
                          \end{equation}
where $\alpha\in(0,1)$ is as in Proposition \ref{mdirichletp}.
Therefore, Lemma \ref{iterating lemma} implies that there exists a universal constant $C=C(n,A)>0$ such that
\begin{equation} 
\frac{1}{\vert B_r\vert}\int_{B_r(x_0)}|T_E u_Q-[T_E u_Q]_{x,r}|^2\,dy\leq C\,\left(\frac{r}{R}\right)^{2\eta}, \quad \forall B_r(x_0)\subset B_R.
\end{equation}
for some $\eta=\eta(\vartheta)\in (0,1)$. Hence, by  Lemma \ref{campanato}, recalling the definition of \(T_E \), we get $u_Q\1_E\in C^{1,\eta}(\overline{E}\cap \C_{r/2}(x_0))$ and $u_Q\1_{E^c}\in C^{1,\eta}(\overline{E}^{c}\cap \C_{r/2}(x_0))$ and \eqref{darichiamare}.
\end{proof}
In the next proposition we rewrite the Euler-Lagrange equation (see Theorem \ref{properties} (ii)) in a more convenient form by exploiting the regularity of $\partial E$.
\begin{proposition} [Euler-Lagrange equation] \label{eulerl}
Let $E$ be a minimizer for \eqref{e:problem} and $(u,\rho)\in\A(E)$. Assume that $f\in C^{1,\vartheta}(\D(x_{0}',r))$ and
$$E\cap \C(x_0,r)=\left\{ x=(x',x_n)\in \D(x_{0}',r)\times\R\,:\,x_{n}<f(x') \right\} \cap \C(x_0,r).$$ 
Then there exists a constant $C>0$ such that
\begin{equation} \label{wel1} 
\begin{split}
-{\rm div}\left(\frac{\nabla f(x')}{\sqrt{1+|\nabla f(x')|^{2}}}\right) &=Q^2 \left(\beta|\nabla u^+|^{2}-|\nabla u^-|^{2}+K\, \rho^{2}\right)(x',f(x'))\\
&-2Q^2\,  \left(\beta \partial_n u^+\,\nabla u^+-\partial_n u^-\,\nabla u^-   \right)(x',f(x'))\cdot(-\nabla f(x'),1)+C
\end{split}
\end{equation}
weakly in $\D(x_{0}',r)$.
\end{proposition}
\begin{proof} Let $E\subset\R^n$ be a minimizer of  $(\mathcal{P}_{\beta,K,Q,R})$ and let $(u,\rho)\in\A(E)$.

Notice that $E \cap \C(x_0,r)$ is an open set of $\R^n$. Moreover, by an approximation argument, we can integrate over $E\cap \C(x_0,r)$ the following identity,
\[
\begin{split}
|\nabla u^+|^2\,\Div\eta &=\Div(|\nabla u^+|^2\eta)-\nabla|\nabla u^+|^2\cdot \eta\\
&=\Div(|\nabla u^+|^2\eta)-2\,\Div(\nabla u^+(\nabla u^+\cdot\eta))+2 \Delta u^+\,\nabla u^+\cdot\eta+2\nabla u^+\cdot(\nabla\eta\nabla u^+)
\end{split}
\]
for every $\eta\in C^{\infty}_c(\C(x_0,r),\R^n)$.
Therefore,
\begin{equation} \label{boh}
\begin{split}
\int_{E\cap \C(x_0,r)}\left( |\nabla u^+|^2\,\Div\eta-2\nabla u^+\cdot(\nabla\eta\nabla u^+) \right)\,dx &=\int_{E\cap\C(x_0,r)} \Div(|\nabla u^+|^2\eta)\,dx\\
&-\int_{E\cap \C(x_0,r)}2\,\Div(\nabla u^+(\nabla u^+\cdot\eta)) \, dx\\
&+\int_{E\cap\C(x_0,r)}2 \Delta u^+\,\nabla u^+\cdot\eta\,dx.
\end{split}
\end{equation}
On the other hand, since $(u,\rho)\in \mathcal{A}(E)$, we have
\[
-\beta \Delta u^+=\rho \quad\text{in}\quad\mathcal{D}'(E\cap \C(x_0,r)).
\]
Moreover, by Theorem \ref{properties} (i) we deduce
\[
\nabla u^+=-K\,\nabla \rho \quad\text{in}\quad E\cap\C(x_0,r).
\]
Then, by multiplying equation \eqref{boh} by $\beta$, we  have
\begin{equation} \label{boh2}
\begin{split}
\int_{E\cap \C(x_0,r)}\beta\left( |\nabla u^+|^2\,\Div\eta-2\nabla u^+\cdot(\nabla\eta\nabla u^+) \right)\,dx &=\int_{E\cap\C(x_0,r)} \beta \Div(|\nabla u^+|^2\eta)\,dx\\
&-\int_{E\cap \C(x_0,r)}2\beta\,\Div(\nabla u^+(\nabla u^+\cdot\eta)) \, dx\\
&+K\,\int_{E\cap\C(x_0,r)}2 \rho\,\nabla \rho\cdot\eta\,dx.
\end{split}
\end{equation}
Integrating by parts the first and the second term in the right hand side of \eqref{boh2}, we can write
\begin{equation} \label{boh3}
\begin{split}
\int_{E\cap \C(x_0,r)}\beta\left( |\nabla u^+|^2\,\Div\eta-2\nabla u^+\cdot(\nabla\eta\nabla u^+) \right)\,dx &=\int_{\partial E\cap\C(x_0,r)} \beta |\nabla u^+|^2\eta\cdot\nu_E\,d\Hf^{n-1}\\
&-\int_{\partial E\cap \C(x_0,r)}2\beta\,(\nabla u^+\cdot\eta)(\nabla u^+\cdot\nu_E)\, d\Hf^{n-1}\\
&+K\,\int_{E\cap\C(x_0,r)}2 \rho\,\nabla \rho\cdot\eta\,dx.
\end{split}
\end{equation}
By arguing similarly as above, one can also prove 
\begin{equation} \label{boh4}
\begin{split}
\int_{E^c\cap \C(x_0,r)}\left( |\nabla u^-|^2\,\Div\eta-2\nabla u^-\cdot(\nabla\eta\nabla u^-) \right)\,dx &=\int_{E^c\cap\C(x_0,r)} \Div(|\nabla u^-|^2\eta)\,dx\\
&-\int_{E^c\cap \C(x_0,r)}2\,\Div(\nabla u^-(\nabla u^-\cdot\eta)) \, dx.
\end{split}
\end{equation}
Integrating by parts the right hand side of \eqref{boh4}, we can write
\begin{equation} \label{boh5}
\begin{split}
\int_{E^c\cap \C(x_0,r)}\left( |\nabla u^-|^2\,\Div\eta-2\nabla u^-\cdot(\nabla\eta\nabla u^-) \right)\,dx &=-\int_{\partial E\cap\C(x_0,r)} |\nabla u^-|^2\eta\cdot\nu_E\,d\Hf^{n-1}\\
&+\int_{\partial E\cap \C(x_0,r)}2\,(\nabla u^-\cdot\eta)\,(\nabla u^-\cdot\nu_E) \, d\Hf^{n-1}.
\end{split}
\end{equation}
Therefore, combining \eqref{boh3} and \eqref{boh5}, we get
\begin{equation} \label{eul1}
\begin{split}
\int_{\R^{n}} & a_{E}\left({\rm div}\eta \left|\nabla u\right|^{2}-2\nabla u\cdot(\nabla \eta\nabla u)\right) dx =\int_{\partial E} \left( \beta |\nabla u^+|^{2}- |\nabla u^-|^{2} \right)\eta\cdot\nu_{E}\,d\Hf^{n-1}\\
&-\int_{\partial E\cap \C(x_0,r)}2 \big( \beta\,(\nabla u^+\cdot\eta)(\nabla u^+\cdot\nu_E)-(\nabla u^-\cdot\eta)\,(\nabla u^-\cdot\nu_E) \big)\, d\Hf^{n-1}\\
&+K\,\int_{E\cap\C(x_0,r)}2 \rho\,\nabla \rho\cdot\eta\,dx.
\end{split}
\end{equation}
Notice that the following identity holds true
\begin{equation} \label{eul2}
\begin{split}
K\,\int_{\R^{n}}\rho^{2}{\rm div}\eta &=K\,\int_{E\cap \C(x_0,r)}\Div(\rho^2\eta)\,dx-K\,\int_{E\cap \C(x_0,r)}2\rho\,\nabla\rho\,\cdot\eta\,dx\\
&=K\,\int_{\partial E\cap\C(x_0,r)} \rho^{2} \eta\cdot\nu_{E}\,d\Hf^{n-1}-K\, \int_{E\cap \C(x_0,r)}2\rho\,\nabla\rho\,\cdot\eta\,dx.
\end{split}
\end{equation}
Combining the Euler-Lagrange equation of Theorem \ref{properties} (ii), \eqref{eul1} and \eqref{eul2}, we find
\begin{equation} \label{wel0}
\begin{split}
\int_{\partial E}{\rm div}_{E}\eta\, d\Hf^{n-1} &=Q^2\int_{\partial E} \left(\beta|\nabla u^+|^{2}-|\nabla u^-|^{2}+K\, \rho^{2} \right)\,\eta\cdot\nu_{E} \,d\Hf^{n-1}\\
&-2Q^2\int_{\partial E}\beta(\eta\cdot\nabla u^+)\,(\nabla u^+\cdot\nu_E)-(\eta\cdot\nabla u^-)\,(\nabla u^-\cdot\nu_E)   \,d\Hf^{n-1}
\end{split}
\end{equation}
for every $\eta\in C^{1}_c(B_r(x_0),\R^n)$ with $\int_{E}\Div\eta\,dx=0$. 
%
%
%
%
%

Now we are ready to prove \eqref{wel1}.
The tangential divergence of $\eta$ on $\partial E$ is
\begin{equation} \label{tdiv}
{\rm div}_E\eta:={\rm div}\eta-\sum_{i,j=1}^{n}(\nu_E)_i\,(\nu_E)_j\,\partial_j\eta_i \quad \text{on }\partial E,
\end{equation}
where $\nu_E: \partial E\to\mathbb{S}^{n-1}$ is the normal vector to $\partial E$:
\begin{equation*}
\nu_E:=\frac{1}{\sqrt{1+|\nabla f|^2}}(-\nabla f,1).
\end{equation*}
Let $\eta:=(0,\dots,0,\eta_n)$, then by \eqref{tdiv} we have
\begin{equation} \label{eqdivergenzatanpart}
{\rm div}_E\eta:=\partial_n \eta_n+\frac{1}{1+|\nabla f|^2}\,\left\{\sum_{j=1}^{n-1} \partial_j\eta_n\,\partial_j f-\partial_n \eta_n\right\} \quad \text{on }\partial E.
\end{equation}
Choose $\eta_n(x):=\varphi(\p x)\,s(x_n)$, where $\varphi\in C^{1}_c(\D(x_0',r))$ is such that $\int_{\D(x_0',r)}\varphi=0$ and $s:(-1,1)\to\R^n$ is such that $s(t)=1$ for every $|t|\leq\|f\|_\infty$.
Since now $\eta_n$ does not depend on the $n$-th component on $\partial E$, we have
\begin{equation} \label{pscalarenueta}
\eta\cdot\nu_E=\frac{\varphi(\p x)}{\sqrt{1+|\nabla f|^2}} \quad \text{on }\partial E\cap \C(x_0,r),
\end{equation}
and the above equation \eqref{eqdivergenzatanpart} reads as
\begin{equation} \label{tdiveta}
{\rm div}_E\eta:=\frac{1}{1+|\nabla f|^2}\,\nabla\varphi\cdot\nabla f \quad \text{on }\partial E\cap \C(x_0,r).
\end{equation}
Moreover, 
\[
\begin{split}
&\int_E {\rm div}\eta\,dx =\int_{\partial E} (\eta\cdot\nu_E)\,d\Hf^{n-1}=\int_{\partial E \cap \C(x_0,r)} \eta_n (\nu_E \cdot e_n)\,d\Hf^{n-1}\\
&= \int_{\partial E \cap \C(x_0,r)} \varphi(\p x) s(f(x))\, (\nu_E \cdot e_n)\,d\Hf^{n-1}\\
&= \int_{\partial E \cap \C(x_0,r)} \frac{\varphi(\p x)}{\sqrt{1+|\nabla f(\p x)|}}\,d\Hf^{n-1}= \int_{\p(\partial E \cap \C(x_0,r))} \varphi \,dx=0.
\end{split}
\]
This implies that \(\eta\) is admissible in \eqref{wel0}. Hence by using $\eta$ as a test function in \eqref{wel0}, by combining \eqref{pscalarenueta} and \eqref{tdiveta}, we have
\begin{equation*}
\begin{split}
&\int_{\D(x_{0}',r)}\frac{\nabla f}{1+|\nabla f|^2}\cdot\nabla\varphi\, dx' =Q^2\int_{\D(x_{0}',r} \left(\beta|\nabla u^+|^{2}-|\nabla u^-|^{2}+K\, \rho^{2} \right)(x',f(x'))\,\frac{\varphi(x')}{\sqrt{1+|\nabla f|^2}} \,dx'\\
&\qquad\qquad-2Q^2\int_{\D(x_{0}',r}\left(\beta\partial_n u^+\nabla u^+-\partial_n u^-\nabla u^-\right)(x',f(x'))\cdot\frac{(-\nabla f,1)}{\sqrt{1+|\nabla f|^2}}\,\varphi(x')  \,dx'
\end{split}
\end{equation*}
for any $\varphi\in C^{1}_c(\D(x_0',r))$ with $\int_{\D(x_0',r)}\varphi=0$. It remains to multiply this equality by 
$\sqrt{1+|\nabla f|^2}$ and use divergence theorem on the left-hand side.

%
%
%
%
%
%
\end{proof}
\begin{corollary}\label{cor:EL on the boundary} 
Let $E$ be a minimizer for \eqref{e:problem} and $(u,\rho)\in\A(E)$. Assume that $f\in C^{1,\vartheta}(\D(x_{0}',r))$ and
$$E\cap \C(x_0,r)=\left\{ x=(x',x_n)\in \D(x_{0}',r)\times\R\,:\,x_{n}<f(x') \right\} \cap \C(x_0,r).$$ 
Then there exists a vector field $M:\mathbb{R}^n\rightarrow\mathbb{R}^n$ such that the matrix 
$\nabla M(\nabla f)$ is uniformly elliptic and H\"older continuous and a H\"older continuous
function $G$ such that
\begin{equation*}
-{\rm div}\left(\nabla M(\nabla f)\,\nabla\partial_i f\right)=\partial_i G \quad \text{weakly on }\partial E\cap \C(x_0,r/2)
\end{equation*}
for every $i=1,\dots,n$.
\end{corollary}
\begin{proof}
Exploiting Proposition \ref{eulerl}, we have
\begin{equation} \label{nonso}
-{\rm div}\left(\frac{\nabla f(x')}{\sqrt{1+|\nabla f(x')|^{2}}}\right)=G(x',f(x')) \quad \text{for a.e. }x'\in\D(x_0',r/2),
\end{equation}
where
\[
\begin{split}
G(x',f(x')) &=Q^2\,\big(\beta|\nabla u^+|^{2}-|\nabla u^-|^{2}+ K\,\rho^{2}\big)(x',f(x')) \\
&-2Q^2\,\left(\beta \partial_n u^+\,\nabla u^+-\partial_n u^-\,\nabla u^-   \right)(x',f(x'))\cdot(-\nabla f(x'),1)+C, \quad x'\in\D(x_0',r/2).
\end{split}
\]
Hence, \eqref{nonso} is equivalent to
\begin{equation} \label{nonso2}
-{\rm div}\left(M(\nabla f)\right)=G \quad \text{a.e. on }\partial E\cap \C(x_0,r/2),
\end{equation}
where
\[
M(\xi):=\frac{\xi}{\sqrt{1+|\xi|^2}}, \quad \forall\xi\in\R^n.
\]
By \cite[Theorem 27.1]{Maggi12} we can take the derivatives of \eqref{nonso2}. Then,
\begin{equation*}
-{\rm div}\left(\nabla M(\nabla f)\,\nabla\partial_i f\right)=\partial_i G \quad \text{a.e. on }\partial E\cap \C(x_0,r/2)
\end{equation*}
for every $i=1,\dots,n$.
Notice that
\[
\nabla M(\xi)=\frac{1}{\sqrt{1+|\xi|^2}}\left(\Id-\frac{\xi\otimes\xi}{1+|\xi|^2} \right) \quad \forall\xi\in\R^n,
\]
meaning that the matrix $\nabla M(\nabla f)$ is uniformly elliptic, more precisely 
\[
|\eta|^2 \geq \nabla M(\nabla f)\eta\cdot\eta\geq (1+\|\nabla f\|_{\infty})^{-3/2}\,|\eta|^2 \quad \forall \eta\in\R^n.
\]
It follows from Theorem \ref{ureg} that $G$  is H\"older continuous. By the definition of $M$ and by the regularity of $f$ we also have that $\nabla M(\nabla f)$ is H\"older continuous.
\end{proof}
We prove now the partial  $C^{2,\vartheta}$-regularity of minimizers. 
\begin{theorem} [$C^{2,\vartheta}$-regularity] \label{C2 reg}
Given \(n\ge 3\), \(A>0\) and \(\vartheta\in (0,1/2)\), there exists \(\varepsilon_{\textnormal{reg}}=\varepsilon_{\textnormal{reg}}(n, A, \vartheta)>0\)  such that if \(E\) is minimizer of\(~\eqref{e:problem}\), \(Q+\beta+K+\frac{1}{K}\le A\), $x_0\in\partial E$, and 
\[
r+\e_E(x_0, r)+Q^2\,D_E(x_0, r)\le \varepsilon_{\textnormal{reg}},
\] 
then \(E\cap \C(x_0,r/2)\)  coincides with  the  epi-graph of a \(C^{2,\vartheta}\)-function $f$. 
In particular, we have that \(\partial E\cap \C(x_0,r/2)\) is a \(C^{2,\vartheta}\) \((n-1)\)-dimensional manifold and
\begin{equation} \label{nholderf}
[f]_{C^{2,\vartheta}(\D(x_0',r/2))}\leq C\left(n,A,r, \vartheta\right).
\end{equation}
\end{theorem}
%
%
\begin{proof} 
Choose $\varepsilon_{\textnormal{reg}}$ as in Theorem \ref{eps reg}. Then there exists $f\in C^{1,\vartheta}(\D(x_0',r/2))$ such that
$$E\cap \C(x_0,r/2)=\left\{ x=(x',x_n)\in \D(x_{0}',r/2)\times\R\,:\,x_{n}<f(x') \right\}.$$
By Corollary \ref{cor:EL on the boundary} we have
\begin{equation*}
-{\rm div}\left(\nabla M(\nabla f)\,\nabla\partial_i f\right)=\partial_i G \quad \text{a.e. on }\partial E\cap \C(x_0,r/2)
\end{equation*}
for every $i=1,\dots,n$, with $\nabla M(\nabla f)$ uniformly elliptic and $G$ - H\"older continuous. We also have that $\nabla M(\nabla f)$ is H\"older continuous. Hence the following Schauder estimates hold in this case
\[
[\nabla \partial_i f]_{C^{0,\vartheta}(\D(x_0',r/2))}\leq C\,\{\|\partial_i f\|_{L^2(\D(x_0',r/2))}+[G]_{C^{0,\eta}(\D(x_0',r/2))}\} \quad \forall i=1,\dots,n,
\]
for some constant $C$ depending on $r$. In particular, $f$ is $C^{2,\vartheta}$ and
\[
[f]_{C^{2,\vartheta}(\D(x_0',r/2))}\leq C\,\{\|\nabla f\|_{L^2(\D(x_0',r/2))}+[G]_{C^{0,\eta}(\C(x_0,r/2))}\}.
\]
By the definition of $G$, recalling \eqref{darichiamare} and Theorem \ref{properties} (i), using Poincar\'e inequality and since $f$ is Lipschitz, one can easily see that there exists $C=C(n,A,\vartheta,r)>0$ such that
$$[G]_{C^{0,\vartheta}(\C(x_0,r/2))} \leq C(n,A,\vartheta,r).$$
By the Lipschitz approximation theorem it follows that
\begin{equation}
\frac{1}{r^{n-1}}\int_{\D(x_0',r/2)}|\nabla f|^2\,dz\leq C_{L}\,\e_{E}(x_0,r)\leq C_{L}\,\varepsilon_{\textnormal{reg}},
\end{equation}
which implies \eqref{nholderf}.
%
\end{proof}
\begin{remark} \label{remmin}
A minimizer $E_Q$ of the problem (\ref{e:problem}) satisfies the hypothesis of Theorems \ref{C2 reg} and \ref{Cinfinity cap} whenever $Q>0$ is small enough. 
Indeed, assume $x_0\in\partial B_1$.  
Then, by the regularity of $\partial B_1$, there exists a radius $r=r(n)>0$ such that
\begin{equation}
r+\textbf{e}_{B_1}(x_0,2r) \leq \frac{\varepsilon_{\text{reg}}}{2},
\end{equation} 
where $\varepsilon_{{\text{reg}}}$ is as in Theorem \ref{Cinfinity cap}. 
On the other hand, by Proposition \ref{Linfinity closed} we have that $E_Q$ converges to $B_1$ in the Kuratowski sense when $Q\to0$.
Hence, by properties of the excess function, $\textbf{e}_{E_Q}(x_0,2r)\to\textbf{e}_{B_1}(x_0,2r) $ when $Q\to 0$. By Theorem \ref{properties} (iii) we also have  $Q^2\,D_{E_Q}(x_0,2r)\to 0$ when $Q\to 0$.
Therefore,
\begin{equation}
r+\textbf{e}_{E_Q}(x_0,2r)+Q^2\,D_{E_Q}(x_0,2r)\leq \varepsilon_{\text{reg}},
\end{equation} 
when $Q>0$ is small enough.
\end{remark}
%


\section{$C^\infty$ regularity} \label{smooth}
In this section, by a bootstrap argument, we obtain the $C^\infty$ partial regularity of minimizers. Since this result is not necessary for the proof of the main theorem, the reader may skip it unless interested.

Improving the regularity from $C^{2,\eta}$ to $C^\infty$ is easier
than from $C^{1,\eta}$ to $C^{2,\eta}$, because we can straighten the boundary in a nice way once it is $C^2$. More precisely, we have the following lemma.
\begin{lemma} \label{diffeo}
Let $k\in\N$, $k\geq 2$ and $f$ is $C^{k,\vartheta}(\D)$. There exists $\varepsilon>0$ such that if 
\[ \|f\|_{C^{2,\vartheta}(\D)}\leq \varepsilon \quad \text{and} \quad f(0)=0,\] 
then there exists a diffeomorphism $\Phi\in C^{k-1,\vartheta}$, \(\Phi: \C_{1-\varepsilon} \to \C_{1-\varepsilon},\)
such that
\[
\Phi(\Gamma_f\cap \C_{1-\varepsilon})=\{x=(x',x_n)\in\D_{1-\varepsilon}\times\R\,:\,x_n=0\},
\]
where $\Gamma_f$ is the graph of $f$. Moreover,
\begin{equation} \label{pdiffeo}
\left(\nabla \Phi (\Phi^{-1}(x))\,(\nabla\Phi(\Phi^{-1}(x)))^T\right)_{jn}=0 \;\;\forall j\neq n \quad \text{and} \quad \left(\nabla \Phi (\Phi^{-1}(x))\,(\nabla\Phi(\Phi^{-1}(x)))^T\right)_{nn}\neq0.
\end{equation}
\end{lemma}
\begin{proof} Define
\[
\Psi(x',x_n):=(x', f(x'))+x_n \frac{(-\nabla f(x'),1)}{\sqrt{1+|\nabla f(x')|^2}} \quad \forall x=(x',x_n)\in\C_{1-\varepsilon},
\]
then $\Phi:=\Psi^{-1}$ is the desired diffeomorphism.
\end{proof}
\begin{lemma} \label{lemmax2}
Let k be a positive integer and let \(f\) be a \(C^{k+1,\vartheta}\)-H\"older continuous function defined on \(\D(x_0,r) \) such that \( \|f\|_{C^{k+1,\vartheta}}\leq \varepsilon \) for some \(\varepsilon>0\) and
$$E\cap \C(x_0,r)=\left\{ x=(x',x_n)\in \D(x_{0}',r)\times\R\,:\,x_{n}<f(x') \right\} \cap \C(x_0,r).$$ 
Suppose \(v\) is a solution of
\[
-\Div(a_E\,\nabla v)=h \quad \text{in }\mathcal{D}'\left(B_r(x_0) \right), \quad a_E:=\1_{E^c}+\beta\, \1_E,
\]
with  $h^+$ and  $h^-$ $C^{k,\eta}$-H\"older continuous respectively on $\overline{E}\cap \C(x_0,r)$ and $\overline{E^c}\cap \C(x_0,r)$,
where $h^+=h\,1_E$, $h^-=h\,1_{E^c}$.
Then
$v^+, v^-$ are $C^{k+1,\eta}$-H\"older continuous respectively on 
$\overline{E}\cap \C(x_0,r)$ and $\overline{E^c}\cap \C(x_0,r)$.

Moreover, 
\begin{equation}  
\Vert v_1 \Vert_{C^{k+1,\eta}(\overline{E}^c\cap \C(x_0,r))}\leq C \;\text{ and }\; \Vert v_{\beta} \Vert_{C^{k+1,\eta}(\overline{E}\cap \C(x_0,r))}\leq C
\end{equation}
for some constant \(C\geq 0\)  which depends on the \(C^{k,\eta}\)- H\"older norms of  \(h^+\) and \( h^- \) and on the \(C^{k+1,\vartheta}\) norm of \(f\). 
\end{lemma}
\begin{proof}  Assume $x_0=0$. Let $H:=\left\{x\in\R^{n}\,:\,x_n=x\cdot e_{n}\leq0\right\}$ be the half space in $\R^n$. 
By Lemma \ref{diffeo}, we can assume that
\[
\Gamma_f \cap \C_{r}=\partial H\cap \C_{r},
\]
where $\Gamma_f \cap \C_{r/2}:=\{(x',f(x'))\,:\,x'\in \D_{r}\}$, $f(0)=0$ and that $v$ solves the following equation
\begin{equation} \label{eqdiu0}
-\Div (a_H A\, \nabla v)=h,
\end{equation}
where by \eqref{pdiffeo}, $A$ is a $C^{k-1,\vartheta}$-continuous elliptic matrix such that $A_{jn}=0$ for every $j\neq n$, $A_{nn}\neq 0$.

We continue the proof by induction on $k$. For clarity, we do the detailed computations for the case $k=1$ and we explain how the formulas look like for bigger $k$.

{\bf Case $\mathbf{k=1}$.} By taking the derivatives with respect to the tangential coordinates $j\neq n$ of \eqref{eqdiu0} we deduce
\begin{equation} \label{ae30}
\begin{split}
-{\rm div} \big(a_H A\,\nabla{\partial_{j}v}\big) &=\partial_{j}h+{\rm div} \big(\partial_j (a_H A)\,\nabla{v}\big)\\
&=\Div(h\, e_j+\partial_j (a_H A)\,\nabla{v})  \qquad\text{in}\;\mathcal{D}'(\mathbb{R}^{n}).
\end{split}
\end{equation}
Notice that $a_H$  is constant along tangential directions and that $(a_H\,A)^+$, $(a_H\,A)^-$  have coefficients respectively in $C^{0,\eta}(\overline{H}^c\cap \C_{r})$ and $C^{0,\eta}(\overline{H}\cap C_{r})$.
Furthermore,
\[
( h\,e_j+\partial_j (a_H A)\,\nabla{v})^+ \in C^{0,\eta}(\overline{H}^c\cap \C_{r})  \,  \text{ and } \, ( h\,e_j+\partial_j (a_H A)\,\nabla{v})^-\in C^{0,\eta}(\overline{H}\cap \C_{r}). 
\]
Hence, exploiting Lemma \ref{tiposchauder} we deduce 
\begin{equation} \label{moltaregolarita0}
\partial_j v^+ \in  C^{1,\eta}(\overline{H}\cap \C_{r})   \quad \text{and} \quad \partial_j v^-\in C^{1,\eta}(\overline{H^c}\cap \C_{r}) \qquad \forall j\neq n.
\end{equation}
Furthermore, by \eqref{eqdiu0} we have
\begin{equation*}
-\sum_{i,j=1}^{n}\{a_H\,A_{ij}\partial_{ij}v+\partial_{i}(a_H\,A_{ij})\partial_{j}v\}=h.
\end{equation*}
Thanks to the form of the matrix $A$ we obtain
\begin{equation} \label{Ann0}
-a_H\,A_{nn}\partial_{nn}v=\sum_{i,j\neq n}\{a_H\,A_{ij}\partial_{ij}v+\partial_{i}(a_H\,A_{ij})\,\partial_{j}v\}+h.
\end{equation}
Since the right hand side of the previous equation is H\"older continuous, we have
\[
\partial_{nn}v^+ \in  C^{0,\eta}(\overline{H}^c \cap \C_{r})  \quad \text{and} \quad  \partial_{nn}v^-\in C^{0,\eta}(\overline{H}\cap \C_{r}).
\]
Moreover, \eqref{moltaregolarita0} implies 
\[
\partial_{nj}v^+ \in  C^{0,\eta}(\overline{H}^c \cap \C_{r})  \quad \text{and} \quad  \partial_{nj} v^-\in C^{0,\eta}(\overline{H}\cap \C_{r})
\]
for every $j\neq n$.
Therefore,
\[
v^+ \in C^{2,\eta}(\overline{H}^c\cap \C_{r}) \quad \text{and} \quad v^- \in C^{2,\eta}(\overline{H}^c\cap \C_{r}).
\]
By 
Lemma \ref{tiposchauder} we deduce also that
\[
\Vert \nabla v^+\Vert_{C^{1,\eta}(\overline{H}\cap \C_{r})}\, \quad \text{ and } \quad \Vert \nabla v^-\Vert_{C^{1,\eta}(\overline{H^c}\cap \C_{r})}
\]
are bounded by a constant which depends on the H\"older norms of  \(\nabla h^+\), \(\nabla h^- \), the coefficients of \((a_H A)^+\) and \((a_H A)^-\). 

{\bf General $\mathbf{k}$}. 
As in the case $k=1$, we start by taking the derivatives of \eqref{eqdiu0} with respect to the tangential coordinates $j\neq n$.  We get an equation similar to \eqref{ae30}:
\begin{equation*}
-{\rm div} \big(a_H A\,\nabla{\partial_{i_1,i_2,\dots, i_k}v}\big) 
=\Div\left(\partial_{i_2,\dots, i_k}h\,e_{i_1}+\sum\partial_{j_1,j_2,\dots, j_l}(a_H A)\nabla\left(\partial_{\{i_1,i_2,\dots, i_k\}\backslash\{{j_1,j_2,\dots, j_l}\}}v\right)\right)
\end{equation*}
 in $\mathcal{D}'(\mathbb{R}^{n})$.
This gives us
\begin{equation}\label{e:tang regularity}
\partial_{i_1,i_2,\dots, i_k} v^+ \in  C^{1,\eta}(\overline{H}\cap \C_{r})   \quad \text{and} \quad \partial_{i_1,i_2,\dots, i_k} v^-\in C^{1,\eta}(\overline{H^c}\cap \C_{r})
\end{equation}
for all $i_1\neq n$, $i_2\neq n$, \dots, $i_k\neq n$.

By \eqref{e:tang regularity} 
\begin{equation*}
\partial_{i_1,i_2,\dots, i_k, n} v^+ \in  C^{0,\eta}(\overline{H}\cap \C_{r})   \quad \text{and} \quad \partial_{i_1,i_2,\dots, i_k, n} v^-\in C^{0,\eta}(\overline{H^c}\cap \C_{r})
\end{equation*}
for all $i_1\neq n$, $i_2\neq n$, \dots, $i_k\neq n$, and thus,
taking derivatives of \eqref{Ann0} in tangential directions,
we get
\begin{equation*}
\partial_{i_1,i_2,\dots, i_{k-1}, n, n} v^+ \in  C^{0,\eta}(\overline{H}\cap \C_{r})   \quad \text{and} \quad \partial_{i_1,i_2,\dots, i_{k-1}, n, n} v^-\in C^{0,\eta}(\overline{H^c}\cap \C_{r}).
\end{equation*}
Induction on the number of normal directions yields
\[
v^+ \in C^{k+1,\eta}(\overline{H}^c\cap \C_{r}) \quad \text{and} \quad v^- \in C^{k+1,\eta}(\overline{H}^c\cap \C_{r}).
\]

\end{proof}

\begin{theorem} [$C^{\infty}$-regularity] \label{Cinfinity cap}
Given \(n\ge 3\) and \(A>0\), there exists \(\varepsilon_{\textnormal{reg}}=\varepsilon_{\textnormal{reg}}(n, A)>0\)  such that if \(E\) is minimizer of\(~\eqref{e:problem}\) with \(Q+\beta+K+\frac{1}{K}\le A\), \(x_0\in\partial E\), and
\[
r+\e_E(x_0, r)+Q^2\,D_E(x_0, r)\le \varepsilon_{\textnormal{reg}},
\] 
then \(E\cap \C(x_0,r/2)\)  coincides with  the  epi-graph of a \(C^{\infty}\)-function $f$. In particular, we have that \(\partial E\cap \C(x_0,r/2)\) is a \(C^{\infty}\) \((n-1)\)-dimensional manifold. Moreover, for every $\vartheta\in(0,\frac12)$ there exists a constant $C(n,A,k,r,\vartheta)>0$ such that
\begin{equation} \label{eq infinity hn}
[f]_{C^{k,\vartheta}(\D(x_0',r/2))}\leq C(n,A,k,r,\vartheta)
\end{equation}
for every $k\in\mathbb{N}$.
\end{theorem}
\begin{proof}
If we choose $\varepsilon_{\textnormal{reg}}$ as in Theorem \ref{C2 reg}, then there exists $f\in C^{2,\vartheta}(\D(x_0',r/2))$ such that
$$E\cap \C(x_0,r/2)=\left\{ x=(x',x_n)\in \D(x_{0}',r/2)\times\R\,:\,x_{n}<f(x') \right\}.$$
By Corollary \ref{cor:EL on the boundary} we have
\begin{equation}\label{eqboots}
-{\rm div}\left(\nabla M(\nabla f)\,\nabla\partial_i f\right)=\partial_i G \quad \text{a.e. on }\partial E\cap \C(x_0,r/2)
\end{equation}
for every $i=1,\dots,n$, with $\nabla M(\nabla f)$ uniformly elliptic
and H\"older continuous and $G$ - H\"older continuous.

Now we argue by induction on $k$. The induction step is divided into two parts:

\textbf{Claim 1:} 
\[
\begin{split}
f\text{ is }C^k\text{-H\"older continuous}\,\Longrightarrow u^+, u^-\text{ are }& C^k \text{-H\"older continuous respectively on } \\
& \overline{E}\cap \C(x_0,r/2) \text{ and } \overline{E^c}\cap \C(x_0,r/2).
\end{split}
\]
Moreover, there exists a universal constant $C=C(n,A)> 0$ and  $\eta\in(0,\frac12)$ such that 
\begin{equation}  \label{stimeuniformiQu}
\Vert Q\,u^+ \Vert_{C^{k,\eta}(\overline{E}\cap \C(x_0,r/2))}\leq C \;\text{ and }\; \Vert Q\,u^- \Vert_{C^{k,\eta}(\overline{E^c}\cap \C(x_0,r/2))}\leq C.
\end{equation}

\textbf{Claim 2:}
\[
f\text{ is }C^k\text{-H\"older continuous}\,\Longrightarrow f\text{ is }C^{k+1}\text{-H\"older continuous}.
\]

To proof Claim 1, we apply Lemma \eqref{lemmax2} to $v=Qu$ and $h=Q\rho$. By \eqref{darichiamare} the norms 
\[
\Vert Q\,\nabla u^+\Vert_{C^{0,\eta}(\overline{H}\cap \C_{r/2})}\, \quad \text{ and } \quad \Vert Q\,\nabla u^-\Vert_{C^{0,\eta}(\overline{H^c}\cap \C_{r/2})}
\]
and bounded by a universal constant. That gives us \eqref{stimeuniformiQu}.

As for Claim 2, notice that by the definition of $M$, 
since $f$ is $C^k$-H\"older continuous, we have that $\nabla M(\nabla f)$ in \eqref{eqboots} is $C^{k-1}$-H\"older continuous. 
By Claim 1 we deduce that $G$  is $C^{k-1}$-H\"older continuous with its norm uniformly bounded. 
Then, using Schauder estimates for \eqref{eqboots}, we get that $f$ is $C^{k+1}$-H\"older continuous.

\end{proof}
%
%
%
%
%

%
%


%
%
%
%
%
%
%
%
%
%
%
\section{Reduction to nearly spherical sets} \label{nearly spherical}
In this section, by combining Proposition \ref{Linfinity closed} with the higher regularity (Theorem \ref{C2 reg}), we prove that for small enough values of the total charge the minimizers are \emph{nearly-spherical sets}. 
Recall the following definition.
\begin{definition}[\(C^{2,\gamma}\)-nearly spherical set] \label{d:ns}
		  An open bounded set $\Omega\subset\mathbb{R}^n$ is called \emph{nearly-spherical}
			of class $C^{2,\gamma}$ parametrized by $\varphi$, if there exists
			$\varphi\in C^{2,\gamma}$ with $\Vert\varphi\Vert_{L^\infty}<\frac{1}{2}$ such that
			\begin{equation*}
				\partial\Omega=\{(1+\varphi(x))x: x\in\partial B_1\}.
			\end{equation*}	
		\end{definition}
\begin{theorem} \label{smooth closed}
Let $\{Q_h\}_{h\in\mathbb{N}}$ be a sequence such that $Q_h>0$ and $Q_h\to 0$ when $h\to\infty$. Let $\{E_h\}_{h\in\mathbb{N}}$ be a sequence of minimizers of (\(\mathcal{P}_{\beta,K,Q_h,R}\)). Then for $h$ big enough $E_h$ is nearly spherical of class $C^\infty$, i.e. there exists $\varphi_{h}\in C^{\infty}$ with uniform bounds and $\Vert\varphi_{h}\Vert_{L^\infty}<\frac{1}{2}$ such that
			\begin{equation*}
				\partial E_{h}=\{(1+\varphi_{h}(x))x: x\in\partial B_1\}.
			\end{equation*}	
Moreover, $\Vert\varphi_{h}\Vert_{ C^{k}}\to 0$ when $h\to\infty$, for every $k\in\mathbb{N}$.
\end{theorem}
\begin{proof} Fix a point $\bar{x}\in\partial B_1$. By Remark \ref{remmin} there exists $\bar{r}>0$ and a smooth function $g$ such that
\begin{equation} \label{gr1}
\partial B_1\cap \textbf{C}\left(\bar{x},r,\nu_{B_1}(\bar{x})\right)=\partial \left\{x\in\R^n\,:\,\textbf{q}^{\nu_{B_1}(\bar{x})}(x-\bar{x})<g(\textbf{p}^{\nu_{B_1}(\bar{x})}(x-\bar{x}))\right\}\cap \textbf{C}\left(\bar{x},r,\nu_{B_1}(\bar{x})\right)
\end{equation}
for every $0<r\leq\bar{r}$. Furthermore, there exist $r_0\leq\bar{r}$ small enough and $f_h\in C^{\infty}\left(\textbf{D}\left(\bar{x},r,\nu_{B_1}(\bar{x})\right)\right)$ such that 
\begin{equation} \label{gr2}
\partial E_h\cap \textbf{C}\left(\bar{x},r,\nu_{B_1}(\bar{x})\right)=\partial \left\{x\in\R^n\,:\,\textbf{q}^{\nu_{B_1}(\bar{x})}(x-\bar{x})<f_h(\textbf{p}^{\nu_{B_1}(\bar{x})}(x-\bar{x}))\right\}\cap \textbf{C}\left(\bar{x},r,\nu_{B_1}(\bar{x})\right)
\end{equation}
for every $h$ big enough and $r\leq r_0$.
Define $\varphi^{\bar{x}}_h(x):=f_h(g^{-1}(x))$ for every $x\in\partial B_1$.
Then $\{\varphi^{\bar{x}}_{h}\}_{h\in\mathbb{N}}$ is a family of $C^{\infty}$ functions with $\Vert\varphi^{\bar{x}}_{h}\Vert_{C^k}$ uniformly bounded (by Theorem \ref{smooth reg apx}) such that 
\begin{equation*}
				\partial E_h\cap \textbf{C}\left(\bar{x},r,\nu_{B_1}(\bar{x})\right)=\{(1+\varphi_{h}^{\bar{x}}(x))x: x\in\partial B_1\}.
			\end{equation*}	
Hence, by a covering argument we obtain a family $\{\varphi_{h}\}_{h\in\mathbb{N}}$ of $C^{\infty}$ functions with $\Vert\varphi_{h}\Vert_{C^k}$ uniformly bounded such that 
\begin{equation*}
				\partial E_h=\{(1+\varphi_{h}(x))x: x\in\partial B_1\}.
			\end{equation*}	
By Ascoli-Arzel$\grave{{\rm a}}$ and the convergence of $\partial E_{h}$ to $\partial B_{1}$ in the sense of Kuratowski we obtain that $\varphi_{h}\to 0$ in $C^{k-1}(\partial B_{1})$ for every $k\in\mathbb{N}$.
\end{proof}
	\section{Theorem \ref{thm:main} for nearly spherical sets} \label{s:thm for nearly spherical}

	To prove Theorem \ref{thm:main} for nearly spherical sets
	we are going to write Taylor expansion for the energy.
	We only need to deal with the repulsive term $\G$, as the expansion for perimeter is well-known.	
	To this end, we need to compute shape derivatives of the energy
	$\G$ near the ball and get a bound on the second derivative. For the convenience of the reader
	we make these calculations later in Section \ref{fuglede} as they are rather technical.
	
	In this section, we first replace our problem with an equivalent one
	and write Euler-Lagrange equations for it. We do it to facilitate
	the computations of Section \ref{fuglede}.
	Then we conclude the proof of Theorem \ref{thm:main} for
	nearly spherical sets given Taylor expansion. 
	Thanks to the quantitative isoperimetric inequality
	for nearly-spherical sets,
	we see that we can be crude in the bounds of Section \ref{fuglede} as we have a small parameter in front of the 
	disaggregating term.
	
	\subsection{Changing minimization problem}

	For a fixed domain $E$ we are solving the following minimization problem.
	\begin{equation*}
		\G(E)=\inf_{\substack{u\in H^1(\mathbb{R}^n) \\ \rho\1_{E^c}=0}}\left\{\frac{1}{2}\int_{\mathbb{R}^n}\left(a_E\vert\nabla u\vert^2+K\rho^2\right)dx: -\div(a_E\nabla u)=\rho, \int_{\mathbb{R}^n}\rho dx=1\right\}.
	\end{equation*}
	
	We want to get rid of the constraints and make it a minimization problem over single
functions rather than over pairs. More precisely, we prove the following lemma.

	\begin{lemma}\label{l:chngprb}
	For any $E\subset\mathbb{R}^n$ the energy $\G$ can be represented in the following way:
	\begin{equation}
		\G(E)=\frac{K}{2\vert E\vert}-\inf_{\psi\in H^1(\mathbb{R}^n)}\left(\frac{1}{2}\int_{\mathbb{R}^n}{a_E\vert\nabla \psi\vert^2}dx+\frac{1}{\vert E\vert}\int_{E}\psi dx-\frac{1}{2\vert E\vert K}\left(\int_{E}\psi dx\right)^2+\frac{1}{2K}\int_{E}\psi^2 dx\right).
	\end{equation}
	
	\end{lemma}
	\begin{proof}
	We use an ``infinite dimension Lagrange multiplier":
	\begin{equation*}
	\begin{aligned}
		\G(E)&=\inf_{\substack{u\in H^1(\mathbb{R}^n) \\ \rho\1_{E^c}=0}}\Big\{\frac{1}{2}\int_{\mathbb{R}^n}a_E\vert\nabla u\vert^2 dx+\frac{1}{2}\int_{E}K\rho^2 dx\\
		&+\sup_{\psi\in H^1(\mathbb{R}^n)}\left[\int_{\mathbb{R}^n}(a_E\nabla u\cdot\nabla\psi-\rho\psi)dx\right]: \int_{E}\rho dx=1\Big\}\\
		&=\inf_{\substack{u\in H^1(\mathbb{R}^n) \\ \rho\1_{E^c}=0}}\sup_{\psi\in H^1(\mathbb{R}^n)}\left\{\frac{1}{2}\int_{\mathbb{R}^n}a_E(\vert\nabla u\vert^2+2\nabla u\cdot\nabla\psi) dx+\frac{1}{2}\int_{E}(K\rho^2-2\rho\psi)dx: \int_{E}\rho dx=1\right\}.
	\end{aligned}	
	\end{equation*}
	
	 The convexity of the problem allows us to use Sion minimax theorem (\cite[Corollary 3.3]{Sion}) and interchange the infimum and the supremum:
	\begin{equation*}
	\begin{aligned}
		\G(E)&=\sup_{\psi\in H^1(\mathbb{R}^n)}\inf_{\substack{u\in H^1(\mathbb{R}^n) \\ \rho\1_{E^c}=0}}\left\{\frac{1}{2}\int_{\mathbb{R}^n}a_E(\vert\nabla u\vert^2+2\nabla u\cdot\nabla\psi) dx+\frac{1}{2}\int_{E}(K\rho^2-2\rho\psi)dx: \int_{E}\rho dx=1\right\}\\
		&=\sup_{\psi\in H^1(\mathbb{R}^n)}\left\{\inf_{u\in H^1(\mathbb{R}^n)}\frac{1}{2}\int_{\mathbb{R}^n}a_E(\vert\nabla u\vert^2+2\nabla u\cdot\nabla\psi) dx+\inf_{\substack{\\ \rho\1_{E^c}=0 \\ \int_{E}\rho dx=1}}\frac{1}{2}\int_{E}(K\rho^2-2\rho\psi)dx \right\}.
	\end{aligned}	
	\end{equation*}
	
	We denote the infimums inside by I and II, that is
	\begin{equation*}
	\begin{aligned}
		\mathrm{I}&:=\inf_{u\in H^1(\mathbb{R}^n)}\left\{\frac{1}{2}\int_{\mathbb{R}^n}a_E(\vert\nabla u\vert^2+2\nabla u\cdot\nabla\psi)dx\right\};\\
		\mathrm{II}&:=\inf_\rho\left\{\frac{1}{2}\int_{E}(K\rho^2-2\rho\psi)dx: \int_{E}\rho dx=1\right\}.
	\end{aligned}	
	\end{equation*}
	 
	We want to compute both I and II in terms of $\psi$.
	
	For I it is immediate. Since $a_E$ is positive we get that
	\begin{equation*}
	\begin{aligned}
		\mathrm{I}&=\inf_{u\in H^1(\mathbb{R}^n)}\left\{\frac{1}{2}\int_{\mathbb{R}^n}a_E(\vert\nabla u\vert^2+2\nabla u\cdot\nabla\psi)dx\right\}\\
		&=\inf_{u\in H^1(\mathbb{R}^n)}\left\{\frac{1}{2}\int_{\mathbb{R}^n}a_E\left(\vert\nabla u+\nabla\psi\vert^2-\vert\nabla \psi\vert^2\right)dx\right\}\\
		&=-\frac{1}{2}\int_{\mathbb{R}^n}{a_E\vert\nabla \psi\vert^2}dx.
	\end{aligned}	
	\end{equation*}
	
	We note that the corresponding minimizing $u$ equals to $-\psi$.
	 

	To compute II, note that
	\begin{equation*}
	\begin{aligned}
		\mathrm{II}&=\inf_\rho\left\{\frac{1}{2}\int_{E}(K\rho^2-2\rho\psi)dx: \int_{E}\rho dx=1\right\}\\
		&=\inf_\rho\left\{\frac{1}{2}\int_{E}\left(\sqrt{K}\rho-\frac{\psi}{\sqrt{K}}\right)^2dx: \int_{E}\rho dx=1\right\}-\frac{1}{2K}\int_{E}\psi^2 dx\\
		&=\frac{K}{2}\inf_f\left\{\int_{E}\left(f-\left(\frac{\psi}{K}-\frac{1}{\vert E\vert}\right)\right)^2dx: \int_{E}f dx=0\right\}-\frac{1}{2K}\int_{E}\psi^2 dx.
	\end{aligned}	
	\end{equation*}	
	Then the minimizing function $f^*$ is the projection in $L^2(E)$ of a function
	$\left(\frac{\psi}{K}-\frac{1}{\vert E\vert}\right)$ onto the linear
	space $\left\{f: \int_{E}f dx=0\right\}$. Thus, $f^*=\left(\frac{\psi}{K}-\frac{1}{\vert E\vert}\right)-c$, where $c$ is the constant such that
	$\int_E f^*=0$, i.e. $c=\frac{\int_E\left(\frac{\psi}{K}-\frac{1}{\vert E\vert}\right)}{\vert E\vert}$. 	
	The corresponding minimizing $\rho$ equals to $\1_E\frac{1}{K}\left(\psi+\frac{\left(1-\frac{1}{K}\int_{E}\psi dx\right)K}{{\vert E\vert}}\right)$.
	
	Bringing it all together,
	\begin{equation}\label{e:newminpb}
	\begin{aligned}
		\G(E)&=\frac{K}{2\vert E\vert}+\sup_{\psi\in H^1(\mathbb{R}^n)}\left(-\frac{1}{2}\int_{\mathbb{R}^n}{a_E\vert\nabla \psi\vert^2}dx-\frac{1}{\vert E\vert}\int_{E}\psi dx+\frac{1}{2\vert E\vert K}\left(\int_{E}\psi dx\right)^2-\frac{1}{2K}\int_{E}\psi^2 dx\right)\\
		&=\frac{K}{2\vert E\vert}-\inf_{\psi\in H^1(\mathbb{R}^n)}\left(\frac{1}{2}\int_{\mathbb{R}^n}{a_E\vert\nabla \psi\vert^2}dx+\frac{1}{\vert E\vert}\int_{E}\psi dx-\frac{1}{2\vert E\vert K}\left(\int_{E}\psi dx\right)^2+\frac{1}{2K}\int_{E}\psi^2 dx\right).
	\end{aligned}	
	\end{equation}
	\end{proof}

	\subsection{Euler-Lagrange}
	
	We now consider the following minimization problem:
	\begin{equation}\label{e:new min problem} 
		\J(E)=\inf_{\psi\in H^1(\mathbb{R}^N)}\left(\frac{1}{2}\int_{\mathbb{R}^n}{a_E\vert\nabla \psi\vert^2}dx+\frac{1}{\vert E\vert}\int_{E}\psi dx-\frac{1}{2\vert E\vert K}\left(\int_{E}\psi dx\right)^2+\frac{1}{2K}\int_{E}\psi^2 dx\right).
	\end{equation}
	\begin{remark}
	Note that $\J(E)\leq 0$. By Lemma \ref{l:chngprb}
	$$\G(E)=\frac{K}{2\vert E\vert}-\J(E).$$
	By the inequality (2.1) in \cite{DPHV},
	$\G(E)\leq C(n,K,\beta,\vert E\vert)$.
	This implies that
	\begin{equation}\label{e:bound on energy}
		\vert\J(E)\vert\leq C(n,K,\beta,\vert E\vert).
	\end{equation}
		
	\end{remark}
	
	A minimizer for this problem exists, and it is
	unique by convexity.
	Indeed, to see coercivity of the functional note that
	\begin{equation*}
	-\frac{1}{2\vert E\vert K}\left(\int_{E}\psi dx\right)^2+\frac{1}{2K}\int_{E}\psi^2 dx\geq 0
	\end{equation*}
	by Jensen inequality. As for convexity, we use that
	\begin{equation*}
	-\frac{1}{2\vert E\vert K}\left(\int_{E}\psi dx\right)^2+\frac{1}{2K}\int_{E}\psi^2 dx
	=\frac{1}{2K}\int_{E}{\left(\psi-\fint_{E}\psi\right)^2}.
	\end{equation*}	
	Note that the minimizers in the definitions of $\J$ and $\G$ coincide since the set is fixed. 
	We denote the minimizer by $\psi_E$. We would also need the interior and
	exterior restrictions of the function $\psi_E$, i.e.
	$$\psi_E^+:=\psi_E|_{E},\quad\psi_E^{-}:=\psi_E|_{E^c}.$$ 
	
	\begin{proposition}
	The following identities hold for $\psi_E$:
	\begin{enumerate}[(i)]
	\item (Euler-Lagrange equation, integral form)
	for any $\Psi\in D^1(\mathbb{R}^n)$
	\begin{equation} \label{e:el}
	\begin{aligned}
		&\int_{\mathbb{R}^n}a_E\nabla\psi_E\cdot\nabla\Psi dx +\frac{1}{K}\int_E\psi_E\Psi dx+\frac{1}{\vert E\vert}\left(\int_E{\Psi}dx\right)\left(1-\frac{1}{K}\int_E\psi_E dx\right)\\
		&=\int_{\mathbb{R}^n}\Psi \left(\frac{\1_E\psi_E}{K} - \div(a_E\nabla\psi_E)\right)dx
		+\int_{\partial E}{\left(\beta\nabla\psi^+_E-\nabla\psi^{1}_E\right)\cdot\nu}\Psi\,d\mathcal{H}^{n-1}\\
		&+\frac{1}{\vert E\vert}\left(\int_E{\Psi}dx\right)\left(1-\frac{1}{K}\int_E\psi_E dx\right)=0.
	\end{aligned}	
	\end{equation}
	\item (Euler-Lagrange equation)
		\begin{equation} \label{e:el diff form}
		\begin{cases}
			-\beta\Delta\psi_E=-\frac{1}{K}\psi_E+\frac{2}{K}\J(E)-\frac{1}{\vert E\vert}\text{ in }E,\\
			\Delta\psi_E=0\text{ in }E^c,\\
			\psi_E^+=\psi_E^{-}\text{ on }\partial E,\\
			\beta\nabla\psi^{+}_E\cdot\nu=\nabla\psi^{-}_E\cdot\nu\text{ on }\partial E.
		\end{cases}
		\end{equation}
	\item 	
		\begin{equation} \label{otherformenergy}
			\J(E)=\frac{1}{2\vert E\vert}\int_E{\psi_E}dx.
		\end{equation}	
	\item There exists a constant 
	$C=C(n,K,\beta,\vert E\vert)$ such that
		\begin{equation} \label{e:bound in D^1}
			\int_{\mathbb{R}^n}{a_E\vert\nabla\psi_E\vert^2}dx\leq C.		
		\end{equation}
	\end{enumerate}
	\end{proposition}
	\begin{proof}
		To prove (\ref{otherformenergy}) we use $\psi_E$ as a test function
		in \eqref{e:el}.
		
		To see (\ref{e:bound in D^1}), we use $\psi_E$ as a test function in \eqref{e:el} and Cauchy-Schwarz inequality to get
		\begin{equation*}
			\int_{\mathbb{R}^n}{a_E\vert\nabla\psi_E\vert^2}dx\leq
			-\frac{1}{\vert E\vert}\left(\int_E\psi_Edx\right).
		\end{equation*}
		Now we apply (\ref{otherformenergy}) and (\ref{e:bound on energy}) to obtain
		\begin{equation*}
			\int_{\mathbb{R}^n}{a_E\vert\nabla\psi_E\vert^2}dx\leq
			-2\J(E)\leq 2C(n,K,\beta,\vert E\vert).
		\end{equation*}		
	\end{proof}	
	
	\begin{proposition}\label{p:radiality of min fct}
	Let $\psi_0$ be the minimizer for $\J(B_1)$.
	Then $\psi_0$ is radial. 
	\end{proposition}
	\begin{proof}
	Let $R:\mathbb{R}^n\rightarrow\mathbb{R}^n$ be any rotation. Since $R(B_1)=B_1$, $\psi_0\circ R$ is also a minimizer for $\J(B_1)$.
	But the minimizer is unique, so we got that $\psi_0\circ R=\psi_0$
	for any rotation $R$. This implies that $\psi_0$ is radial.
	
	\end{proof}
		
	\subsection{Proof of Theorem \ref{thm:main}}
	
		We will use the following notation.
		\begin{definition}
			For an open set $\Omega$, $x_\Omega$ denotes the barycenter of $\Omega$, namely $$x_\Omega=\frac{1}{\vert\Omega\vert}\int_{\Omega}{x}dx.$$
		\end{definition}	

		We want to prove that for $Q$ small enough the only minimizer
		of $\F(\Omega)=P(\Omega)+Q^2\G(\Omega)$ 
		for $\Omega$ nearly spherical is a ball.
		
		We will use the following theorem proved by Fuglede. 
		
		\begin{theorem} \label{t: fuglede for perimeter}(\cite[Theorem 1.2]{Fuglede89})
			There exists a constant $c=c(N)$ such that for any 
			$\Omega$ --- nearly spherical set parametrized by $\varphi$
			with $\vert\Omega\vert=\vert B_1\vert$, $x_\Omega=0$,
			the following inequality holds
			\begin{equation*}
				P(\Omega)-P(B_1)\geq 
				c\Vert \varphi\Vert^2_{H^1(\partial B_1)}.
			\end{equation*}
		\end{theorem}		
		

		We will also need the following bound on the energy $\J$, see Section \ref{fuglede} for the proof.
		
		\begin{lemma} \label{taylor}
			Given $\vartheta\in(0,1]$, there exists $\delta=\delta(N,\vartheta)>0$
			and a bounded function $g$ such that for every nearly spherical set $E$
			parametrized by $\varphi$ with $\Vert\varphi\Vert_{C^{2,\vartheta}(\partial B_1)}<\delta$
			and $\vert E\vert=\vert B_1\vert$, we have
			\begin{equation*}
				\J(E)\geq \J(B_1)
				-g(\Vert\varphi\Vert_{C^{2,\vartheta}})\Vert\varphi\Vert^2_{H^1(\partial B_1)}.
			\end{equation*}
		\end{lemma}

	Finally, we are ready to prove the main result of the paper.
	\begin{proof}[Proof of Theorem \ref{thm:main}]
		Argue by contradiction. Suppose there exists a sequence of minimizers
		$E_h$ corresponding to $Q_h\rightarrow 0$ such that $E_h$ are not balls.
		By Theorem \ref{smooth closed} we have that starting from a certain $h$
		the sets (possibly, translated) are nearly-spherical 
		parametrized by $\varphi_h$
		with $\Vert\varphi_h\Vert_{C^{2,\gamma}(\partial B_1)}<\delta$, where $\delta$ is the one of Lemma \ref{taylor}.
		
		To apply Theorem \ref{t: fuglede for perimeter} and Lemma \ref{taylor} we need the sets to have barycenters at the origin. 
		It is not necessarily true for the sequence $E_h$, however,
		we can exploit the fact that nearly-spherical sets have barycenters close to the origin.
		Indeed, suppose that $E$ is the nearly-spherical set parametrized by $\varphi_E$. Then, using that the barycenter of the ball
		$B_1$ is at the origin, we have
		\begin{equation*}
		\begin{split}
			x_E&=\frac{1}{\vert E\vert}\int_{\partial B_1}{x\frac{\left(1+\varphi_E(x)\right)^{n+1}}{n+1}}\,d\mathcal{H}^{n-1}(x)\\
			&=\frac{1}{\vert E\vert}\int_{\partial B_1}{x\frac{\left(1+\varphi_E(x)\right)^{n+1}}{n+1}}\,d\mathcal{H}^{n-1}(x)
			-\frac{1}{\vert E\vert}\int_{\partial B_1}{x}\,d\mathcal{H}^{n-1}(x)\\
			&=\frac{1}{\vert E\vert}\int_{\partial B_1}{x\frac{\sum_{i=1}^{n+1}{N+1 \choose i}\varphi_E(x)^i}{n+1}}\,d\mathcal{H}^{n-1}(x).
		\end{split}		
		\end{equation*}				
		If $\Vert\varphi_E\Vert_{L^\infty(\partial B_1)}<1$, then
		the last computation yields
		$$\vert x_E\vert<C\Vert\varphi_E\Vert_{L^\infty(\partial B_1)}.$$
		We chose the sequence $E_h$ so that $E_h\rightarrow B_1$
		in $L^\infty$, thus, $x_{E_h}\rightarrow 0$. So if we now
		look at the sequence of sets $\tilde{E}_h=\{x-x_{E_h}:x\in E_h\}$, we see that $\tilde{E}_h\rightarrow B_1$ in $L^\infty$
		and $x_{\tilde{E}_h}=0$. It remains to apply Theorem \ref{smooth closed} to the sequence $\{E_h\}$ to see that these new
		translated sets are still nearly-spherical.
		For the sake of simplicity let us not rename the sequence
		and assume that the sequence $\{E_h\}$ is such that
		$x_{E_h}=0$.
		
		Now we can apply Theorem \ref{t: fuglede for perimeter} and Lemma \ref{taylor}.
		We want to show that $\F(E_h)>F(B_1)$ for $h$ big enough. 
		Indeed, if $Q_h$ is small enough, we have
		\begin{equation*}
		\begin{aligned}
			\F(E_h)&=P(E_h)+Q^2_h\,\G(E_h)
			\geq P(B_1)+c\Vert\varphi_h\Vert^2_{H^1(\partial B_1)}
			+Q_h^2\left(\frac{K}{2\vert B_1\vert}-\J(E_h)\right)\\
			&\geq P(B_1)+c\Vert\varphi_h\Vert^2_{H^1(\partial B_1)}
			+Q_h^2\left(\frac{K}{2\vert B_1\vert}-\J(B_1)-c'\Vert\varphi_h\Vert^2_{H^1(\partial B_1)}\right)\\
			&>P(B_1)
			+Q_h^2\left(\frac{K}{2\vert B_1\vert}-\J(B_1)\right)
			=\F(B_1).
		\end{aligned}	
		\end{equation*}
	\end{proof}
	
	We can now prove Corollary \ref{cor:existence in R^n}, which
	follows from Theorem \ref{thm:main} and properties of minimizers
	established in \cite{DPHV}.
	
	\begin{proof}[Proof of Corollary \ref{cor:existence in R^n}]
		Let $Q_0$ be the one of Theorem \ref{thm:main}.  
		Let $E$ be an open set such that $\vert E\vert=\vert B_1\vert$. Let us show that $\F(E)\geq \F(B_1)$. 
		If $E$ is bounded, then $\F(E_h)\geq F(B_1)$ by Theorem \ref{thm:main}. Assume now that $E$ is unbounded.
		
		We can assume that $E$ is of finite perimeter, since otherwise
		$\F(E)=\infty$. Then, by \cite[Remark 13.12]{Maggi12},
		there exists a sequence $R_h\rightarrow\infty$ such that
		$E\cap B_{R_h}\rightarrow E$ in $L^1$, $P(E\cap B_{R_h})\rightarrow P(E)$. Rescale the sets so that their volumes are 
		the same as the one of the ball, i.e. $$\Omega_h=\alpha_h\left(E\cap B_{R_h}\right)\text{ with }\alpha_h=\left(\frac{\vert B_1\vert}{\left\vert E\cap B_{R_h}\right\vert}\right)^{1/n}.$$
		Note that since $\vert E\vert=\vert B_1\vert$, 
		$\alpha_h\rightarrow 1$, so also for $\Omega_h$ we have 
		$\vert\Omega_h\Delta E\vert\rightarrow 0$, 
		$P(\Omega_h)\rightarrow P(E)$.
		Now, by the continuity of the functional $\G$ in $L^1$ (see 
		\cite[Proposition 2.6]{DPHV}), we get 
		\begin{equation}\label{e:convergence to F(E)}
			F(\Omega_h)=P(\Omega_h)+\G(\Omega_h)\rightarrow P(E)+\G(E)=\F(E).
		\end{equation}
		On the other hand, $\Omega_h\subset \alpha_h B_{R_h}$,
		so it is bounded and hence, by Theorem \ref{thm:main},
		$$\F(\Omega_h)\geq \F(B_1)\text{ for every }h.$$
		Combining the last inequality with \eqref{e:convergence to F(E)}, we get $\F(E)\geq \F(B_1)$. Thus, the infimum in the problem 
		(\ref{e:problem in R^n}) is achieved on balls.
		
		Let us show that the only minimizers are the balls.
		Let $E$ be a minimizer for (\ref{e:problem in R^n}).
		If $E$ is bounded, then by Theorem \ref{thm:main} it should be
		a ball of radius $1$. We now explain why $E$ cannot be unbounded. Indeed, suppose the contrary holds. Then there we can find a sequence of points $x_k$ such that $x_k\in E$, $\vert x_k-x_j\vert\geq 1$ for $k\ne j$ (for example, we can define $x_k:=E\backslash B_{\max\{\vert x_1\vert, \vert x_2\vert, \dots, \vert x_{k-1}\vert\}+1}$). Now, by density estimates for minimizers (Theorem \ref{properties} (v)), we have
		\begin{equation}\label{e:density est for minimizers}
			\frac{\vert B_r(x)\cap E\vert}{\vert B_r\vert}\geq \frac{1}{C}\text{ for }x\in E,\,r\in(0,\overline{r}).
		\end{equation}
		Note that even though Theorem \ref{properties} (v) deals with
		minimizers of (\ref{e:problem}), the constants $C$ and $\overline{r}$ do not depend on $R$, so it applies in our case.
		It remains to use \eqref{e:density est for minimizers}
		for $x=x_k$ and $r=\min(1/2\overline{r},1/2)$ to see that
		$$\vert E\vert\geq \sum_{k=1}^\infty{\vert B_r(x_k)\cap E\vert}\geq\sum_{k=1}^\infty\frac{\vert B_r\vert}{C}=\infty,$$
		which contradicts the fact that $\vert E\vert=\vert B_1\vert$.
		Thus, $E$ is bounded and it is a ball of radius $1$.
	\end{proof}

\section{Proof of Lemma \ref{taylor}} \label{fuglede}

	We will need the following technical lemma, which is almost identical to \cite[Lemma A.1]{fkstab}.		
	Since we need a slightly different conclusion than in \cite{fkstab}, we repeat the proof here.
	Throughout this section we will be using the following notation.
	
	\begin{notation}
	We denote by $J_{\Phi_t}(x)$ the jacobian of $\Phi_t$ at $x$: 
	$$J_{\Phi_t}(x)=\det\nabla\Phi_t(x).$$
	\end{notation} 	
	
	\begin{lemma}\label{lm:vectorfield}
		Given $\vartheta\in(0,1]$ there exists $\delta=\delta(n,\gamma)>0$,
		a modulus of continuity $\omega$,
		and a bounded function $g$ such that for every nearly spherical set $E$
		parametrized by $\varphi$ with $\Vert\varphi\Vert_{C^{2,\vartheta}(\partial B_1)}<\delta$
		and $\vert\Omega\vert=\vert B_1\vert$, we can find an autonomous vector field $X_\varphi$
		for which the following holds true:
		\begin{enumerate}[(i)]
			\item $\div{X_\varphi}=0$ in a $\delta$-neighborhood of $\partial B_1$;
			\item if $\Phi_t:=\Phi(t,x)$ is the flow of $X_\varphi$, i.e.
				\begin{equation*}
				  \partial_t\Phi_t=X_\varphi(\Phi_t),\qquad \Phi_0(x)=x,
				\end{equation*}
				then $\Phi_1(\partial B_1)=\partial E$ and $\vert\Phi_t(B_1)\vert=\vert B_1\vert$ for all $t\in[0,1]$;
			\item denote $E_t:=\Phi_t(B_1)$, then

				\begin{equation} \label{e:flow close to id}
				\Vert\Phi_t-Id\Vert_{C^{2,\vartheta}}\leq\omega(\Vert\varphi\Vert_{C^{2,\vartheta}(\partial B_1)})
				 \text{ for every }t\in[0,1],
				\end{equation}
						
%
%
				\begin{equation}\label{e:bound on jacobian}
					\left\vert J_\Phi\right\vert\leq g(\Vert\varphi\Vert_{C^{2,\vartheta}(\partial B_1)})\text{ in a neighborhood of }B_1, 
				\end{equation}
				
				\begin{equation}\label{e:phi to X nu}
					\left\Vert X\cdot\nu\right\Vert_{H^1(\partial E_t)}\leq g(\Vert\varphi\Vert_{C^{2,\vartheta}(\partial B_1)})\left\Vert \varphi\right\Vert_{H^1(\partial B_1)}, 
				\end{equation}
				
				and for the tangential part of $X$, defined as $X=X-(X\cdot\nu)\nu$,
				there holds
				\begin{equation}\label{e:tang part of X}
					\left\vert X^\tau\right\vert\leq \omega(\Vert\varphi\Vert_{C^{2,\vartheta}(\partial B_1)})\left\vert X\cdot\nu\right\vert\text{ on }\partial E_t. 
				\end{equation}

			\end{enumerate}
			
	\end{lemma}	
		
	\begin{proof}
	Such a vector field can be constructed for any smooth set, see for example \cite{Dambrine02}. However, for the ball one can write an explicit expression
	in a neighborhood of $\partial B_1$. The proof for the case of the ball
	can be found in 	\cite[Lemma A.1]{fkstab}. 
	For the convenience of the reader we provide the expression here, as well
	as brief explanation of how to get the needed bounds. 
	In polar coordinates, $\rho=\vert x\vert$, $\theta=x/\vert x\vert$ the field looks like this:
	\begin{equation*}
	\begin{aligned}
		&X_{\varphi}(\rho,\theta)=\frac{(1+\varphi(\theta))^n-1}{n\rho^{n-1}}\theta,\\
		&\Phi_t(\rho,\theta)=\left(\rho^n+t\left(\left(1+\varphi(\theta)\right)^n-1\right)\right)^\frac{1}{N}\theta	
	\end{aligned}
	\end{equation*}
	for $\vert\rho-1\vert\ll 1$. Then we extend this vector field globally in order to satisfy (\ref{e:flow close to id}).
	Notice that \eqref{e:bound on jacobian} is a direct consequence of \eqref{e:flow close to id}.
	
	By direct computation we get
	\begin{equation}\label{e:X nu close to the ball}
		(X\cdot\theta)\circ\Phi_t-X\cdot\nu_{\partial B_1}=(X\cdot\nu_{\partial B_1})f\text{ on }\partial B_1,
	\end{equation}
	with $\Vert f\Vert_{C^{2,\vartheta}(\partial B_1)}\leq\omega\left(\Vert \varphi\Vert_{C^{2,\vartheta}(\partial B_1)}\right)$.
	Now we can get the bound \eqref{e:phi to X nu}. Indeed, \eqref{e:X nu close to the ball} together with \eqref{e:bound on jacobian} gives us
	\begin{equation*}
		\Vert X\cdot\nu\Vert_{H^1(\partial E_t)}\leq g\left(\Vert \varphi\Vert_{C^{2,\vartheta}(\partial B_1)}\right)\Vert X\cdot\nu\Vert_{H^1(\partial B_1)}.
	\end{equation*}		
	From the definition of $X$, on $\partial B_1$ we have
	$$\varphi-X\cdot\nu=\frac{1}{n}\sum_{i=2}^n{{{n}\choose{i}}\varphi^i}$$
	and thus,
	\begin{equation*}
		\Vert \varphi-X\cdot\nu\Vert_{H^1(\partial B_1)}
		\leq\omega\left(\Vert \varphi\Vert_{C^{2,\vartheta}(\partial B_1)}\right)
		\Vert X\cdot\nu\Vert_{H^1(\partial B_1)},
	\end{equation*}
	yielding the inequality \eqref{e:phi to X nu}.
	
	To see \eqref{e:tang part of X} we use that by definition $X$ is parallel
	to $\theta$ close to $\partial B_1$. 
	Thus,
	\begin{equation*}
	\begin{split}
		\left\vert X^\tau\circ\Phi_t\right\vert
		&=\left\vert\left(\left(X\cdot\theta\right)\theta\right)\circ\Phi_t
		-\left(\left(X\cdot\nu\right)\nu\right)\circ\Phi_t\right\vert\\
		&=\big\vert\left(X\cdot\nu_{\partial B_1}\right)\left(1+\omega\left(\Vert \varphi\Vert_{C^{2,\vartheta}(\partial B_1)}\right)\right)\nu_{\partial B_1}\left(1+\omega\left(\Vert \varphi\Vert_{C^{2,\vartheta}(\partial B_1)}\right)\right)\\
		&-\left(X\cdot\nu_{\partial B_1}\right)\left(1+\omega\left(\Vert \varphi\Vert_{C^{2,\vartheta}(\partial B_1)}\right)\right)\nu_{\partial B_1}\left(1+\omega\left(\Vert \varphi\Vert_{C^{2,\vartheta}(\partial B_1)}\right)\right)\big\vert\\
		&=\omega\left(\Vert \varphi\Vert_{C^{2,\vartheta}(\partial B_1)}\right)
		\left\vert \left(X\cdot\nu\right)\circ\Phi_t\right\vert.
	\end{split}	
	\end{equation*}
	\end{proof}	
	
	In what follows we omit the subscript $\varphi$ for brevity.
	

	\subsection{First derivative}	
	We want to compute $\frac{d}{dt}\J(E_t)$.
	
	Let $\psi_t$ be the minimizer in the minimization problem	
	\eqref{e:new min problem} for $E_t$. 
	Recall that by \eqref{e:el diff form}  it means that $\psi_t$
	satisfies
	\begin{equation} \label{e:el diff form on E_t}
		\begin{cases}
			-\beta\Delta\psi_t=-\frac{1}{K}\psi_t+\frac{2}{K}\J(E_t)-\frac{1}{\vert B_1\vert}\text{ in }E_t,\\
			\Delta\psi_t=0\text{ in }E_t^c,\\
			\psi_t^+=\psi_t^{-}\text{ on }\partial E_t,\\
			\beta\nabla\psi^+_{t}\cdot\nu=\nabla\psi^{-}_{t}\cdot\nu\text{ on }\partial E_t.
		\end{cases}
	\end{equation}

	First we notice that $\psi_t$ is regular since it is a solution
	to a transmission problem. More precisely, by Lemma \ref{lemmax2},
	the following holds.
	\begin{proposition}\label{bound on psi_t}
		There exists $\delta>0$ such that if $\Vert\varphi\Vert_{C^{2,\vartheta}(\partial B_1)}<\delta$, then
		$$\Vert \psi_t\Vert_{C^2(\overline{E_t})}\leq g(\Vert\varphi\Vert_{C^{2,\vartheta}(\partial B_1)})\text{ for every }t\in [0,1],$$
		where $g$ is a bounded function.
	\end{proposition}	
		
%
	
	To compute the derivative of $\J(E_t)$ we would like to use Hadamard formula (see \cite[Chapter 5]{HP}).
	For that, we first need to prove the following proposition.
	
	\begin{proposition}\label{p:diff of psi}
		The function $t\mapsto\psi_t$ is differentiable in $t$ and its 
		derivative $\dot{\psi}_t$ satisfies
		\begin{equation}
			\begin{cases}\label{e:EL for phi dot at t}
			-\beta\Delta\dot{\psi}_t=-\frac{1}{K}\dot{\psi}_t+\frac{2}{K}\dot{\J}(E_t)\text{ in }E_t,\\
			\Delta\dot{\psi}_t=0\text{ in }E_t^c,\\
			\dot{\psi}_t^+-\dot{\psi}_t^{-}=
			-\left(\nabla\psi_t^+-\nabla\psi_t^{-}\right)\cdot\nu(X\cdot\nu)\text{ on }\partial E_t,\\
			\beta\nabla\dot{\psi}^+_t\cdot\nu-\nabla\dot{\psi}^{-}_t\cdot\nu=
			-\left(\left(\beta\nabla[\nabla\psi_t^+]-\nabla[\nabla\psi_t^{-}]\right) X\right)\cdot\nu
			\text{ on }\partial E_t.
			\end{cases}
		\end{equation}	
	\end{proposition}
	\begin{proof}
		The proof is standard, see (\cite[Chapter 5]{HP}) for the general strategy and (\cite[Theorem 3.1]{ADK}) for a different kind of a transmission problem. 
		We were unable to find a result covering our particular case in the literature,
		so we provide a proof here.
	
		We first deal with material derivative of the function $\psi$,
		i.e. we shall look at the function $t\mapsto\tilde{\psi}_t:=\psi_t(\Phi_t(x))$. The advantage is that its derivative in time
		is in $H^1$ as we will see. Note that the time derivative of $\psi_t$
		itself is not in $H^1$ as it has a jump on $\partial E_t$.
		
		\textbf{Step 1:} moving everything to a fixed domain.

		We introduce the following notation:
		\begin{equation*}
			A_t(x):=D\Phi_t^{-1}(x)\left(D\Phi_t^{-1}\right)^t(x)J_{\Phi_t}(x).
		\end{equation*}
		Note that $A_t$ is symmetric and positive definite and for $t$
		small enough it is elliptic with a constant independent of $t$.
		
		Now we perform a change of variables in Euler-Lagrange equation
		for $\psi_t$ (\ref{e:el}) to get 
		Euler-Lagrange equation for $\tilde{\psi}_t$:
		\begin{equation} \label{e:el moved to the ball}
			\int_{\mathbb{R}^n}\nabla\Psi \left(a_B A_t\nabla\tilde{\psi}_t\right)dx
			+\frac{1}{K}\int_{B}\Psi \tilde{\psi}_t J_{\Phi_t}(x)dx
			+\frac{1}{\vert B\vert}\left(\int_B{\Psi}J_{\Phi_t}(x)dx\right)\left(1-\frac{1}{K}\int_B\tilde{\psi}_t J_{\Phi_t}(x)dx\right)=0
		\end{equation}
		for any $\Psi\in D^1(\mathbb{R}^n)$.

		\textbf{Step 2:} convergence of the material derivative.
		
		Let us for convenience denote
		\[
		f(t):=\frac{1}{\vert B\vert}\left(1-\frac{1}{K}\int_B\tilde{\psi}_t J_{\Phi_t}(x)dx\right).
		\]
		We write the difference of equations (\ref{e:el moved to the ball})
		for $\tilde{\psi}_{t+h}$ and $\tilde{\psi}_t$ and divide it by $h$ to get
		\begin{equation*}
		\begin{aligned}
		&\int_{\mathbb{R}^n}\nabla\Psi \left(a_B \frac{A_{t+h}\nabla\tilde{\psi}_{t+h}-A_t\nabla\tilde{\psi}_t}{h}\right)dx
		+\frac{1}{K}\int_{B}\Psi \left(\frac{\tilde{\psi}_{t+h}-\tilde{\psi}_t}{h}\right)J_{\Phi_t}(x)dx\\
		&+\frac{1}{K}\int_B{\Psi\tilde{\psi}_{t+h}\frac{J_{\Phi_{t+h}}-J_{\Phi_t}}{h}}\,dx
		+\left(\int_B{\Psi J_{\Phi_t}(x)}dx\right)\frac{f(t+h)-f(t)}{h}\\
		&+\left(\int_B{\Psi\frac{J_{\Phi_{t+h}}(x)-J_{\Phi_t}(x)}{h}}\,dx\right)f(t+h)
=0
		\end{aligned}	
		\end{equation*}
		for any $\Psi\in D^1(\mathbb{R}^n)$.

		Now, introducing also $g_h(x):=\frac{\tilde{\psi}_{t+h}-\tilde{\psi}_t}{h}$ for convenience, we get
		\begin{equation}\label{e:EL for g_h}
		\begin{aligned}
		&\int_{\mathbb{R}^n}\nabla\Psi \left(a_B A_{t+h}\nabla g_h\right)dx
		+\frac{1}{K}\int_{B}\Psi g_h J_{\Phi_t}(x) dx
+\int_{\mathbb{R}^n}\nabla\Psi \left(a_B \frac{A_{t+h}-A_t}{h}\nabla \tilde{\psi}_t\right)dx\\
		&+\frac{1}{K}\int_B{\Psi\tilde{\psi}_{t+h}\frac{J_{\Phi_{t+h}}-J_{\Phi_t}}{h}}\,dx
		+\left(\int_B{\Psi J_{\Phi_t}(x)}dx\right)\frac{f(t+h)-f(t)}{h}\\
		&+\left(\int_B{\Psi\frac{J_{\Phi_{t+h}}(x)-J_{\Phi_t}(x)}{h}}\,dx\right)f(t+h)=0
		\end{aligned}	
		\end{equation}
		for any $\Psi\in D^1(\mathbb{R}^n)$.

		Now we want to get a uniform bound on $g_h$ in $D^1(\mathbb{R}^n)$. To do that we argue in a way similar to the proof of (\ref{e:bound in D^1}). We use $g_h$ as a test function in \eqref{e:EL for g_h} and get
		\begin{equation*}
		\begin{aligned}
		&\int_{\mathbb{R}^n}a_B\nabla g_h\cdot\left(A_{t+h}\nabla g_h\right)dx
		+\frac{1}{K}\int_{B} g_h^2 J_{\Phi_t}(x) dx
		+\int_{\mathbb{R}^n}a_B\nabla g_h\cdot\left( \frac{A_{t+h}-A_t}{h}\nabla \tilde{\psi}_t\right)dx\\
		&+\frac{1}{K}\int_B{g_h\tilde{\psi}_{t+h}\frac{J_{\Phi_{t+h}}-J_{\Phi_t}}{h}}\,dx
		+\left(\int_B{g_h J_{\Phi_t}(x)}dx\right)\frac{f(t+h)-f(t)}{h}\\
		&+\left(\int_B{g_h\frac{J_{\Phi_{t+h}}(x)-J_{\Phi_t}(x)}{h}}\,dx\right)f(t+h)
		=0.
		\end{aligned}	
		\end{equation*}
		Since $\frac{A(t+h,x)-A(t,x)}{h}$ is bounded in $L^\infty$
		and $A_t$ is uniformly elliptic we know that there exist
		some positive constant $c$ independent of $h$ such that
		\begin{equation*}
		\int_{\mathbb{R}^n}a_B\nabla g_h\cdot\left(A_{t+h}\nabla g_h\right)dx
		+\int_{\mathbb{R}^n}a_B\nabla g_h\cdot\left( \frac{A_{t+h}-A_t}{h}\nabla \tilde{\psi}_t\right)dx
		\geq c\int_{\mathbb{R}^n}\left\vert\nabla g_h\right\vert^2dx
		-C\int_{\mathbb{R}^n}\left\vert\nabla \psi_t\right\vert^2dx.
		\end{equation*}
		Thus,
		\begin{equation}\label{e:bound on H^1 norm of g - first bound}
		\begin{aligned}
		c\int_{\mathbb{R}^n}\left\vert\nabla g_h\right\vert^2dx
		+\frac{1}{K}\int_{B} g_h^2 J_{\Phi_t}(x)dx
		\leq C\int_{\mathbb{R}^n}\left\vert\nabla \psi_t\right\vert^2dx
		+\frac{1}{K}\int_B{\left\vert g_h\tilde{\psi}_{t+h}\frac{J_{\Phi_{t+h}}-J_{\Phi_t}}{h}\right\vert}\,dx\\
		+\left\vert\frac{f(t+h)-f(t)}{h}\right\vert\int_B{\left\vert g_h J_{\Phi_t}(x)\right\vert}dx
		+\vert f(t+h)\vert\int_B{\left\vert g_h\frac{J_{\Phi_{t+h}}(x)-J_{\Phi_t}(x)}{h}\right\vert}\,dx\\
		\leq C
		+C\int_B{\vert g_h\vert}\,dx
		+\left\vert\frac{f(t+h)-f(t)}{h}\right\vert\int_B{\vert g_h\vert}dx
		+\vert f(t+h)\vert\int_B{\left\vert g_h\right\vert}\,dx,		
		\end{aligned}
		\end{equation}
		where in the last inequality we used 
		the inequality \eqref{e:bound in D^1},
		Proposition \ref{bound on psi_t},
		and \eqref{e:flow close to id}.
		We want to show now that $f$ is bounded and Lipschitz. Indeed,
		we recall the definition of $f$ and use	
		the definition of $\tilde{\psi}_t$ and \eqref{otherformenergy}
		\begin{equation*}
		f(t)=\frac{1}{\vert B\vert}\left(1-\frac{1}{K}\int_B{\tilde{\psi}_{t}J_{\Phi_t}}(x)dx\right)
		=\frac{1}{\vert B\vert}-\frac{2}{K}\J(E_t).
		\end{equation*}
		We get that $f$ is bounded by \eqref{e:bound on energy}.
		To get Lipschitz continuity, we notice that by direct computation in Lagrangian coordinates one can get that
		$\frac{\J(E_{t+h})-\J(E_t)}{h}$ is uniformly bounded, 
		see \cite[Lemma 3.2]{DPHV}.
		Plugging this information into \eqref{e:bound on H^1 norm of g - first bound}, we get
		\begin{equation*}
		c\int_{\mathbb{R}^n}\left\vert\nabla g_h\right\vert^2dx
		+\frac{1}{K}\int_{B} g_h^2 J_{\Phi_t}(x)dx
		\leq C
		+C\int_B{\vert g_h\vert}\,dx.
		\end{equation*}
		Finally, we use Young's inequality and \eqref{e:flow close to id} to obtain
		\begin{equation}\label{e:bound on H^1 norm of g}
		c\int_{\mathbb{R}^n}\left\vert\nabla g_h\right\vert^2dx
		+\frac{1}{2K}\int_{B} g_h^2 J_{\Phi_t}(x)dx
		\leq C.
		\end{equation}
		Thus, $g_h$ is uniformly bounded 
		in $D^1(\mathbb{R}^n)$ and up to a subsequence, there exists a weak
		limit $g_0$ as $h$ goes to zero. Note that $g_0$ satisfies
		\begin{equation}\label{e:equation for g_0}
		\begin{split}
		\int_{\mathbb{R}^n}\nabla\Psi \left(a_B A_t\nabla g_0\right)dx
		+\int_{\mathbb{R}^n}\nabla\Psi \left(a_B \frac{d}{dt}A_t\nabla \tilde{\psi}_t\right)dx
		+\frac{1}{K}\int_{B}\Psi g_0 J_{\Phi_t}(x)\,dx\\
		+\frac{1}{K}\int_{B}{\Psi\tilde{\psi}_t\dot{J}_{\Phi_t}}\,dx
		-\frac{1}{\vert B\vert K}\left(\int_B{\Psi J_{\Phi_t}(x)}dx\right)\left(\int_B g_0 J_{\Phi_t}(x)\,dx-\int_{B}{\tilde{\psi}_t\dot{J}_{\Phi_t}}\,dx\right)\\
		+\frac{1}{\vert B\vert}\left(\int_{B}{\Psi\dot{J}_{\Phi_t}(x)}\,dx\right)\left(1-\frac{1}{K}\int_{B}{\tilde{\psi}_t J_{\Phi_t}(x)}\,dx\right)=0
		\end{split}	
		\end{equation}		
		for any $\Psi\in D^1(\mathbb{R}^n)$.
		Let us show that the equation \eqref{e:equation for g_0} has unique solution. To that end, assume that both $g_0$ and $g_0'$
		are solutions of \eqref{e:equation for g_0}. Then their
		difference $w=g_0-g_0'$ satisfies the following equation		
		\begin{equation}\label{e:linear part of equation for g_0}
		\int_{\mathbb{R}^n}\nabla\Psi \left(a_B A_t\nabla w\right)dx
		+\frac{1}{K}\int_{B}\Psi w J_{\Phi_t}(x)\,dx
		-\frac{1}{\vert B\vert K}\int_B{\Psi J_{\Phi_t}(x)}dx\int_B w J_{\Phi_t}(x)\,dx=0
		\end{equation}		
		for any $\Psi\in D^1(\mathbb{R}^n)$.
		Since $w\in D^1(\mathbb{R}^n)$, we can test \eqref{e:linear part of equation for g_0} with $w$ and get
		\begin{equation*}
		\int_{\mathbb{R}^n}\nabla w \left(a_B A_t\nabla w\right)dx
		+\frac{1}{K}\int_{B} w^2 J_{\Phi_t}(x)\,dx
		-\frac{1}{\vert B\vert K}\left(\int_B{w J_{\Phi_t}(x)}\,dx\right)^2=0.
		\end{equation*}		
		By Cauchy-Schwartz inequality, it yields 
		\begin{equation*}
		\int_{\mathbb{R}^n}\nabla w \left(a_B A_t\nabla w\right)dx
\leq 0,		
		\end{equation*}		
		which in turn gives us $w=0$ by ellipticity of $A_t$.
		Thus, the solution of \eqref{e:equation for g_0} is unique
		and thus the whole sequence $g_h$ converges to $g_0$.		
		
		To get the strong convergence of the material derivative,
		we observe that using $g_h$ as a test function in its
		Euler-Lagrange equation, we get the convergence
		of the norm in $H^1$ to the norm of $g_0$. That, together
		with weak convergence, gives us strong convergence of $g_h$.
				
		\textbf{Step 3:} existence of the shape derivative.
		
		We want to show that	
		$$\dot{\psi}_t=\frac{d}{dt}\tilde{\psi}_t-X\cdot\nabla\psi_t$$			in $D^1(E_t)\cap D^1(E_t^c)$.
		Indeed, since $\psi_t(x)=\dot{\psi}_t(\Phi_t^{-1}(x))$, 
		we have
		\begin{equation}\label{e:diff of composition}
			\frac{\psi_{t+h}(x)-\psi_{t+h}(x)}{h}
			=\frac{\psi_{t+h}(\Phi^{-1}_{t+h}(x))-\psi_{t}(\Phi^{-1}_{t+h}(x))}{h}+\frac{\psi_t(\Phi^{-1}_{t+h}(x))-\psi_{t}(\Phi^{-1}_t(x))}{h}.
		\end{equation}			
		The first term on the right-hand side converges strongly
		to $\frac{d}{dt}\psi_t(\Phi^{-1}_{t}(x))$ as $h$ goes to $0$
		by Step 2 and continuity of $\Phi_t$. As for the second term,
		by Proposition \ref{bound on psi_t} and the definition of $\Phi$, 
		it converges to $-\nabla \psi_t(\Phi_t^{-1}(x))\cdot X$. 				strongly in $D^1(E_t)\cap D^1(E_t^c)$.

		\textbf{Step 4:} the equation for the shape derivative.
		
		Now that we know that $t\mapsto\psi_t$ is differentiable,
		we can differentiate the Euler-Lagrange equation for $\psi_t$
		given by (\ref{e:el diff form on E_t})
		and we get
		\begin{equation*}
			\begin{cases}
				-\beta\Delta\dot{\psi}_t=-\frac{1}{K}\dot{\psi}_t+\frac{2}{K}\dot{\J}(E_t)\text{ in }E_t\\
				\Delta\dot{\psi}_t=0\text{ in }E_t^c\\
				\dot{\psi}_t^+-\dot{\psi}_t^{-}=
				-\left(\nabla\psi_t^+-\nabla\psi_t^{-}\right)\cdot X\text{ on }\partial E_t\\
		\beta\nabla\dot{\psi}^+_t\cdot\nu-\nabla\dot{\psi}^{1}_t\cdot\nu=
			-\left(\left(\beta\nabla[\nabla\psi_t^+]-\nabla[\nabla\psi_t^{-}]\right) X\right)\cdot\nu
			\text{ on }\partial E_t
		\end{cases}.
		\end{equation*}		
		
		Now we can use the boundary conditions in (\ref{e:el diff form on E_t})
		to get rid of the tangential part in the right-hand side.
		Indeed,
		\begin{equation*}
			-\left(\nabla\psi_t^+-\nabla\psi_t^{-}\right)\cdot X
			=-\left(\nabla^\tau\psi_t^+-\nabla^\tau\psi_t^{-}\right)\cdot X^\tau-\left(\nabla\psi_t^+-\nabla\psi_t^{-}\right)\cdot\nu (X\cdot\nu)
		\end{equation*}		
		and $\nabla^\tau\psi_t^+=\nabla^\tau\psi_t^{-}$ by differentiating
		the equality $\psi_t^+=\psi_t^{-}$ on the boundary of $E_t$.
	\end{proof}			

	The following observation, which is a consequence of equality
	for $\dot{\psi}_t$ will be useful for us.	
	\begin{lemma}\label{rmk:first bounds on psi dot}
		There exists $f\in H^{3/2}(E_t)\cap H^{3/2}(E_t^c)$ such that 
		\begin{equation}\label{e:def of f}
		f^\pm=\nabla\psi_t^\pm\cdot X\text{ on }\partial E_t, \quad 
		\Vert f^\pm\Vert_{H^{3/2}}\leq C\Vert\nabla\psi_t^\pm\cdot X\Vert_{H^1(\partial E_t)}.
		\end{equation} 
		Consider a function $v:=\dot{\psi}_t+f$, 
		Then $v$ satisfies the equation
		\begin{equation*}
			\begin{cases}
			-\beta\Delta v=-\frac{1}{K}v+\frac{2}{K}\dot{\J}(E_t)
			-\beta\Delta f+\frac{1}{K}f\text{ in }E_t,\\
			\Delta v=\Delta f\text{ in }E_t^c,\\
			v^+-v^{-}=0\text{ on }\partial E_t,\\
			\beta\nabla v^+\cdot\nu-\nabla v^{-}\cdot\nu=
			\left(-\left(\beta\nabla[\nabla\psi_t^+]-\nabla[\nabla\psi_t^{-}]\right) X
			+\beta\nabla f^+
			-\nabla f^-\right)\cdot\nu
			\text{ on }\partial E_t.
			\end{cases}
		\end{equation*}
		\begin{equation}\label{e:v on the boundary}
			v=\dot{\psi}_t^\pm+\nabla\psi_t^\pm\cdot X\text{ on }\partial E_t.
		\end{equation}
		Moreover, the following bounds hold:
		\begin{equation}\label{e:D^1 bound on v}
			\Vert v\Vert_{W^{1,2}(E_t)}+\Vert v\Vert_{D^{1,2}(E_t^c)}			\leq C\left(\vert\dot{\J}(E_t)\vert
			+\Vert X\cdot\nu\Vert_{H^1(\partial E_t)}\right);
		\end{equation}			
		\begin{equation}\label{e:L^2* bound on v}
			\Vert v\Vert_{L^{2^*}(\mathbb{R}^n)}\leq C\left(\vert\dot{\J}(E_t)\vert
			+\Vert X\cdot\nu\Vert_{H^1(\partial E_t)}\right).
		\end{equation}			
		\end{lemma}
		\begin{proof}
		The function $f$ exists since $\nabla\psi_t^\pm\cdot X\in H^1(\partial E_t)$. The equation for $v$ follows from the equation for $\dot{\psi}_t$ and the definition of $f$.
		Using divergence theorem, we get
		\begin{equation*}
		\begin{split}
			\int_{E_t}\frac{1}{K}v^2{dx}+\int_{E_t}\beta\vert\nabla v\vert^2{dx}+\int_{E_t^c}\vert\nabla v\vert^2{dx}
			=\int_{E_t}\left(\frac{2}{K}\dot{\J}(E_t)
			-\beta\Delta f+\frac{1}{K}f\right)v\,{dx}\\
			-\int_{E_t^c}\Delta f\,v\,{dx}
			+\int_{\partial E_t}\left(\left(-\left(\beta\nabla[\nabla\psi_t^+]-\nabla[\nabla\psi_t^{-}]\right) X
			+\beta\nabla f^+
			-\nabla f^-\right)\cdot\nu\right)v\,dx
		\end{split},	
		\end{equation*}			
		which by Young, Cauchy-Schwarz, and trace inequalities,
		recalling \eqref{e:def of f}, implies that
		\begin{equation*}
		\begin{split}
			\Vert v\Vert_{W^{1,2}(E_t)}+\Vert v\Vert_{D^{1}(E_t^c)}			\leq C\left(\vert\dot{\J}(E_t)\vert
			+\Vert\nabla\psi_t\cdot X\Vert_{H^1(\partial E_t)}\right)
		\end{split},	
		\end{equation*}			
		which in turn implies by Proposition \ref{bound on psi_t}
		and \eqref{e:tang part of X}
		\begin{equation*}
			\Vert v\Vert_{W^{1,2}(E_t)}+\Vert v\Vert_{D^{1}(E_t^c)}			\leq C\left(\vert\dot{\J}(E_t)\vert
			+\Vert X\cdot\nu\Vert_{H^1(\partial E_t)}\right).
		\end{equation*}			
		Moreover, we also can bound the $L^{2^*}$ norm of $v$. 
		Indeed, since $v$ doesn't have a jump on the boundary of $E_t$, 
		we know by \eqref{e:D^1 bound on v} that it belongs to the space $D^{1}(\mathbb{R}^n)$. Thus,
		employing Gagliardo-Nirenberg-Sobolev inequality we get
		\eqref{e:L^2* bound on v}.

	\end{proof}

	\begin{proposition}\label{p:first_derivative}
		For any $t\in [0,1]$, 
	\begin{equation*}
	\begin{aligned}
		\dot{\J}(E_t)&=\left(1-\frac{1}{K}\int_{E_t}\psi_t dx\right)\frac{1}{\vert E_t\vert}\int_{\partial E_t}{\psi_{E_t}(X\cdot\nu)}d\mathcal{H}^{n-1}\\
		&+\frac{1}{2}\int_{\partial E_t}{\left(\beta\vert\nabla \psi_t^+\vert^2-\vert\nabla \psi_t^{-}\vert^2\right)(X\cdot\nu)}d\mathcal{H}^{n-1}+\frac{1}{2K}\int_{\partial E_t}\psi_t^2(X\cdot\nu)d\mathcal{H}^{n-1}\\
		&-\int_{\partial E_t}{\left(\nabla\psi_t^{-}\cdot\nu\right)\left(\left(\nabla\psi_t^+-\nabla\psi_t^{-}\right)\cdot\nu\right)(X\cdot\nu)}d\mathcal{H}^{n-1}\\
		&=\left(1-\frac{1}{K}\int_{E_t}\psi_t dx\right)\frac{1}{\vert E_t\vert}\int_{E_t}{\div(\psi_tX)}dx
		+\frac{1}{2}\int_{\mathbb{R}^n}{\div(a_{E_t}\vert\nabla \psi_t\vert^2 X)}dx\\
		&\qquad+\frac{1}{2K}\int_{E_t}\div(\psi_t^2 X) dx
		-\int_{\mathbb{R}^n}{\div\left(a_{E_t}\left(\nabla\psi_t\cdot\nu\right)^2X\right)}dx.
	\end{aligned}	
	\end{equation*}
		In particular, $$\dot{\J}(B_1)=0.$$
	\end{proposition}		
		
	\begin{proof}
	We note that by (\ref{otherformenergy})
	\begin{equation*}
		\frac{d}{dt}\J(E_t)=\frac{1}{2\vert E_t\vert}\int_{E_t}{\dot{\psi}_t}dx+\frac{1}{2\vert E_t\vert}\int_{\partial E_t}{\psi_{E_t}(X\cdot\nu)}d\mathcal{H}^{n-1}.
	\end{equation*}
	
	Now we use the definition of $\J$ to get
	\begin{equation*}
	\begin{aligned}
		\dot{\J}(E_t)&=\int_{\mathbb{R}^n}{a_{E_t}\nabla\psi_t\cdot\nabla\dot{\psi}_t}dx+\frac{1}{2}\int_{\partial E_t}{\left(\beta\vert\nabla \psi_t^+\vert^2-\vert\nabla \psi_t^{-}\vert^2\right)(X\cdot\nu)}d\mathcal{H}^{n-1}\\
		&+2\dot{\J}(E_t)-\frac{2}{K}\vert E\vert 2\dot{\J}(E_t)\J(E_t)+\frac{1}{K}\int_{E_t}\dot{\psi_t}\psi_t dx+\frac{1}{2K}\int_{\partial E_t}\psi_t^2(X\cdot\nu)dx.
	\end{aligned}
	\end{equation*}
	
	Using (\ref{e:EL for phi dot at t}), we obtain
	\begin{equation*}
	\begin{aligned}
		\dot{\J}(E_t)&=-\frac{1}{\vert E_t\vert}\left(\int_{E_t}{\dot{\psi}_t}dx\right)\left(1-\frac{1}{K}\int_{E_t}\psi_t dx\right)+2\dot{\J}(E_t)\left(1-\frac{1}{K}\int_{E_t}{\psi_t}dx\right)\\
		&+\frac{1}{2}\int_{\partial E_t}{\left(\beta\vert\nabla \psi_t^+\vert^2-\vert\nabla \psi_t^{-}\vert^2\right)(X\cdot\nu)}d\mathcal{H}^{n-1}+\frac{1}{2K}\int_{\partial E_t}\psi_t^2(X\cdot\nu)dx\\
		&+\int_{\partial E_t}{\left(\beta\dot{\psi}^+\nabla\psi_t^+\cdot\nu-\dot{\psi}^{-}\nabla\psi_t^{-}\cdot\nu\right)}d\mathcal{H}^{n-1}\\
		&=\left(1-\frac{1}{K}\int_{E_t}\psi_t dx\right)\frac{1}{\vert E_t\vert}\int_{\partial E_t}{\psi_{E_t}(X\cdot\nu)}d\mathcal{H}^{n-1}\\
		&+\frac{1}{2}\int_{\partial E_t}{\left(\beta\vert\nabla \psi_t^+\vert^2-\vert\nabla \psi_t^{-}\vert^2\right)(X\cdot\nu)}d\mathcal{H}^{n-1}+\frac{1}{2K}\int_{\partial E_t}\psi_t^2(X\cdot\nu)d\mathcal{H}^{n-1}\\
		&-\int_{\partial E_t}{\left(\nabla\psi_t^{-}\cdot\nu\right)\left(\left(\nabla\psi_t^+-\nabla\psi_t^{-}\right)\cdot\nu\right)(X\cdot\nu)}d\mathcal{H}^{n-1}\\
		&=\left(1-\frac{1}{K}\int_{E_t}\psi_t dx\right)\frac{1}{\vert E_t\vert}\int_{E_t}{\div(\psi_tX)}dx
		+\frac{1}{2}\int_{\mathbb{R}^n}{\div(a_{E_t}\vert\nabla \psi_t\vert^2 X)}dx\\
		&\qquad+\frac{1}{2K}\int_{E_t}\div(\psi_t^2 X) dx
		-\int_{\mathbb{R}^n}{\div\left(a_{E_t}\left(\nabla\psi_t\cdot\nu\right)^2X\right)}dx.
	\end{aligned}	
	\end{equation*}
	
	Note that from the second to last expression it is easy to see that 
	$\dot{\J}(B_1)=0$ as $\psi_0$ is radial by Proposition \ref{p:radiality of min fct} and the volume
	of $E_t$ is constant (hence $\int_{\partial B_1}{(X\cdot\nu)}d\mathcal{H}^{n-1}=0$).

\end{proof}
%
%
%

	\subsection{Second derivative} 
	
	Now we differentiate again to get
	

	\begin{equation*}
	\begin{aligned}
		\ddot{\J}(E_t)
		&=-\frac{2}{K}\dot{\J}(E_t)\int_{E_t}{\mathrm{div}(\psi_t X)}dx\\
		&+\frac{1-\frac{2}{K}\vert E_t\vert\J(E_t)}{\vert E_t\vert}\left(\int_{E_t}{\mathrm{div}(\dot{\psi}_t X)}dx+\int_{\partial E_t}{\mathrm{div}(\psi_t X)(X\cdot\nu)}d\mathcal{H}^{n-1}\right)\\
		&+\int_{\partial E_t}{(\beta\nabla\psi_t^+\cdot\nabla\dot{\psi}_t^+
		-\nabla\psi_t^{-}\cdot\nabla\dot{\psi}_t^{-})(X\cdot\nu)}d\mathcal{H}^{n-1}\\
		&+\frac{1}{2}\int_{\partial E_t}{\nabla\left[\beta\vert\nabla \psi_t^+\vert^2-\vert\nabla \psi_t^{-}\vert^2\right]\cdot X(X\cdot\nu)}d\mathcal{H}^{n-1}\\
		&+\frac{1}{K}\int_{\partial E_t}{\psi_t\dot{\psi}_t(X\cdot\nu)}d\mathcal{H}^{n-1}
		+\frac{1}{K}\int_{\partial E_t}{\psi_t\nabla\psi_t\cdot X(X\cdot\nu)}d\mathcal{H}^{n-1}\\
		&-2\int_{\partial E_t}\left(\beta\left(\nabla\dot{\psi}_t^+\cdot\nu\right)\left(\nabla\psi_t^+\cdot\nu\right)
		-\left(\nabla\dot{\psi}_t^{-}\cdot\nu\right)\left(\nabla\psi_t^{-}\cdot\nu\right)
		\right)(X\cdot\nu)d\mathcal{H}^{n-1}\\
		&-\int_{\partial E_t}\nabla\left[\beta(\nabla\psi_t^+\cdot\nu)^2-(\nabla\psi_t^{-}\cdot\nu)^2\right]\cdot X(X\cdot\nu)d\mathcal{H}^{n-1}.
	\end{aligned}	
	\end{equation*}
	
	Using that the vector field $X$ is divergence-free in the neighborhood of $\partial B_1$ we get for $t$ small enough
	
	\begin{equation}\label{e:second_derivative}
	\begin{aligned}
		\ddot{\J}(E_t)
		&=-\frac{2}{K}\dot{\J}(E_t)\int_{\partial E_t}{\psi_t (X\cdot\nu)}d\mathcal{H}^{n-1}\\
		&+\frac{1-\frac{2}{K}\vert E_t\vert\J(E_t)}{\vert E_t\vert}\left(\int_{\partial E_t}{\dot{\psi}_t (X\cdot\nu)}d\mathcal{H}^{n-1}+\int_{\partial E_t}{(\nabla\psi_t^+\cdot X)(X\cdot\nu)}d\mathcal{H}^{n-1}\right)\\
		&+\int_{\partial E_t}{(\beta\nabla\psi_t^+\cdot\nabla\dot{\psi}_t^+
		-\nabla\psi_t^{-}\cdot\nabla\dot{\psi}_t^{-})(X\cdot\nu)}d\mathcal{H}^{n-1}\\
		&+\frac{1}{2}\int_{\partial E_t}{\nabla\left[\beta\vert\nabla \psi_t^+\vert^2-\vert\nabla \psi_t^{-}\vert^2\right]\cdot X(X\cdot\nu)}d\mathcal{H}^{n-1}\\
		&+\frac{1}{K}\int_{\partial E_t}{\psi_t\dot{\psi}_t^+(X\cdot\nu)}d\mathcal{H}^{n-1}
		+\frac{1}{K}\int_{\partial E_t}{\psi_t\nabla\psi_t^+\cdot X(X\cdot\nu)}d\mathcal{H}^{n-1}\\
		&-2\int_{\partial E_t}\left(\beta\left(\nabla\dot{\psi}_t^+\cdot\nu\right)\left(\nabla\psi_t^+\cdot\nu\right)
		-\left(\nabla\dot{\psi}_t^{-}\cdot\nu\right)\left(\nabla\psi_t^{-}\cdot\nu\right)
		\right)(X\cdot\nu)d\mathcal{H}^{n-1}\\
		&-\int_{\partial E_t}\nabla\left[\beta(\nabla\psi_t^+\cdot\nu)^2-(\nabla\psi_t^{-}\cdot\nu)^2\right]\cdot X(X\cdot\nu)d\mathcal{H}^{n-1}.
	\end{aligned}	
	\end{equation}

	Now to prove Lemma \ref{taylor} we only need the following bound 
	on the second derivative.	
	
	\begin{lemma}\label{l:secdercont}
		There exist $\delta>0$ and a bounded function $g$
		such that if 
		$\Vert\varphi\Vert_{C^{2,\vartheta}}<\delta$, then
		\begin{equation*}
			\left\vert\ddot{\J}(E_t)\right\vert
			\leq g\left(\Vert\varphi\Vert_{C^{2,\vartheta}}\right)\Vert X\cdot\nu\Vert^2_{H^{1}(\partial B_1)}.
		\end{equation*}
	\end{lemma}

	We will need the following proposition.
	\begin{proposition}\label{sobolev bounds for transmission problem}
	$$		\Vert\dot{\psi}_t^+\Vert_{H^1(\partial E_t)}+
		\Vert\dot{\psi}_t^{-}\Vert_{H^{1}(\partial E_t)}
		\leq C\left(\Vert X\cdot\nu\Vert_{H^{1}(\partial E_t)}+\left\vert\dot{\mathcal{J}}(E_t)\right\vert\right).$$
	\end{proposition}		

	To prove the proposition we will use the following theorem
	concerning Sobolev bounds.
	\begin{theorem}(\cite[Theorem 4.20]{McL})\label{t:sobolev bounds for transmission problem}
		Let $G_1$ and $G_2$ be bounded open subsets of $\mathbb{R}^n$
		such that $\overline{G}_1\Subset G_2$ and $G_1$ intersects
		an $(n-1)$-dimensional manifold $\Gamma$, and put
		$$\Omega_j^\pm=G_j\cap\Omega^\pm\text{ and }\Gamma_j=G_j\cap\Gamma\text{ for }j=1,2.$$
		Suppose, for an integer $r\geq 0$, that $\Gamma_2$ is $C^{r+1,1}$, and consider two equations
		$$\mathcal{P}u^\pm=f^\pm\text{ on }\Omega_2^\pm,$$
		where $\mathcal{P}$ is strongly elliptic on $G_2$ with coefficients in $C^{r,1}(\overline{\Omega_2^\pm})$. 
		If $u\in L^2(G_2)$ satisfies
		$$u^\pm\in H^1(\Omega_2^\pm),\quad [u]_\Gamma\in H^{r+\frac{3}{2}}(\Gamma_2),\quad[\mathcal{B}_\nu u]_\Gamma\in H^{r+\frac{1}{2}}(\Gamma_2)\footnote{$\mathcal{B}_\nu$ here denotes conormal derivative. In our case it reduces to $a_E\partial_\nu$ since we deal with Laplacian.},$$
		and if $f^\pm\in H^r(\Omega_2^\pm)$, then $u^\pm\in H^{r+2}(\Omega_1^\pm)$ and
		\begin{equation*}
		\begin{aligned}
		\Vert u^+\Vert_{H^{r+2}(\Omega_1^+)}+\Vert u^+\Vert_{H^{r+2}(\Omega_1^-)}&\leq C\left(\Vert u^+\Vert_{H^1(\Omega_2^+)}+\Vert u^-\Vert_{H^1(\Omega_2^-)}\right)\\
		&+ C\left(\Vert [u]_{\Gamma_2}\Vert_{H^{r+\frac{3}{2}}(\Gamma_2)}+\Vert [\mathcal{B}_\nu u]_{\Gamma_2}\Vert_{H^{r+\frac{1}{2}}(\Gamma_2)}\right)\\
		&+C\left(\Vert f^+\Vert_{H^r(\Omega_2^+)}+\Vert f^-\Vert_{H^r(\Omega_2^-)}\right).
		\end{aligned}
		\end{equation*}
	\end{theorem}		

	We need an analogue of the above theorem for $r=-\frac{1}{2}$.
	To get it, we are going to interpolate between $r=0$ and $r=-1$. 
	We first prove the following lemma.
	\begin{lemma}\label{l:H^1/2 bound}
		Let $E$ be a set with the boundary in $C^{1,1}$ and 
		let $R>0$ be such that $B_R\supset \overline{E}$.
		Consider the equation
		\begin{equation}\label{e:transmission pb eq}
			\begin{cases}
			\beta\Delta u^+=f^+\text{ in }E,\\
			\Delta u^-=f^-\text{ in }B_R\backslash E,\\
			u^+=u^-\text{ on }\partial E,\\
			\beta\nabla u^+\cdot\nu-\nabla u^-\cdot\nu=g\text{ on }\partial E,\\
			u^-=0\text{ on }\partial B_R,
			\end{cases}
		\end{equation}
		where $f^+\in H^{-1}(E)$,$f^-\in H^{-1}(B_R\backslash E)$, 
		and $g\in H^{-1/2}(\partial E)$ are given.
		Then there exists $u$ - the solution of \eqref{e:transmission pb eq} in $W^{1,2}_0(B_R)$ and it satisfies
		\begin{equation}\label{H^1 bound for solutions}
			\left\Vert u\right\Vert_{H^1(B_R)}^2\leq
			C\left(\Vert f^+\Vert_{H^{-1}(E)}^2+\Vert f^-\Vert_{H^{-1}(B_R\backslash E)}^2+\Vert g\Vert_{H^{-1/2}(\partial E)}^2\right)
	\end{equation}
		with $C=C(n,R)>0$.
		Moreover, if $f^+1\in H^{-1/2}(E)$,$f^-\in H^{-1/2}(B_R\backslash E)$, 
		and $g\in L^2(\partial E)$, then
		\begin{equation}\label{H^1/2 bound for solutions}
			\left\Vert u\right\Vert_{H^{3/2}(B_R)}^2\leq
			C\left(\Vert f^+\Vert_{H^{-1/2}(E)}^2+\Vert f^-\Vert_{H^{-1/2}(B_R\backslash E)}^2+\Vert g\Vert_{L^2(\partial E)}^2\right)
	\end{equation}
		with $C=C(n,R)>0$.
	\end{lemma}	
	\begin{proof}
	First we observe that the solution in $H^1$ exists since it is a minimizer of
	the following convex functional:
	\begin{equation*}
		\int_{E_t}\left(\frac{1}{2}\beta\left\vert\nabla u^+\right\vert^2-f^+ u^+\right)+\int_{E_t^c}\left(\frac{1}{2}\left\vert\nabla u^-\right\vert^2-f^- u^-\right)+\int_{\partial E_t}g(u^+-u^-).
	\end{equation*}
	Note that if we test the equation with the solution itself, we get
	\begin{equation*}
		\int_{E_t}{\frac{1}{2}\beta\left\vert\nabla u^+\right\vert^2}dx
		+\int_{E_t^c}{\frac{1}{2}\left\vert\nabla u^-\right\vert^2}dx
		=-\int_{E_t}{f^+ u^+}dx-\int_{E_t}{f^- u^-}dx+\int_{\partial E_t}{u^+ g}\,d\mathcal{H}^{n-1}.
	\end{equation*}
	By Poincar\'e, Cauchy-Schwarz, Young, and the trace inequality we obtain \eqref{H^1 bound for solutions}.

	Now we consider an operator that takes the functions of the right-hand
	side and returns the solution of the corresponding transmission problem, i.e. we define $T(f_1,f_2,g)$ for $f_1\in H^r(E_t)$, $f_2\in H^r(E_t^c)$, $g\in H^{r+\frac{1}{2}}(\partial E_t)$ as the only $H^1$ solution of \eqref{e:transmission pb eq}.
	
	By \eqref{H^1 bound for solutions}, $T:H^r\times H^r\times H^{r+\frac{1}{2}}\rightarrow H^{r+2}$ for $r=-1$. 			
	Moreover, \eqref{H^1 bound for solutions} together with Theorem \ref{t:sobolev bounds for transmission problem} yields $T:H^r\times H^r\times H^{r+\frac{1}{2}}\rightarrow H^{r+2}$ for $r\geq 0$ - integer.
	Thus, interpolating
	between $r=0$ and $r=-1$ we get that $T:H^{-\frac{1}{2}}\times H^{-\frac{1}{2}}\times L^2\rightarrow H^{\frac{3}{2}}$, so \eqref{H^1/2 bound for solutions} holds for appropriately regular right-hand side. 
	\end{proof}
	
	\begin{proof}(Proposition \ref{sobolev bounds for transmission problem})
	Since we are interested only in the value of $\dot{\psi}_t$ on 
	$\partial E_t$, we multiply it by a cut-off function $\eta$. 
	The function $\eta\in C_c^\infty(\mathbb{R}^n)$ is such that
	$$0\leq\eta\leq 1,\quad\eta\equiv 1\text{ in }B_2,\quad\eta\equiv 0\text{ outside of }B_3,\quad\vert\nabla\eta\vert\leq 2,\quad\vert\Delta\eta\vert\leq C(n).$$
	We would also like to eliminate the jump on the boundary
	in order to use Lemma \ref{l:H^1/2 bound}, so we consider a function $u:=v\eta$, where $v$ is as in Lemma \ref{rmk:first bounds on psi dot} (we recall that $v=\dot{\psi}_t+f$, where $f$ is a $H^{3/2}$ continuation of $\nabla\psi_t\cdot X$ from $\partial E_t$ inside and outside).
	For $\delta$ small enough, all sets $E_t$ lie inside of $B_2$, 
	so 
	\begin{equation}\label{e:u on the boundary}
		u=\dot{\psi}_t+\nabla\psi_t\cdot X\text{ on }\partial E_t.
	\end{equation}	
	Note that $u$ satisfies	
	\begin{equation*}
	\begin{cases}
		-\beta\Delta u=-\frac{1}{K}v+\frac{2}{K}\dot{\J}(E_t)+\Delta f\text{ in }E_t,\\
		\Delta u=\nabla v\cdot\nabla\eta+\left(\dot{\psi}_t+f\right)\Delta\eta\text{ in }E_t^c,\\
		u^+-u^{-}=0\text{ on }\partial E_t,\\
		\beta\nabla u^+\cdot\nu-\nabla u^{-}\cdot\nu=
\left(-\left(\beta\nabla[\nabla\psi_t^+]-\nabla[\nabla\psi_t^{-}]\right) X
+\beta\nabla f^+-\nabla f^{-}\right)\cdot\nu
			\text{ on }\partial E_t,\\
		u=0\text{ on }\partial B_3.
	\end{cases}
	\end{equation*}		
	By Lemma \ref{l:H^1/2 bound},
	\begin{equation*}
	\begin{aligned}
		\Vert u^+\Vert_{H^{\frac{3}{2}}(E_t)}+\Vert u^-\Vert_{H^{\frac{3}{2}}(E_t^c)}\leq
		&\,\,C\left(\left\Vert \left(\beta\nabla[\nabla\psi_t^+
		\cdot X]\right)\cdot\nu\right\Vert_{L^2(\Gamma_2)}+
		\left\Vert \left(\nabla[\nabla\psi_t^{-}\cdot X]\right)\cdot\nu\right\Vert_{L^2(\Gamma_2)}\right)\\
		&+ C\left(\left\Vert \left(\beta\nabla[\nabla\psi_t^+]
		\cdot X\right)\cdot\nu\right\Vert_{L^2(\Gamma_2)}+
		\left\Vert \left(\nabla[\nabla\psi_t^{-}]\cdot X\right)\cdot\nu\right\Vert_{L^2(\Gamma_2)}\right)\\
		&+C\left(\left\Vert \frac{1}{K}v\right\Vert_{H^{-\frac{1}{2}}(E_t)}
		+\left\Vert\frac{2}{K}\dot{\J}(E_t)\right\Vert_{H^{-\frac{1}{2}}(E_t)}+\left\Vert\Delta f\right\Vert_{H^{-\frac{1}{2}}(E_t)}\right)\\
		&+C\left(\left\Vert\nabla v\cdot\nabla\eta\right\Vert_{H^{-\frac{1}{2}}(E_t^c)}
		+\left\Vert v\Delta\eta\right\Vert_{H^{-\frac{1}{2}}(E_t^c)}\right).
	\end{aligned}
	\end{equation*}	
	Now we employ Proposition \ref{bound on psi_t}, 
	inequality \eqref{e:tang part of X}, and the definition of $f$ 
	to get
	\begin{equation*}
	\begin{aligned}
		\Vert u^+\Vert_{H^{\frac{3}{2}}(E_t)}+\Vert u^-\Vert_{H^{\frac{3}{2}}(E_t^c)}&\leq C\left(\Vert X\cdot\nu\Vert_{H^{1}(\partial E_t)}+\left\vert\dot{\mathcal{J}}(E_t)\right\vert\right)\\
		&\qquad\qquad+C\left(
		\left\Vert\nabla v\cdot\nabla\eta\right\Vert_{H^{-\frac{1}{2}}(E_t^c)}
		+\left\Vert v\Delta\eta\right\Vert_{H^{-\frac{1}{2}}(E_t^c)}\right).
	\end{aligned}
	\end{equation*}	
	Remembering \eqref{e:u on the boundary}, using trace inequality
	and properties of $\eta$, we have
	\begin{equation*}
	\begin{split}
		\Vert \dot{\psi}_t^+\Vert_{H^1}(\partial E_t)+\Vert \dot{\psi}_t^-\Vert_{H^1(\partial E_t)}&\leq C\left(\Vert X\cdot\nu\Vert_{H^{1}(\partial E_t)}+\left\vert\dot{\mathcal{J}}(E_t)\right\vert\right)\\
		&+C\left(
		\left\Vert\nabla v\cdot\nabla\eta\right\Vert_{H^{-\frac{1}{2}}(E_t^c)}
		+\left\Vert v\Delta\eta\right\Vert_{H^{-\frac{1}{2}}(E_t^c)}\right)\\
		&\leq C\left(\Vert X\cdot\nu\Vert_{H^{1}(\partial E_t)}+\left\vert\dot{\mathcal{J}}(E_t)\right\vert\right)+C\left(
		\left\Vert\nabla v\cdot\nabla\eta\right\Vert_{L^2(E_t^c)}
		+\left\Vert v\Delta\eta\right\Vert_{L^2(E_t^c)}\right)\\
		&\leq C\left(\Vert X\cdot\nu\Vert_{H^{1}(\partial E_t)}+\left\vert\dot{\mathcal{J}}(E_t)\right\vert\right)+C\left(
		\left\Vert\nabla v\right\Vert_{L^2(E_t^c)}
		+\left\Vert v\right\Vert_{L^2(B_3\backslash B_2)}\right).
	\end{split}	
	\end{equation*}
	Now it remains to recall the bounds \eqref{e:D^1 bound on v}
	and \eqref{e:L^2* bound on v} and notice that 
	$\Vert\cdot\Vert_{L^2(B_3\backslash B_2)}~\leq~C\Vert\cdot~\Vert_{L^{2^*}(B_3\backslash B_2)}$.	
	\end{proof}
		
	\begin{proof}(Lemma \ref{l:secdercont})

		Let us first show that the lemma is implied by the following claim.
		
		{\bf Claim: } 
		$\left\vert\ddot{\J}(E_t)\right\vert\leq C\left(\left\Vert X\cdot\nu\right\Vert^2_{H^1(\partial B_1)}+\dot{\J}(E_t)\left\Vert X\cdot\nu\right\Vert_{H^1(\partial B_1)}\right).$

		Indeed, suppose we proved the claim. Denote $\dot{\J}(E_t)$ by $f(t)$.
		Then we know the following:
		\begin{equation*}
		\begin{cases}
		\left\vert f'(t)\right\vert\leq C\left(\left\Vert X\cdot\nu\right\Vert^2_{H^1(\partial B_1)}+f(t)\left\Vert X\cdot\nu\right\Vert_{H^1(\partial B_1)}\right),\\
		f(0)=0.
		\end{cases}
		\end{equation*}
		Let us show that 
		\begin{equation}\label{e:bound for first derivative}
		\vert f(t)\vert\leq\left\Vert X\cdot\nu\right\Vert_{H^1(\partial B_1)},
		\end{equation} 
		then the lemma will follow immediately.
		Suppose that there exists a time $t\in (0,1]$ such that 
		the inequality (\ref{e:bound for first derivative}) fails.
		We denote by $t^*$ the first time when it happens, i.e.
		$$t^*:=\inf_{t\in [0,1]}\left\{t:(\ref{e:bound for first derivative})\text{ fails}\right\}.$$
		Since inequality (\ref{e:bound for first derivative}) is true for $t=0$,
		the following holds:
		\begin{equation*}
			\vert f(t^*)\vert=\left\Vert X\cdot\nu\right\Vert_{H^1(\partial B_1)}, \quad\vert f(t)\vert\leq\left\Vert X\cdot\nu\right\Vert_{H^1(\partial B_1)}\text{ for }t\in [0,t^*].
		\end{equation*}
		Now, as $f(0)=0$, we can write
		\begin{equation*}
			f(t^*)=\int_0^{t^*}{f'(t)}dt
		\end{equation*}				
		and thus
		\begin{equation*}
		\begin{aligned}
			\left\Vert X\cdot\nu\right\Vert_{H^1(\partial B_1)}&=\vert f(t^*)\vert\leq\int_0^{t^*}{\vert f'(t)\vert}dt\\
			&\leq\int_0^{t^*}{C\left(\left\Vert X\cdot\nu\right\Vert^2_{H^1(\partial B_1)}+f(t)\left\Vert X\cdot\nu\right\Vert_{H^1(\partial B_1)}\right)}dt
			\leq 2C\left\Vert X\cdot\nu\right\Vert^2_{H^1(\partial B_1)}.
		\end{aligned}	
		\end{equation*}				
		However, that cannot hold for 
		$\Vert X\cdot\nu\Vert_{H^1(\partial B_1)}$ small enough. 
		That means that (\ref{e:bound for first derivative}) holds
		for all times $t$.
			
		{\bf Proof of the claim.}	
		
		By \eqref{e:second_derivative} we have
%
%
		\begin{equation*}
	\begin{aligned}
		&\ddot{\J}(E_t)
		=-\frac{2}{K}\dot{\J}(E_t)\int_{\partial E_t}{\psi_t (X\cdot\nu)}d\mathcal{H}^{n-1}\\
		&+\frac{1}{2}\int_{\partial E_t}{\nabla\left[\beta\vert\nabla \psi_t^+\vert^2-\vert\nabla \psi_t^{-}\vert^2\right]\cdot X(X\cdot\nu)}d\mathcal{H}^{n-1}\\
		&+\int_{\partial E_t}{\left(\frac{1-\frac{2}{K}\vert B_1\vert\J(E_t)}{\vert B_1\vert}+\frac{1}{K}\psi_t\right)(\nabla\psi_t^+\cdot X)(X\cdot\nu)}d\mathcal{H}^{n-1}\\
		&+\int_{\partial E_t}{\left(\left(\frac{1-\frac{2}{K}\vert B_1\vert\J(E_t)}{\vert B_1\vert}\right)+\frac{1}{K}\psi_t\right)\dot{\psi}_t^+ (X\cdot\nu)}d\mathcal{H}^{n-1}\\
		&+\int_{\partial E_t}{(\beta\nabla\psi_t^+\cdot\nabla\dot{\psi}_t^+
		-\nabla\psi_t^{-}\cdot\nabla\dot{\psi}_t^{-})(X\cdot\nu)}d\mathcal{H}^{n-1}\\
		&-2\int_{\partial E_t}\left(\beta\left(\nabla\dot{\psi}_t^+\cdot\nu\right)\left(\nabla\psi_t^+\cdot\nu\right)
		-\left(\nabla\dot{\psi}_t^{-}\cdot\nu\right)\left(\nabla\psi_t^{-}\cdot\nu\right)
		\right)(X\cdot\nu)d\mathcal{H}^{n-1}\\
		&-\int_{\partial E_t}\nabla\left[\beta(\nabla\psi_t^+\cdot\nu)^2-(\nabla\psi_t^{-}\cdot\nu)^2\right]\cdot X(X\cdot\nu)d\mathcal{H}^{n-1}\\	
		&=:I_1(t)+I_2(t)+I_3(t)+I_4(t)+I_5(t)+I_6(t)+I_7(t).
	\end{aligned}	
	\end{equation*}

	
	We start with $I_1$. Using the expression for $\dot{\J}(E_t)$ obtained in Proposition \ref{p:first_derivative}, we get
	\begin{equation*}
	\begin{aligned}
		-\frac{K}{2}I_1(t)
		=&\dot{\J}(E_t)\int_{\partial E_t}{\psi_t (X\cdot\nu)}d\mathcal{H}^{n-1}
		=\left(1-\frac{1}{K}\int_{E_t}\psi_t dx\right)\frac{1}{\vert B_1\vert}\left(\int_{\partial E_t}{\psi_t(X\cdot\nu)}d\mathcal{H}^{n-1}\right)^2\\
		&+\frac{1}{2}\int_{\partial E_t}{\left(\beta\vert\nabla \psi_t^+\vert^2-\vert\nabla \psi_t^{-}\vert^2\right)(X\cdot\nu)}d\mathcal{H}^{n-1}\int_{\partial E_t}{\psi_t (X\cdot\nu)}d\mathcal{H}^{n-1}\\
		&+\frac{1}{2K}\int_{\partial E_t}\psi_t^2(X\cdot\nu)d\mathcal{H}^{n-1}\int_{\partial E_t}{\psi_t (X\cdot\nu)}d\mathcal{H}^{n-1}.
	\end{aligned}	
	\end{equation*}	
	

	Thus, $$\left\vert I_1(t)\right\vert\leq g(\Vert\psi_t\Vert_{C^1(\overline{E_t})})\Vert X\cdot\nu\Vert_{L^1(\partial E_t)}^2$$ for some bounded function $g$.

	To prove the bounds for $I_2$, $I_3$, and $I_7$, we rewrite $X$ as $(X\cdot\nu)\nu+X^\tau$ and use that
	\begin{equation*}
		\vert X^\tau\circ\Phi_t\vert\leq\omega(\Vert\varphi\Vert_{C^{2,\vartheta}})
		\vert X\cdot\nu_{B_1}\vert.
	\end{equation*}

	Indeed,
	\begin{equation*}
	\begin{aligned}
		&I_2(t)=\frac{1}{2}\int_{\partial E_t}{\nabla\left[\beta\vert\nabla \psi_t^+\vert^2-\vert\nabla \psi_t^{-}\vert^2\right]\cdot X(X\cdot\nu)}d\mathcal{H}^{n-1}\\
		&=\frac{1}{2}\int_{\partial E_t}{\nabla\left[\beta\vert\nabla \psi_t^+\vert^2-\vert\nabla \psi_t^{-}\vert^2\right]\cdot\nu(X\cdot\nu)^2}d\mathcal{H}^{n-1}\\
		&+\frac{1}{2}\int_{\partial E_t}{\nabla\left[\beta\vert\nabla \psi_t^+\vert^2-\vert\nabla \psi_t^{-}\vert^2\right]\cdot X^\tau(X\cdot\nu)}d\mathcal{H}^{n-1}
	\end{aligned}
	\end{equation*}
	and thus
	$$\left\vert I_2(t)\right\vert\leq g(\Vert\psi_t\Vert_{C^2(\overline{E_t})})\Vert X\cdot\nu\Vert_{L^2(\partial E_t)}^2$$ for some bounded function $g$. $I_3$ and $I_7$ are treated in the same way.

	To bound $I_4$, $I_5$, and $I_6$ we use Proposition \ref{sobolev bounds for transmission problem} and Proposition \ref{bound on psi_t}.
	Let us show the inequality for $I_5$, $I_4$ and $I_6$ can be treated in a similar way.
	\begin{equation}
	\begin{split}
	&\left\vert\int_{\partial E_t}{\left(\beta\nabla\psi_t^+\cdot\nabla\dot{\psi}_t^+
		-\nabla\psi_t^{-}\cdot\nabla\dot{\psi}_t^{-}\right)(X\cdot\nu)}d\mathcal{H}^{n-1}\right\vert\\
	&\leq\int_{\partial E_t}{\left(\left\vert\beta\nabla\psi_t^+\cdot\nabla\dot{\psi}_t^+\right\vert
		+\left\vert\nabla\psi_t^{-}\cdot\nabla\dot{\psi}_t^{-}\right\vert\right)\left\vert X\cdot\nu\right\vert}d\mathcal{H}^{n-1}\\
		&\leq
		\left(
		\left(\int_{\partial E_t}{\left\vert\beta\nabla\psi_t^+\cdot\nabla\dot{\psi}_t^+\right\vert^2}\right)^\frac{1}{2}
		+\left(\int_{\partial E_t}{\left\vert\nabla\psi_t^{-}\cdot\nabla\dot{\psi}_t^{-}\right\vert^2}\right)^\frac{1}{2}
		\right)
\Vert X\cdot\nu\Vert_{L^2(\partial E_t)}\\
		&\leq
		g(\Vert\psi_t\Vert_{C^2(\overline{E_t})})
		\left(
		\left(\int_{\partial E_t}{\left\vert\nabla\dot{\psi}_t^+\right\vert^2}\right)^\frac{1}{2}
		+\left(\int_{\partial E_t}{\left\vert\nabla\dot{\psi}_t^{-}\right\vert^2}\right)^\frac{1}{2}
		\right)
\Vert X\cdot\nu\Vert_{L^2(\partial E_t)}\\
		&\leq
		g(\Vert\psi_t\Vert_{C^2(\overline{E_t})})
		\left(
\Vert X\cdot\nu\Vert_{H^{1}(\partial E_t)}+\left\vert\dot{\mathcal{J}}(E_t)\right\vert		
		\right)
\Vert X\cdot\nu\Vert_{L^2(\partial E_t)}
	\end{split}
	\end{equation}
	\end{proof}
	
	Now we are ready to prove Lemma \ref{taylor}.
	\begin{proof} (Lemma \ref{taylor})
		\begin{equation*}
			\J(E)=\J(B_1)+\dot{\J}(B_1)+\int_0^1{(1-s)\ddot{\J}(E_s)}ds.
		\end{equation*}
		By Proposition \ref{p:first_derivative} we know that $\dot{\J}(B_1)=0$.
		Now use Lemma \ref{l:secdercont} to bound the integral.
	\end{proof}

\appendix

\section{Second derivative on the ball}\label{sec: second der on the ball}

		We want to show that the second variation of the energy
		which we know is bounded by $\Vert\varphi\Vert_{H^1}^2$
		is actually bounded by a stronger $H^{1/2}$ norm on the ball.
				We don't need this for our main results but it is a sharp bound
		so we prove it for the sake of completeness. 
		
		We first show the following proposition.
		\begin{proposition} Given $\vartheta\in(0,1]$ there exists $\delta=\delta(n,\vartheta)$ such that for any $\varphi\in C^{2,\vartheta}$ with $\Vert\varphi\Vert_{C^{2,\vartheta}}<\delta$ we have 		
		\begin{equation*}
			\partial^2 \G(B_1)[\varphi,\varphi]:=
			\hat{c}_1\int_{\partial B_1}{\varphi^2}d\mathcal{H}^{n-1}
			+\int_{\partial B_1}{\hat{c}_2 H(\varphi)^++\hat{c}_3(\nabla H(\varphi)^{-}\cdot\nu)\varphi}d\mathcal{H}^{n-1},
		\end{equation*}
		where $H(\varphi)$ is the unique solution of
			\begin{equation*}
		\begin{cases}
			\beta\Delta u=\frac{1}{K}u\text{ in }B_1,\\
			\Delta u=0\text{ in }B_1^c,\\
			u^+-u^{-}
			=c_1\varphi
			\text{ on }\partial B_1,\\
			\beta\nabla u^+\cdot\nu
			-\nabla u^{-}\cdot\nu
			=c_2\varphi\text{ on }\partial B_1
		\end{cases}
	\end{equation*}
 		and $\hat{c}_1$, $\hat{c}_2$, $\hat{c}_3$, $c_1$, and	$c_2$ are constants depending only on $\beta$, $K$ and dimension $n$.
		\end{proposition}		
		
		\begin{proof}
		We have
		\begin{equation*}
		\begin{aligned}
		\ddot{\J}(B_1)
		&=-\frac{2}{K}\dot{\J}(B_1)\int_{\partial B_1}{\psi_0 (X\cdot\nu)}d\mathcal{H}^{n-1}\\
		&+\frac{1-\frac{2}{K}\vert B_1\vert\J(B_1)}{\vert B_1\vert}\left(\int_{\partial B_1}{\dot{\psi}_0^+ (X\cdot\nu)}d\mathcal{H}^{n-1}+\int_{\partial B_1}{(\nabla\psi_0^+\cdot X)(X\cdot\nu)}d\mathcal{H}^{n-1}\right)\\
		&+\int_{\partial B_1}{(\beta\nabla\psi_0^+\cdot\nabla\dot{\psi}_0^+
		-\nabla\psi_0^{-}\cdot\nabla\dot{\psi}_0^{-})(X\cdot\nu)}d\mathcal{H}^{n-1}\\
		&+\frac{1}{2}\int_{\partial B_1}{\nabla\left[\beta\vert\nabla \psi_0^+\vert^2-\vert\nabla \psi_0^{-}\vert^2\right]\cdot X(X\cdot\nu)}d\mathcal{H}^{n-1}\\
		&+\frac{1}{K}\int_{\partial B_1}{\psi_0\dot{\psi}_0^+(X\cdot\nu)}d\mathcal{H}^{n-1}
		+\frac{1}{K}\int_{\partial B_1}{\psi_0\nabla\psi_0^+\cdot X(X\cdot\nu)}d\mathcal{H}^{n-1}\\
		&-2\int_{\partial B_1}\left(\beta\left(\nabla\dot{\psi}_0^+\cdot\nu\right)\left(\nabla\psi_0^+\cdot\nu\right)
		-\left(\nabla\dot{\psi}_0^{-}\cdot\nu\right)\left(\nabla\psi_0^{-}\cdot\nu\right)
		\right)(X\cdot\nu)d\mathcal{H}^{n-1}\\
		&-\int_{\partial B_1}\nabla\left[\beta(\nabla\psi_0^+\cdot\nu)^2-(\nabla\psi_0^{-}\cdot\nu)^2\right]\cdot X(X\cdot\nu)d\mathcal{H}^{n-1},
	\end{aligned}	
	\end{equation*}		
		where $\psi_0$ is the minimizer of the energy of the ball $B_1$, meaning it solves the equation 
	\begin{equation*}
		\begin{cases}
			-\beta\Delta\psi_0=-\frac{1}{K}\psi_t+\frac{2}{K}\J(B_1)-\frac{1}{\vert B_1\vert}\text{ in }B_1,\\
			\Delta\psi_0=0\text{ in }B_1^c,\\
			\psi_0^+=\psi_0^{-}\text{ on }\partial B_1,\\
			\beta\nabla\psi^+_0\cdot\nu=\nabla\psi^{-}_0\cdot\nu\text{ on }\partial B_1.
		\end{cases}
	\end{equation*}
	We recall that $\dot{\J}(B_1)=0$ by Proposition \ref{p:first_derivative} and $\psi_0$ is radial by Proposition \ref{p:radiality of min fct}. This allows us to say that $\psi_0(x)=\psi(\vert x\vert)$ for some function $\psi:\mathbb{R}^+\rightarrow \mathbb{R}$ and we get
		\begin{equation*}
		\begin{aligned}
		\ddot{\J}(B_1)
		&=\frac{1-\frac{2}{K}\vert B_1\vert\J(B_1)}{\vert B_1\vert}\left(\int_{\partial B_1}{\dot{\psi}_0^+ (X\cdot\nu)}d\mathcal{H}^{n-1}+(\psi^+)'(1)\int_{\partial B_1}{(X\cdot\nu)^2}d\mathcal{H}^{n-1}\right)\\
		&-\beta(\psi^+)'(1)\int_{\partial B_1}{(\nabla\dot{\psi}_0^+
		-\nabla\dot{\psi}_0^{-})\cdot\nu(X\cdot\nu)}d\mathcal{H}^{n-1}\\
		&-\big(\beta(\psi^+)'(1)(\psi^+)''(1)-(\psi^-)'(1)(\psi^-)''(1)\big)\int_{\partial B_1}{(X\cdot\nu)^2}d\mathcal{H}^{n-1}\\
		&+\frac{1}{K}\psi(1)\int_{\partial B_1}{\dot{\psi}_0^+(X\cdot\nu)}d\mathcal{H}^{n-1}
		+\frac{1}{K}\psi(1)(\psi^+)'(1)\int_{\partial B_1}{(X\cdot\nu)^2}d\mathcal{H}^{n-1}\\
		&=\Big(\frac{1-\frac{2}{K}\vert B_1\vert\J(B_1)}{\vert B_1\vert}(\psi^+)'(1)-\big(\beta(\psi^+)'(1)(\psi^+)''(1)-(\psi^-)'(1)(\psi^-)''(1)\big)\\
		&\qquad\qquad\qquad\qquad\qquad\qquad\qquad\qquad\qquad+\frac{1}{K}\psi(1)(\psi^+)'(1)\Big)\int_{\partial B_1}{(X\cdot\nu)^2}d\mathcal{H}^{n-1}\\
		&+\left(\frac{1-\frac{2}{K}\vert B_1\vert\J(B_1)}{\vert B_1\vert}+\frac{1}{K}\psi(1)\right)\int_{\partial B_1}{\dot{\psi}_0^+(X\cdot\nu)}d\mathcal{H}^{n-1}\\
		&+(\beta-1)(\psi^+)'(1)\int_{\partial B_1}{\nabla\dot{\psi}_0^{-}\cdot\nu(X\cdot\nu)}d\mathcal{H}^{n-1},
	\end{aligned}	
	\end{equation*}	
	where $\dot{\psi}_0$ satisfies
	\begin{equation*}
	\begin{cases}
		-\beta\Delta\dot{\psi}_0=-\frac{1}{K}\dot{\psi}_0\text{ in }B_1,\\
		\Delta\dot{\psi}_0=0\text{ in }B_1^c,\\
		\dot{\psi}_0^+-\dot{\psi}_0^{-}=
		-\left((\psi^+)'(1)-(\psi^-)'(1)\right)(X\cdot\nu)\text{ on }\partial B_1,\\
		\beta\nabla\dot{\psi}^+_0\cdot\nu-\nabla\dot{\psi}^{-}_0\cdot\nu=
		-\left(\beta(\psi^+)''(1)-(\psi^-)''(1)\right) (X\cdot\nu)
		\text{ on }\partial B_1.
	\end{cases}
	\end{equation*}				
		Remembering that $\G(E)=\frac{K}{2\vert E\vert}-\J(E)$, we get
		\begin{equation*}
				\partial^2 \G(B_1)[\varphi,\varphi]:=
				\hat{c}_1\int_{\partial B_1}{\varphi^2}d\mathcal{H}^{n-1}
				+\int_{\partial B_1}{\hat{c}_2 H(\varphi)^++\hat{c}_3(\nabla H(\varphi)^{-}\cdot\nu)\varphi}d\mathcal{H}^{n-1}
			\end{equation*}
		where $$\hat{c}_1=-\frac{1-\frac{2}{K}\vert B_1\vert\J(B_1)}{\vert B_1\vert}(\psi^+)'(1)
		+\big(\beta(\psi^+)'(1)(\psi^+)''(1)-(\psi^-)'(1)(\psi^-)''(1)\big)
		-\frac{1}{K}\psi(1)(\psi^+)'(1),$$ 
		$$\hat{c}_2=-\frac{1-\frac{2}{K}\vert B_1\vert\J(B_1)}{\vert B_1\vert}
		-\frac{1}{K}\psi(1),$$ 
		$$\hat{c}_3=-(\beta-1)(\psi^+)'(1)$$  
		and $H(\varphi)$ is the unique solution of
			\begin{equation*}
		\begin{cases}
			\beta\Delta u=\frac{1}{K}u\text{ in }B_1,\\
			\Delta u=0\text{ in }B_1^c,\\
			u^+-u^{-}
			=c_1\varphi
			\text{ on }\partial B_1,\\
			\beta\nabla u^+\cdot\nu
			-\nabla u^{-}\cdot\nu
			=c_2\varphi\text{ on }\partial B_1,
		\end{cases}
	\end{equation*}
 		where $c_1=-\left((\psi^+)'(1)-(\psi^-)'(1)\right)$, 
 		$c_2=-\left(\beta(\psi^+)''(1)-(\psi^-)''(1)\right)$.

	\end{proof}
	
	We are now ready to show the following lemma.
	
	\begin{lemma}	
	 Given $\vartheta\in(0,1]$ there exists $\delta=\delta(n,\vartheta)$, $c=c(n,\beta,K)$ such that for any $\varphi\in C^{2,\vartheta}$ with $\Vert\varphi\Vert_{C^{2,\vartheta}}<\delta$ we have
	$$\partial^2 \G(B_1)[\varphi,\varphi]\geq -c\Vert \varphi\Vert^2_{H^\frac{1}{2}(\partial B_1)}.$$
	\end{lemma}
	\begin{proof}
	Consider $\varphi$ in the basis of spherical harmonics,
	\begin{equation*}
		\varphi=\sum_{m=0}^\infty\sum_{i=1}^{N(m,n)}\alpha_{m,i}Y_{m,i}.
	\end{equation*}
	First, we would like to bound $\partial^2 \G(B_1)[Y_{m,i},Y_{m,i}]$.
	One can easily see that $H(Y_{m,i})~=~R(r)Y_{m,i}$, where $R(r)$ is the only solution of the following system:
	\begin{equation*}
		\begin{cases}
			R_1''(r)+\frac{n-1}{r}R_1'(r)+(-\frac{1}{\beta K}+\frac{\lambda_{m,i}}{r^2})R_1(r)=0\text{ for }r\leq 1,\\
			R_2''(r)+\frac{n-1}{r}R_2'(r)+\frac{\lambda_{m,i}}{r^2}R_2(r)=0\text{ for }r\geq 1,\\
			R_1(1)-R_2(1)=c_1,\\
			\beta R_1'(1)-R_2'(1)=c_2,
		\end{cases}
	\end{equation*}
	where $\lambda_{m,i}=-m(m+n-2)$.
	
	A straightforward computation gives us that $R_2(r)=Ar^{-(m+n-2)}$
	for some constant $A$.
	
	Let us search for $R_1$ in the form $R_1(r)=\sum_{k=0}^\infty a_kr^k$.
	The equation for $R_1$ then will take the following form:
	\begin{equation*}
		\sum_{k=0}^\infty a_{k+2}(k+2)(k+1)r^k+\frac{n-1}{r}\sum_{k=0}^\infty a_{k+1}(k+1)r^k+\left(-\frac{1}{\beta K}+\frac{\lambda_{m,i}}{r^2}\right)\sum_{k=0}^\infty a_kr^k=0\text{ for }r\leq 1
	\end{equation*}
	If $m\geq 2$, it means that
	\begin{equation*}
		a_0=0,\quad a_1=0,\quad a_k\big(k(n+k-2)-m(n+m-2)\big)=\frac{1}{\beta K}a_{k-2}\text{ for }k\geq 2.
	\end{equation*}
	The recurrent condition can be rewritten as
	\begin{equation*}
		a_k(k-m)(k+n+m-2)=-a_{k-2}\text{ for }k\geq 2.
	\end{equation*}
	Hence,
	\begin{equation*}
	\begin{cases}
		a_m=C,\\
		a_{m+2i}=\beta K a_{m+2(i-1)}\frac{1}{2i(2i+2m+n-2)}\text{ for }i\geq 1,\\
		a_k=0\text{ for all other }k,
	\end{cases}
	\end{equation*}
	where $C$ is a constant.		
	So, the coefficients $a_k$ decrease as $\frac{(\beta K)^k}{(k!)^2}$ and the series is absolutely converging. Note that $a_k=Cb_k$, where $\{b_k\}_{k}$ is the following fixed sequence:
	\begin{equation*}
	\begin{cases}
		b_m=1,\\
		b_{m+2i}=(\beta K)^i\Pi_{j=1}^{i}\frac{1}{2j(2j+2m+n-2)}\text{ for }i\geq 1,\\
		b_k=0\text{ for all other }k.
	\end{cases}
	\end{equation*}
	Our system for $R$ then becomes
	\begin{equation*}
		\begin{cases}
			R_1(r)=C\sum_{i=0}^\infty b_{m+2i}r^{m+2i}\text{ for }r\leq 1,\\
			R_2(r)=Ar^{-(m+N-2)}\text{ for }r\geq 1,\\
			C\sum_{i=0}^\infty b_{m+2i}-A=c_1,\\
			\beta C\sum_{i=0}^\infty (m+2i)b_{m+2i}+A(m+n-2)=c_2
		\end{cases}
	\end{equation*}
	with $A$ and $C$ unknowns.
	We are interested in the value of $\vert R_2'(1)\vert$:
	\begin{equation*}
	\begin{split}
		\vert R_2'(1)\vert&=\vert A(N+m-2)\vert\\
		&\qquad\qquad=\left\vert\frac{c_1(m+N-2)+c_2}{\sum_{i=0}^\infty b_{m+2i}+\beta\sum_{i=0}^\infty (m+2i)b_{m+2i}}\beta \sum_{i=0}^\infty (m+2i)b_{m+2i}-c_2\right\vert\sim m.
	\end{split}	
	\end{equation*}	
	
	Thus, 
	\begin{equation} \label{e:asymptotics of J(Y)}
		\vert\partial^2 \G(B_1)[Y_{m,i},Y_{m,i}]\vert=\vert\hat{c}_1+ \hat{c}_2A+\hat{c}_3A(n+m-2)\vert\sim m.
	\end{equation}

	Now recall that $\varphi$ is such that $\vert\Omega\vert=\vert B_1\vert$ and $x_\Omega=0$. It means that
	\begin{equation*}
	\begin{aligned}
		&\vert B_1\vert=\vert\Omega\vert=\int_{\partial B_1}{\frac{\left(1+\varphi(x)\right)^n}{n}}d\mathcal{H}^{n-1};\\
		&0=x_\Omega=\int_{\partial B_1}{y\frac{\left(1+\varphi(x)\right)^{n+1}}{n+1}}d\mathcal{H}^{n-1}.
	\end{aligned}	
	\end{equation*}
		Hence,
	\begin{equation*}
	\begin{split}
		\left\vert\int_{\partial B_1}{\varphi(x)}d\mathcal{H}^{n-1}\right\vert&=\left\vert\int_{\partial B_1}{\sum_{i=2}^n{n\choose i}\frac{\varphi(x)^i}{n}}d\mathcal{H}^{n-1}\right\vert
		\\
		&\leq C(n)\int_{\partial B_1}{\varphi(x)^2}d\mathcal{H}^{n-1}\leq C(n)\delta\|\varphi\|_{L^2}
		\end{split}
	\end{equation*}	
	and
		\begin{equation*}
		\left\vert\int_{\partial B_1}{x_i\varphi(x)}d\mathcal{H}^{n-1}\right\vert\leq\int_{\partial B_1}{\sum_{j=2}^n{n\choose j}\left\vert\frac{\varphi(x)^j}{n+1}\right\vert} d\mathcal{H}^{n-1}\leq C(n)\delta\|\varphi\|_{L^2}.
	\end{equation*}
	
	Thus, for $\delta$ sufficiently small we have
	$$a_0^2+\sum_{i=1}^na_{1,i}^2\leq 2\sum_{m=2}^\infty\sum_{i=1}^{N(m,n)}a_{m,i}^2,$$
	which in turn implies
	$$\partial^2 \G(B_1)[\varphi,\varphi]\geq -c\Vert \varphi\Vert^2_{H^\frac{1}{2}(\partial B_1)},$$
	thanks to (\ref{e:asymptotics of J(Y)}).
	\end{proof}

\end{document}